\documentclass[11pt,twoside]{preprint}
\usepackage[a4paper,innermargin=1.4in,outermargin=1.4in,bottom=1.5in,marginparwidth=1.5in,marginparsep=3mm]{geometry}
\usepackage[full]{textcomp}
\usepackage[osf]{newtxtext}
\usepackage{amsmath,amsthm,amssymb,enumerate}
\usepackage{hyperref}
\usepackage{mathrsfs}
\usepackage{breakurl}
\usepackage{tikz}
\usepackage{tikz-3dplot}

\usepackage{xcolor}
\usepackage{mhequ}
\usepackage{comment}
\usepackage{longtable}

\usepackage{microtype}

\DeclareSymbolFont{timesoperators}{T1}{ptm}{m}{n}
\SetSymbolFont{timesoperators}{bold}{T1}{ptm}{b}{n}

\makeatletter
\renewcommand{\operator@font}{\mathgroup\symtimesoperators}
\makeatother

\tdplotsetmaincoords{60}{170}

\colorlet{darkblue}{blue!90!black}
\colorlet{darkred}{red!90!black}
\colorlet{darkgreen}{green!60!black}

\long\def\mateText#1{{\color{darkred}Mate:\ #1}}

\usetikzlibrary{snakes}
\usetikzlibrary{decorations}
\usetikzlibrary{positioning}
\usetikzlibrary{shapes}

\colorlet{symbols}{black!50}
\definecolor{connection}{rgb}{0.7,0.1,0.1}
\tikzset{
root/.style={circle,fill=black!50,inner sep=0pt, minimum size=3mm},
        dot/.style={circle,fill=black,inner sep=0pt, minimum size=0.3mm},
        xi-b/.style={very thin,draw=black,diamond,fill=blue!10,inner sep=0pt, minimum size=1.5mm},
        dotred/.style={circle,fill=black!50,inner sep=0pt, minimum size=2mm},
        var/.style={circle,fill=black!10,draw=black,inner sep=0pt, minimum size=3mm},
        kernel/.style={semithick,shorten >=2pt,shorten <=2pt},
        kernels/.style={snake=zigzag,shorten >=2pt,shorten <=2pt,segment amplitude=1pt,segment length=4pt,line before snake=2pt,line after snake=5pt,},
        rho/.style={densely dotted,draw=black!80!white,semithick,shorten >=2pt,shorten <=2pt},
        rhodelta/.style={densely dotted,draw=black,very thick,shorten >=2pt,shorten <=2pt},
           testfcn/.style={dotted,semithick,shorten >=2pt,shorten <=2pt},
        renorm/.style={shape=circle,fill=white,inner sep=1pt},
        labl/.style={shape=rectangle,fill=white,draw=black,inner sep=2pt},
        xic/.style={very thin,circle,fill=symbols,draw=black,inner sep=0pt,minimum size=1.2mm},
        Z/.style={very thin,rectangle,fill=blue!50,draw=black,inner sep=0pt,minimum size=1.2mm},
        xi/.style={very thin,circle,fill=blue!10,draw=black,inner sep=0pt,minimum size=1.1mm},
        xismall/.style={very thin,circle,fill=blue!10,draw=black,inner sep=0pt,minimum size=0.7mm},
        xi3/.style={very thin,circle,fill=green!80!black,draw=black,inner sep=0pt,minimum size=1.1mm},
        arg/.style={very thin,rectangle,fill=blue!50,draw=black,inner sep=0pt,minimum size=1.2mm},
        xix/.style={crosscircle,fill=blue!10,draw=black,inner sep=0pt,minimum size=1.2mm},
	xib/.style={very thin,circle,fill=blue!10,draw=black,inner sep=0pt,minimum size=1.6mm},
	xie/.style={very thin,circle,fill=green!50!black,draw=black,inner sep=0pt,minimum size=1.6mm},
	xid/.style={very thin,circle,fill=symbols,draw=black,inner sep=0pt,minimum size=1.6mm},
	xibx/.style={crosscircle,fill=blue!10,draw=black,inner sep=0pt,minimum size=1.6mm},
	kernels2/.style={very thick,draw=connection,segment length=12pt},
	kernels3/.style={very thick,draw=green!70!black,segment length=12pt},
	kerneldiff/.style={very thick,draw=green!70!black,decorate, decoration={snake, segment length=3.5pt, amplitude=1.5pt}},
	kernels4/.style={thick,draw=blue!70!black,segment length=12pt},
	not3/.style={thin,circle,draw=green!70!black,fill=green!80!black,inner sep=0pt,minimum size=0.5mm},
	vari/.style={very thin,circle,fill=white,draw=black,inner sep=0pt,minimum size=1.1mm},
	noise/.style={very thin,rectangle,fill=white,draw=black,inner sep=0pt,minimum size=1.1mm},
	base/.style={very thin,circle,fill=black,draw=black,inner sep=0pt,minimum size=1.1mm},	
	int/.style={very thin,rectangle,fill=black,draw=black,inner sep=0pt,minimum size=0.7mm},
			not/.style={thin,circle,fill=symbols,draw=connection,fill=connection,inner sep=0pt,minimum size=0.5mm},
not+/.style={thin,circle, fill=symbols,draw=connection,fill=connection,inner sep=0pt,minimum size=0.8mm},
			bnot/.style={thin,circle,draw=black,fill=black,inner sep=0pt,minimum size=0.2mm},
			bnot+/.style={thin,circle,draw=black,fill=black,inner sep=0pt,minimum size=0.8mm},
	>=stealth,
        }

\makeatletter
\def\DeclareSymbol#1#2#3{\expandafter\gdef\csname MH@symb@#1\endcsname{\tikz[baseline=#2,scale=0.15,draw=symbols,line join=round]{#3}}\expandafter\gdef\csname MH@symb@#1s\endcsname{\scalebox{0.7}{\tikz[baseline=#2,scale=0.15,draw=symbols,line join=round]{#3}}}}
\def\<#1>{\csname MH@symb@#1\endcsname}
\makeatother

\makeatletter
\def\DeclareSymbolSmall#1#2#3{\expandafter\gdef\csname MH@symb@#1\endcsname{\tikz[baseline=#2,scale=0.11,draw=symbols,line join=round]{#3}}\expandafter\gdef\csname MH@symb@#1s\endcsname{\scalebox{0.7}{\tikz[baseline=#2,scale=0.15,draw=symbols,line join=round]{#3}}}}
\def\<#1>{\csname MH@symb@#1\endcsname}
\makeatother

\DeclareSymbol{0}{-2.8}{\node[xib] {};}
\DeclareSymbol{PAM-1}{0}{\draw[kernels2] (0,0) node[not] {} -- (0,1.5) node[xi] {};}
\DeclareSymbol{PAM-2a}{-0.5}{\draw[kernels2] (-1,1) node[xi] {} -- (0,0) node[not] {} -- (1,1) node[xi] {};}
\DeclareSymbol{PAM-I2a}{-4}{\draw[kernels2] (-1,1) node[xi] {} -- (0,0) node[not] {} -- (1,1) node[xi] {};\draw (0,0) node[not] {} -- (0,-1.4) node[bnot] {};}
\DeclareSymbol{PAM-2b}{0.5}{\draw[kernels2] (-1,1)  -- (0,2) node[xi] {}; \draw[kernels2] (-1,1) node[not] {} -- (0,0) node[not] {} -- (1,1) node[xi] {};}
\DeclareSymbol{PAM-2new}{-0.5}{\draw[kernels3] (-1,1) node[xi] {} -- (0,0) node[not3] {} -- (1,1) node[xi] {};}
\DeclareSymbol{PAM-3}{0.5}{\draw[kernels2] (-2,2) node[xi] {} -- (-1,1)  -- (0,2) node[xi] {}; \draw[kernels2] (-1,1) node[not] {} -- (0,0) node[not] {} -- (1,1) node[xi] {};}
\DeclareSymbol{PAM-4a}{-4}{\draw[kernels2] (-1,1) node[xi] {} -- (0,0) node {} -- (1,1) node[xi] {}; \draw (0,0) node[not] {} -- (0,-1.4) {}; \draw[kernels2] (1,-0.4) node[xi] {} -- (0,-1.4) node[not] {} -- (-1,-0.4) node[xi] {};} 
\DeclareSymbol{PAM-2I}{-4}{\draw[kernels2] (-1,1) node[xi] {} -- (0,0) node {} -- (1,1) node[xi] {}; \draw (0,0) node[not] {} -- (0,-1.4) {};} 
\DeclareSymbol{PAM-4b}{-6}{\draw[kernels2] (-1,1) node[xi] {} -- (0,0) node {} -- (1,1) node[xi] {}; \draw[kernels2] (0,0) node[not] {} -- (1,-1) {} -- (2,0) node[xi] {}; \draw[kernels2] (1,-1) node[not] {} -- (2,-2) node[not] {} -- (3,-1) node[xi] {};}
\DeclareSymbol{PAM-4new}{-4}{\draw[kernels2] (-1,1) node[xi] {} -- (0,0) node {} -- (1,1) node[xi] {}; \draw (0,0) node[not] {} -- (0,-1.4) {}; \draw[kernels3] (1,-0.4) node[xi] {} -- (0,-1.4) node[not3] {} -- (-1,-0.4) node[xi] {};} 

\DeclareSymbolSmall{other-2-Small}{-2}{\draw (-1,1) node[xismall] {} -- (0,0) node[xismall] {};}
\DeclareSymbolSmall{other-4-Small}{-6}{\draw (-1,1) node[xismall] {} -- (0,0) node[xismall] {} -- (-1,-1) node[xismall] {} -- (0,-2) node[xismall] {};}

\DeclareSymbolSmall{PAM-2a-Small}{-2}{\draw[kernels2] (-1,1) node[xismall] {} -- (0,0) node[not] {} -- (1,1) node[xismall] {};}
\DeclareSymbolSmall{PAM-4a-Small}{-4}{\draw[kernels2] (-1,1) node[xismall] {} -- (0,0) node {} -- (1,1) node[xismall] {}; \draw (0,0) node[not] {} -- (0,-1.4) {}; \draw[kernels2] (1,-0.4) node[xismall] {} -- (0,-1.4) node[not] {} -- (-1,-0.4) node[xismall] {};} 
\DeclareSymbolSmall{PAM-4b-Small}{-6}{\draw[kernels2] (-1,1) node[xismall] {} -- (0,0) node {} -- (1,1) node[xismall] {}; \draw[kernels2] (0,0) node[not] {} -- (1,-1) {} -- (2,0) node[xismall] {}; \draw[kernels2] (1,-1) node[not] {} -- (2,-2) node[not] {} -- (3,-1) node[xismall] {};}

\DeclareSymbol{Phi-2}{-2}{\draw (-1,1) node[xi] {} -- (0,0) node[bnot] {} -- (1,1) node[xi] {};}
\DeclareSymbolSmall{Phi-2-Small}{-2}{\draw (-1,1) node[xismall] {} -- (0,0) node[bnot] {} -- (1,1) node[xismall] {};}
\DeclareSymbol{Phi-3}{-2}{\draw (-1,1) node[xi] {} -- (0,0) node[bnot] {} -- (1,1) node[xi] {}; \draw (0,0) -- (0,1.5) node[xi] {}; }
\DeclareSymbol{Phi-3new}{-0.5}{\draw[kernels3] (-1,1) node[xi] {} -- (0,0) node[not3] {} -- (1,1) node[xi] {}; \draw[kernels3] (0,0) -- (0,1.5) node[xi] {}; }
\DeclareSymbol{Phi-4new}{-4}{\draw (-1,1) node[xi] {} -- (0,0) node {} -- (1,1) node[xi] {}; \draw (0,0) node[bnot] {} -- (0,-1.4) {}; \draw[kernels3] (1,-0.4) node[xi] {} -- (0,-1.4) node[not3] {} -- (-1,-0.4) node[xi] {};} 
\DeclareSymbol{Phi-5new}{-4}{\draw (-1,1) node[xi] {} -- (0,0) node {} -- (1,1) node[xi] {}; \draw (0,1.4) node[xi] {} -- (0,0) node[bnot] {} -- (0,-1.4) {}; \draw[kernels3] (1,-0.4) node[xi] {} -- (0,-1.4) node[not3] {} -- (-1,-0.4) node[xi] {};}
\DeclareSymbolSmall{Phi-4-Small}{-4}{\draw (-1,1) node[xismall] {} -- (0,0) node {} -- (1,1) node[xismall] {}; \draw (0,0) node[bnot] {} -- (0,-1.4) {}; \draw (1,-0.4) node[xismall] {} -- (0,-1.4) node[bnot] {} -- (-1,-0.4) node[xismall] {};} 
\DeclareSymbol{Phi-I2}{-4}{\draw (-1,1) node[xi] {} -- (0,0) node[bnot] {} -- (1,1) node[xi] {};\draw (0,0) node[bnot] {} -- (0,-1.4) node[bnot] {};}
\DeclareSymbol{Phi-I3}{-4}{\draw (-1,1) node[xi] {} -- (0,0) node[bnot] {} -- (1,1) node[xi] {};\draw (0,1.4) node[xi] {} -- (0,0) node[bnot] {} -- (0,-1.4) node[bnot] {};}

\DeclareSymbolSmall{Phi-2-Small}{-2}{\draw (-1,1) node[xismall] {} -- (0,0) node[bnot] {} -- (1,1) node[xismall] {};}
\DeclareSymbolSmall{Phi-4-Small}{-4}{\draw (-1,1) node[xismall] {} -- (0,0) node {} -- (1,1) node[xismall] {}; \draw (0,0) node[bnot] {} -- (0,-1.4) {}; \draw (1,-0.4) node[xismall] {} -- (0,-1.4) node[bnot] {} -- (-1,-0.4) node[xismall] {};}

\DeclareSymbol{lolly}{0}{\draw (0,0) node[dot] {} -- (0,1.5) node[xi] {};}	
\DeclareSymbol{cherry}{-2}{\draw (-1,1) node[xi] {} -- (0,0) node[not] {} -- (1,1) node[xi] {};}
\DeclareSymbol{triple}{-2}{\draw (-1,1) node[xi] {} -- (0,0) node[not] {} -- (1,1) node[xi] {}; \draw (0,0) node[not] {}-- (0,1.5) node[xi] {};}
\DeclareSymbol{fork}{-3.5}{\draw (-1,1) node[xi] {} -- (0,0) node[not] {} -- (1,1) node[xi] {}; \draw (0,-1.5) node[not] {} -- (0,0) node[not] {}-- (0,1.5) node[xi] {};}
\DeclareSymbol{fork2}{-3.5}{\draw (-1,1) node[xi] {} -- (0,0) node[not] {} -- (1,1) node[xi] {}; \draw (0,-1.5) node[not] {} -- (0,0) node[not] {};}
\DeclareSymbol{fork1}{-3.5}{\draw  (0,0) node[not] {} -- (1,1) node[xi] {}; \draw (0,-1.5) node[not] {} -- (0,0) node[not] {};}

\DeclareSymbol{lolly-b}{0}{\draw[kernels2] (0,0) node[dot] {} -- (0,1.5) node[xi-b] {};}
\DeclareSymbol{cherry-b}{-2}{\draw[kernels2] (-1,1) node[xi-b] {} -- (0,0) node[not] {} -- (1,1) node[xi-b] {};}
\DeclareSymbol{triple-b}{-2}{\draw[kernels2] (-1,1) node[xi-b] {} -- (0,0) node[not] {} -- (1,1) node[xi-b] {}; \draw[kernels2] (0,0) node[not] {}-- (0,1.5) node[xi-b] {};}

\DeclareSymbol{cherry-green}{-2}{\draw[kernels3] (-1,1) node[xi] {} -- (0,0) node[not3] {} -- (1,1) node[xi] {};}

\makeatletter
\newenvironment{claim}[1][MM]
               {\list{$\bullet$}{%
  \setbox\@tempboxa\hbox{#1}\@tempdima\wd\@tempboxa%
  \setlength{\labelwidth}{\@tempdima}
  \advance\@tempdima by 1em%
  \setlength{\leftmargin}{\@tempdima}
  \setlength{\parsep}{1mm}\setlength{\itemindent}{0mm}%
  \setlength{\labelsep}{2mm}\setlength{\itemsep}{0mm}%
  \setlength{\topsep}{1mm}%
}}{\endlist}
\makeatother

\newtheorem{theorem}{Theorem}[section]
\newtheorem{lemma}[theorem]{Lemma}
\newtheorem{proposition}[theorem]{Proposition}
\newtheorem{corollary}[theorem]{Corollary}

\theoremstyle{definition}
\newtheorem{assumption}[theorem]{Assumption}
\newtheorem{definition}[theorem]{Definition}
\newtheorem{example}[theorem]{Example}

\theoremstyle{remark}
\newtheorem{remark}[theorem]{Remark}

\def\scal#1{\langle #1 \rangle}

\newcommand{\vn}[1]{{\vert\kern-0.3ex\vert\kern-0.3ex\vert #1 
    \vert\kern-0.3ex\vert\kern-0.3ex\vert}}

\newcommand{\bn}[1]{{[\kern-0.5ex] #1 
    [\kern-0.5ex]}}

\allowdisplaybreaks

\def\dash{\leavevmode\unskip\kern0.18em--\penalty\exhyphenpenalty\kern0.18em}
\def\slash{\leavevmode\unskip\kern0.15em/\penalty\exhyphenpenalty\kern0.15em}

\newcommand{\N}{\mathbb{N}}

\newcommand\bone{\mathbf{1}}

\newcommand\cR{\mathcal{R}}

\newcommand\cK{\mathcal{K}}
\newcommand\cD{\mathcal{D}}

\newcommand\cB{\mathcal{B}}
\newcommand\cC{\mathcal{C}}
\newcommand\CC{{\mathcal{C}}}
\newcommand\cN{\mathcal{N}}

\newcommand\cI{\mathcal{I}}
\newcommand\cJ{\mathcal{J}}

\newcommand\cS{\mathcal{S}}
\newcommand\cG{\mathcal{G}}
\newcommand\cV{\mathcal{V}}

\newcommand\cW{\mathcal{W}}

\newcommand\cO{\mathcal{O}}
\newcommand\cF{\mathcal{F}}

\newcommand\scR{\mathscr{R}}

\newcommand\scZ{\mathscr{Z}}
\newcommand\scD{\mathscr{D}}

\newcommand\scT{\mathscr{T}}
\newcommand\scP{\mathscr{P}}
\newcommand{\scJ}{\mathscr{J}}
\newcommand{\scG}{\mathscr{G}}

\newcommand\frK{\mathfrak{K}}
\newcommand\frD{\mathfrak{D}}
\newcommand\frs{\mathfrak{s}}
\newcommand\frm{\mathfrak{m}}
\newcommand\frb{\mathfrak{b}}

\newcommand{\bv}{\mathbf{v}}
\newcommand{\bi}{\mathbf{i}}

\newcommand{\bw}{\mathbf{w}}

\newcommand{\bY}{\mathbf{Y}}

\newcommand{\us}{\underline{s}}

\newcommand{\uone}{\underline{1}}

\def\one{\mathbf{1}}

\def\les{\lesssim}

\def\half{\textstyle{1\over 2}}
\def\quarter{\textstyle{1\over 4}}
\def\eps{\varepsilon}
\def\Dir{{\textnormal{\tiny Dir}}}

\newcommand{\R}{\mathbb{R}}
\newcommand{\E}{\mathbf{E}}

\DeclareMathOperator{\supp}{supp}

\DeclareMathOperator{\intr}{int}

\DeclareMathOperator{\Ercf}{Ercf}
\DeclareMathOperator{\Tr}{Tr}

\def\d{\partial}
\def\id{\mathrm{id}}

\begin{document}
\title{Boundary renormalisation of SPDEs}

\author{Máté Gerencsér$^1$ and Martin Hairer$^2$}
\institute{%
Technische Universität Wien, Austria, \email{mate.gerencser@tuwien.ac.at}%
\and%
Imperial College London, UK, \email{m.hairer@imperial.ac.uk}%
}

\maketitle

\begin{abstract}
We consider the continuum parabolic Anderson model (PAM)
and the dynamical $\Phi^4$ equation
on the $3$-dimensional cube with boundary conditions.
While the Dirichlet solution theories are relatively standard, the case of Neumann \slash Robin boundary conditions gives rise to a divergent boundary renormalisation.
Furthermore for $\Phi^4_3$ a `boundary triviality' result is obtained: if one approximates the equation with Neumann boundary conditions and the usual bulk renormalisation, then the 
limiting process coincides with the one obtained using Dirichlet boundary conditions. \\[.4em]
\end{abstract}

\tableofcontents

\section{Introduction}

Most of the existing literature on singular SPDEs (and associated operators) considers equations on flat domains 
without boundaries, like $\mathbb{T}^d$ or $\R^d$.
There are also some recent results where boundary conditions are considered and this raises analytical complications but where the final statement is completely analogous to the periodic case. Examples are \cite{Chouk2021}, \cite{Cyril19}, \cite[Thms.~1.1,~1.5]{GH17}, \cite[Thm.~1.1]{HP19}.

In some situations, however, the effect of the boundaries is more drastic.
A notable example is the open KPZ equation
for which both in its derivation from exclusion processes \cite{Corwin2018, Parekh2018, GPS} and its solution theory via regularity structures \cite{GH17}
an approximation-dependent finite boundary correction term arises. A similar phenomenon was observed in the context of a stochastic homogenisation problem in \cite{HP19}.

The goal of the present article is to extend the list of nontrivial boundary effects by two well-known SPDEs endowed with appropriate boundary conditions. In contrast to the previous examples, the boundary renormalisation in these examples does not remain bounded but diverges logarithmically. A peculiar consequence is that if this divergence comes with a positive sign (which turns out to be the case for the dynamical $\Phi^4_3$ model) then removing the boundary renormalisation does not destroy the convergence of the Wong--Zakai approximations but rather ``trivialises'' the boundary condition to a vanishing Dirichlet one. More precisely, the sequence of smooth approximations with any \textit{fixed} boundary condition of Robin type (including Neumann) converges to the same limit as that with Dirichlet boundary conditions, but it is possible to choose an $\eps$-dependent  Robin boundary condition in such a way that the limit still 
exists but is different.
This is a mechanism somewhat analogous to that underlying the triviality result of \cite{triv}, although 
here the limit is still nontrivial in the bulk.

\subsection{The parabolic Anderson model}

Consider the $3$-dimensional continuum
parabolic Anderson model formally given by
\begin{equ}[e:PAM]\tag{PAM}
(\partial_t-\Delta)u=u\xi
\end{equ}
on $D=(-1,1)^3$, with $\xi$ denoting spatial white noise (constant in time),
and with some initial condition $u_0\in\cC^{\delta}(D)$, $\delta>-1/2$.
Fix a symmetric nonnegative $\rho\in\cC^\infty_c(\R^3)$ integrating to $1$,
set $\rho_\eps(x)=\eps^{-3}\rho(\eps^{-1}x)$, and let $\xi_\eps = \rho_\eps *\xi$.

It is then known \cite{H0} that if we endow $D$ with \textit{periodic} boundary conditions
and consider the sequence of problems
\begin{equ}[e:ePAM]\tag{$\eps$-PAM}
(\partial_t-\Delta )v=v(\xi_\eps-C_\eps)\;,
\end{equ}
then a suitable (diverging) choice of $C_\eps$ yields a limit $u$ independent 
of $\rho$, which we then call ``the'' solution to the otherwise ill-posed problem \eqref{e:PAM}.

First we show the existence of the Dirichlet solution of \eqref{e:PAM}, where our result is 
rather unsurprising.
\begin{theorem}\label{thm:PAM-D}
Let for $\eps\in(0,1]$ $u^{\Dir}_\eps$ be the solution to
\eqref{e:ePAM} with boundary conditions
\begin{equ}
u_\eps^{\Dir}=0\quad\text{on }\R^+\times\partial D.
\end{equ}
Then for any $\kappa>0$, $u_\eps^{\Dir}$ converges in probability in
$\cC\big([0,1],\cC^{\delta\wedge(1/2)-\kappa}(D)\big)\cap\cC^{1/2-\kappa}_{\text{loc}}\big((0,1]\times \bar D\big)$ to a limit $u^{\Dir}$
independent of $\rho$ as $\eps \to 0$.
\end{theorem}

However, in order make sense of a notion of ``Neumann solution'' for \eqref{e:PAM}, 
one needs to renormalise the boundary condition as well! A similar phenomenon 
was previously observed in the case of the KPZ equation in \cite{GH17} and in the case of
a stochastic homogenisation problem in \cite{HP19}, but the boundary renormalisation 
remained bounded in these examples.
\begin{theorem}\label{thm:PAM-N}
Given a constant $a_\rho$, let  $ u_\eps$ for $\eps\in(0,1]$ be the
solution to \eqref{e:ePAM} with boundary conditions
\begin{equ}
\partial_n u_\eps=-\Big(a_\rho+\tfrac{|{\log \eps}|}{8\pi}\Big) u_\eps\quad
\text{on }\R^+\times\partial D.
\end{equ}
Then $u_\eps$ converges (in the same sense as in Theorem \ref{thm:PAM-D}) 
to a limit $u$. Moreover, one can choose $a_\rho$ to depend on $\rho$
in such a way that the limit does not depend on $\rho$.
\end{theorem}

\begin{remark}\label{rem:reno-PAM}
The constant $C_\eps$ is the same as in the translation invariant case, although it will be obtained somewhat differently. It will be decomposed as $C_\eps=\ell_\eps(\<PAM-2a-Small>)+\ell_\eps(\<PAM-4a-Small>)+4\ell_\eps(\<PAM-4b-Small>)$, with the `usual' tree notation of regularity structures
and $\ell_\eps$ denoting the BPHZ renormalisation character, see Section~\ref{sec:reg str}
and the references \cite{CH,BHZ}.
Here $\ell_\eps(\<PAM-2a-Small>)$ is of order $\eps^{-1}$ and $\ell_\eps(\<PAM-4a-Small>)$ and $\ell_\eps(\<PAM-4b-Small>)$ are of order $|{\log \eps}|$.
This might be surprising to readers familiar with the renormalisation of PAM: from e.g. \cite{HP15} one would expect instead a decomposition of the form $C_\eps=\ell_\eps(\<other-2-Small>)+\ell_\eps(\<other-4-Small>)$.
It follows from more general identities between renormalisation constants \cite{BGHZ, Mate19} that these two expressions are equivalent, although this will not play an explicit role in the present paper.
\end{remark}

\begin{remark}\label{rem:dimension}
The dimension $d=3$ is crucial for the boundary renormalisation to appear.
Indeed, it is easy to see that in $2$ dimensions the trick in \cite{HL_2D} (and described below) can be used to solve both the Dirichlet and the Neumann problem in a quite straightforward way.
For $\xi$ with H\"older-regularity between $-4/3$ and $-3/2$ (corresponding to ``dimension'' strictly between $8/3$ and $3$), 
the setup of Section \ref{sec:proofs} allows one to apply the results of \cite{GH17} directly, with the translation invariant models from \cite{CH},
resulting again in `standard' statements for both the Dirichlet and the Neumann problems.

One can also justify the $d=3$ threshold by simple power counting: nontrivial boundary behaviour may only be expected if a product in the regularity structure creates a nonintegrable singularity near the boundary, that is, of order below $-1$.
The lowest degree product for \eqref{e:PAM} with $\alpha$-regular noise $\xi$ is of order $2\alpha+2$, which indeed shows that $\alpha=-3/2$ is critical.
\end{remark}

\begin{remark}\label{rem:paracontrolled}
In \cite{Ismael} paracontrolled calculus was
employed to give meaning to the Neumann PAM.
In their result however, the renormalisation in \eqref{e:ePAM}
took the form of $\lambda_\eps$ (in place of $C_\eps$) for some deterministic
 \emph{function} $\lambda_\eps$.
Theorem \ref{thm:PAM-N} can be seen as giving a precise form of this
function, namely showing that one can choose
$\lambda_\eps=C_\eps+\hat C_\eps\delta_\partial$.
Here $C_\eps$ is the same constant as in the translation-invariant case,
$\hat C_\eps$ is the logarithmic term from the theorem,
and $\delta_\partial$ is the `Dirac distribution on $\partial D$'
(see \eqref{eq:dirac} for the precise definition).
Hence our solution theory is parametrised by two real constants, instead of a space of (locally) smooth functions.
\end{remark}

\subsection{The dynamical \texorpdfstring{$\Phi^4_3$}{Phi\^4\_3} equation}
Take the dynamical $\Phi^4_3$ equation
\begin{equ}[e:Phi43]\tag{$\Phi^4_3$}
(\partial_t-\Delta)u=-u^3+\xi
\end{equ}
on $D=(-1,1)^3$, this time $\xi$ standing for the $1+3$-dimensional space-time white noise,
and with initial condition $u_0\in\CC^\delta(D)$ with $\delta>-2/3$.
Consider the regularised equation
\begin{equ}[e:Phi43-eps]\tag{$\eps$-$\Phi^4_3$}
(\partial_t-\Delta)u=-u^3+3C_\eps u+\xi_\eps.
\end{equ}

We again start by solving the Dirichlet problem.
\begin{theorem}\label{thm:Phi4-D}
Let for $\eps\in(0,1]$ $u_\eps^{\Dir}$ be the solution of \eqref{e:Phi43-eps} with boundary conditions
\begin{equ}
u_\eps^{\Dir}=0\quad\text{on }\R^+\times\partial D.
\end{equ}
Then for any $\kappa>0$, and some random time $T>0$, $u_\eps^{\Dir}$ converges in probability in
$\cC([0,T],\cC^{\delta\wedge(-1/2)-\kappa}(D))$ to a limit $u^{\Dir}$
independent of $\rho$ as $\eps \to 0$.
\end{theorem}

\begin{remark}
Although $u^{\Dir}$ is a distribution, it does satisfy homogeneous Dirichlet boundary condition in a relatively strong sense, see Remark \ref{rem:Dirichlet} below.
\end{remark}

For the Neumann problem we once again see boundary renormalisation, this time with a twist: if one does not change the boundary conditions (i.e. $b_\eps\equiv 0$ below), then the solutions still converge, but to the Dirichlet solution $u^{\Dir}$!
\begin{theorem}\label{thm:Phi4-N}
Fix $b\in(-\infty,\infty]$ and let $(b_\eps)_{\eps\in(0,1]}\subset\R$ be such that $\eps b_\eps\to 0$ and
\begin{equ}\label{eq:boundary-reno}
\lim_{\eps\to 0}\frac{|{\log\eps}|}{32\pi}-b_\eps=b.
\end{equ}
Given a constant $a_\rho$, let for $\eps\in(0,1]$ $u_\eps$ be the solution of \eqref{e:Phi43-eps} with boundary conditions
\begin{equ}
\d_n u_\eps=3\big(a_\rho+b_\eps\big) u_\eps\quad
\text{on }\R^+\times\partial D.
\end{equ}
Then $u_\eps$ converges (in the same sense as in Theorem \ref{thm:Phi4-D}) to a limit $u$. Moreover, one can chose $a_\rho$ to depend on $\rho$ in such a way that the limit $u=u^b$ only depends on $b$, and this dependence is continuous. Finally, one has $u^\infty=u^{\Dir}$.
\end{theorem}
\begin{remark}\label{rem:reno-Phi}
The renormalisation in the bulk is given by the same $C_\eps$ as in the translation invariant case. It is of the form $C_\eps=\ell_\eps(\<Phi-2-Small>)-3\ell_\eps(\<Phi-4-Small>)$, where $\ell_\eps(\<Phi-2-Small>)$ is of order $\eps^{-1}$
and $\ell_\eps(\<Phi-4-Small>)$ is of order $|{\log \eps}|$.
\end{remark}

\subsection{The linearisation step}\label{sec:lin}
As usual for singular SPDEs, we start by comparing the solutions to the linear equation with additive noise.
Since this step is nontrivially affected by the boundary conditions, we give a short outline below.

First consider the case of the PAM.
Introduce, for $\eps\in[0,1]$,
$Y_\eps$ as the solution of
\begin{equ}\label{eq:Y}
\Delta Y_\eps=\xi_\eps\quad
\text{on }D,\quad\quad \partial_n Y_\eps=0\quad \text{on }\partial D,
\end{equ}
in other words,
\begin{equ}\label{eq:Y2}
Y_\eps(z)=\int_D G(z,z')\xi_\eps(z')\,dz',
\end{equ}
where $G$ is the Neumann Green's function on $D$.
Equation \eqref{eq:Y} of course only makes classical sense for $\eps> 0$,
but $Y_0$ can be simply interpreted by \eqref{eq:Y2} in the distributional sense.

 Notice then that if $u$ solves \eqref{e:ePAM} for some $\eps>0$,
then $v=ue^{Y_\eps}$ satisfies 
\begin{equ}[e:etPAM]\tag{tPAM}
(\partial_t -\Delta )v=v(|\nabla Y_\eps|^2-C_\eps)- 2\nabla v\cdot\nabla Y_\eps
\end{equ}
on $(0,1)\times D$ and $v(0,x)=u_0(x)e^{Y_\eps}(x)$.
This is essentially the same transformation that is used in
the two-dimensional case in \cite{HL_2D}, where
it allows one to bypass the abstract theory completely.
This is not the case here, but nevertheless it will make the boundary
renormalisation easier to handle.

\begin{remark}\label{rem:bcPAM}
It is worth pointing out that the choice of boundary condition for $Y_\eps$
is \emph{not} related to the choice of boundary condition of \eqref{e:PAM},
i.e.\ the former will always be Neumann with vanishing data.
\end{remark}

In contrast to Remark \ref{rem:bcPAM}, for the dynamical $\Phi^4_3$ equation the reference linear solution \emph{does} depend of the choice of boundary conditions.
Introduce, for $\eps\in[0,1]$ and $a\in[0,\infty)$, $\Psi_{\eps,a}$ as the stationary solution of
\begin{equ}\label{eq:Psi}
(\d_t-\Delta) \Psi_{\eps,a}=\xi_\eps\quad
\text{on }D,\quad\quad \partial_n \Psi_{\eps,a}=-3a\Psi_{\eps,a}\quad \text{on }\partial D,
\end{equ}
in other words,
\begin{equ}\label{eq:Psi2}
\Psi_{\eps,a}(z)=\int_{\R\times D}\cG_{3a}(z,z')\xi_\eps(z')\,dz',
\end{equ}
where $\cG_{3a}$ are the Robin heat kernels (see Section~\ref{sec:robin} for a detailed discussion).
As before, equation \eqref{eq:Psi} 	only makes classical sense for $\eps> 0$,
but $\Psi_{0,a}$ can be simply interpreted by \eqref{eq:Psi2} in the distributional sense.
Recall that boundary conditions involving the normal derivative of the solution can be incorporated on the right-hand side of the equation as a Dirac mass on the boundary. For example, \eqref{eq:Psi2} can be equivalently written as
\begin{equ}
\Psi_{\eps,a}(z)=\int_{\R\times D}\cG_{0}(z,z')\big(\xi_\eps-3a\Psi_{\eps,a}\delta_\d\big)(z')\,dz',
\end{equ}
see e.g. \cite[Rem~1.5]{HP19}.
It is actually not obvious that this expression also makes sense at $\eps=0$. 
To multiply $\delta_\d$ and $\Psi_{0,a}$, one would have to make sense of the restriction of the distribution $\Psi_{0,a}$ to a lower dimensional subset, which is of course for a general distribution not possible.
In the case of $\Psi_{0,a}$ it turns out it is, see Section \ref{sec:restrict}.

Using the above mentioned equivalence, one sees that if $u$ solves \eqref{e:Phi43-eps} for $\eps>0$ with boundary condition $\d_n u=3(a_\rho+ b_\eps) u$, then $v=u-\Psi_{\eps,c_\eps}$ satisfies the equation
\begin{equs}[eq:Phi4-lin]
(\d_t-\Delta)v&=-v^3-3v^2\Psi_{\eps,c_\eps}-3v(\Psi_{\eps,c_\eps}^2-C_\eps-(a_\rho+b_\eps+c_\eps)\delta_\d)
\\
&\qquad-(\Psi_{\eps,c_\eps}^3-3C_\eps\Psi_{\eps,c_\eps}-3(a_\rho+b_\eps+c_\eps)\delta_\partial\Psi_{\eps,c_\eps}),
\end{equs}
with boundary condition $\d_nv=-3c_\eps v$.
This suggests that it might be a good idea to look for $c_\eps$ for which the distributions $\Psi_{\eps,c_\eps}^2-C_\eps-(a_\rho+b_\eps+c_\eps)\delta_\d$ converge. This is \emph{almost} the case: in Lemma \ref{lem:reno Phi43 square} we construct a sequence $c_\eps$ such that  $\Psi_{\eps,c_\eps}^2-\ell_\eps(\<Phi-2-Small>)-(a_\rho + b_\eps+c_\eps)\delta_\d$ converges to a $\rho$-independent limit.

We remark that the $a=\infty$ endpoint of \eqref{eq:Psi} and \eqref{eq:Psi2} also makes perfect sense and corresponds to the case of homogeneous Dirichlet boundary condition (see Section \ref{sec:robin}). This fact will be essential in the proof of Theorem \ref{thm:Phi4-N}.

The rest of the article is structured as follows. In Section \ref{sec:prelim} we set up a number of technical tools concerning function spaces and singular kernels with singularities at boundaries.
In Section \ref{sec:calculations} we compute explicitly the boundary renormalisation of a few stochastic objects corresponding to the simplest `trees' associated to each equation.
In Section \ref{sec:reg str} we define the regularity structures and the models associated to the equations, which contain a few further terms that, while do not require boundary renormalisation, are affected by the boundary conditions and therefore fall outside of the scope of the generic construction of models \cite{CH}.
In Section \ref{sec:proofs} we formulate each equation in the analytic framework of \cite{GH17}, which then yields the proofs of the main results.

\subsection*{Acknowledgements}

{\small
MG thanks the support of the Austrian Science Fund (FWF) through the Lise Meitner programme M2250-N3 during a significant part of the project. MH gratefully acknowledges support from the Royal Society through a research professorship. Thanks also to 
Etienne Pardoux for numerous discussions on the topic of this article.
}

\section{Preliminaries}\label{sec:prelim}

\subsection{H\"older spaces}\label{sec:prepare}

In the computations we encounter space-time function or distribution spaces that exhibit one or more of the following complications:
they may have temporal singularities at the initial time;
they may have spatial singularities at the boundary or at the edges of the cubic domain;
they may themselves `live' on the boundary of the domain $D$.
A completely general framework to deal with all of these seems rather cumbersome,
therefore we choose to stay in our very concrete examples.

We have two main settings: a $3$-dimensional spatial and a $1+3$-dimensional spatiotemporal one. The former is used in dealing with the stochastic objects of the PAM, while the latter is used for the stochastic objects of the $\Phi^4_3$ as well as the analytic side of both equations.
We now introduce some notation that refer to different, but rather analogous objects in the two settings.

In the spatial case ($d=3$) we use the Euclidean scaling, according to which sizes of vectors are measured and denoted by $\|\cdot\|$.
The distance function to a given set $S$ is denoted by $|\cdot|_S$ and the projection to $S$ (wherever well-defined) is denoted by $\pi_S$.
The scaled dimension therefore coincides with the usual dimension, and is also denoted by $\frs=3$.
It will be convenient to shift the whole problem by
$(1,1,1)$, that is, to work on the cube $D=[0,2]^3$.
We denote by $\d$ the boundary of $D$ and
by $\d^2$ the edges of $D$.
By symmetry, large part of the calculation can be restricted
on the tetrahedron $Q=\{x\in\R^3:\,0\leq x_3\leq x_1\leq x_2\leq 1\}$ (see Figure~\ref{fig:Q}),
where some notations simplify: we have $|x|_\d=x_3$, $|x|_{\d^2}=x_1$, and $\pi_\d (x_1,x_2,x_3)=(x_1,x_2,0)$. Let us also use the shorthand $Q_\d=\pi_\d Q=\{y=(y_1,y_2,0)\,:\,0\leq y_1\leq y_2\leq 1\}$.

\begin{figure}
\begin{center}
\begin{tikzpicture}[tdplot_main_coords]
\def\R{2}
\def\RR{4}

  \foreach \x in {0,90,180,270} {\coordinate (Base-\x) at ($(\x:\RR)-(\RR,0)$); };

\begin{scope}[yshift=\RR cm]
  \foreach \x in {0,90,180,270} {\coordinate (Top-\x) at ($(\x:\RR)-(\RR,0)$) ; };
\end{scope}

  \foreach \x in {0,90,180,270} {\coordinate (base-\x) at ($(\x:\R)-(\R,0)$); };

\begin{scope}[yshift=\R cm]
  \foreach \x in {0,90,180,270} {\coordinate (top-\x) at ($(\x:\R)-(\R,0)$) ; };
\end{scope}
\begin{scope}[yshift=5 cm]
\coordinate (arrow3) at (0,0);
\end{scope}

\coordinate (arrow2) at (-5,5);
\coordinate (arrow1) at (-5,-5);

\draw[black!15,thick] (base-270)
  \foreach \x in {0,90,180,270} {-- (base-\x)};
\draw[black!15,thick] (top-270)
  \foreach \x in {0,90,180,270} {-- (top-\x)};
\foreach \x in {0,90,180,270} {\draw[black!15,thick] (base-\x) -- (top-\x);}

\draw[black!35,thick] (Base-270)
  \foreach \x in {0,90,180,270} {-- (Base-\x)};
\draw[black!35,thick] (Top-270)
  \foreach \x in {0,90,180,270} {-- (Top-\x)};
\foreach \x in {0,90,180,270} {\draw[black!35,thick] (Base-\x) -- (Top-\x);}

\foreach \x in {1,2,3} {\draw[black!35,thick,->,shorten >=0.3cm] (0,0) -- (arrow\x);}
\foreach \x in {1,2,3} {\node at (arrow\x) {$x_\x$};}

\draw[thick] (base-0) -- (base-90) -- (base-180);
\draw[thin] (base-180) -- (base-0);
\draw[thick] (base-0) -- (top-180);
\draw[thick] (base-90) -- (top-180);
\draw[thick] (base-180) -- (top-180);
\draw[fill=blue!80!black, fill opacity=0.3] (base-90) -- (base-180) -- (top-180);
\draw[fill=blue, fill opacity=0.2] (base-90) -- (base-0) -- (top-180);
\end{tikzpicture}
\end{center}
\caption{The tetrahedron $Q$ (blue) and the domain $D$ (gray).}\label{fig:Q}
\end{figure}
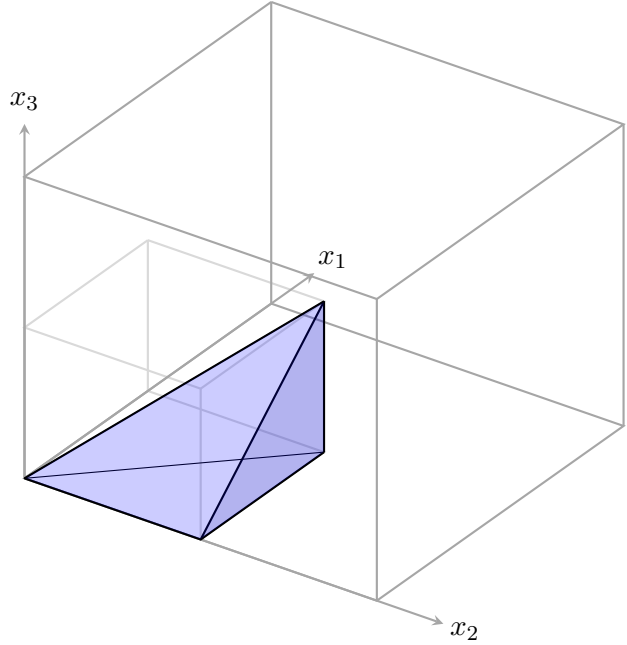

The analogous objects in the spatiotemporal case ($d=4$) are as follows. The scaling is parabolic, and correspondingly the scaled dimension is $\frs=5$.
Points in $\R^4$ are denoted by $(x_0,x_1,x_2,x_3)$ or $(t,x_1,x_2,x_3)$, depending on convenience.
This time $\d$ and $\d^2$ denote the boundary and the edges of $\R\times D$, respectively, and we set
$Q=\{x\in\R^4:\,0\leq x_3\leq x_1\leq x_2\leq 1\}$ as well as $Q_\d=\pi_\d Q$.
Similarly as before, on $Q$, $|x|_\d=x_3$, $|x|_{\d^2}=x_1$, and $\pi_\d (x_0,x_1,x_2,x_3)=(x_0,x_1,x_2,0)$.
Below we collect a few properties of H\"older spaces and their weighted variants.

First recall that for $\alpha<0$, the H\"older space $\CC^\alpha$ on $\R^d$ is defined as the space of distributions $u$ that satisfy the bounds
\begin{equ}\label{eq:holder def0}
u(\psi_y^\lambda)\leq C \lambda^\alpha\;,
\end{equ}
for some $C$, uniformly in $y$ over compacts, $\lambda\in(0,1]$, and appropriately normalised test functions.
Here and below $\psi_y^\lambda$ is the rescaling (to scale $\lambda$) and recentering (around $y$) of $\psi$ in such a way that its integral is the same as that of $\psi$. More details can be found e.g. in \cite[Sec~3]{H0}.
Keeping in mind that the lower dimensional set $\d$ admits a different scaling, it is natural to define $\CC^\alpha(\d)$ as the set of distributions that vanish on test functions whose support does not intersect $\d$ and satisfy the bounds 
\begin{equ}
u(\psi_y^\lambda)\leq C \lambda^{\alpha-1}.
\end{equ}
The best proportionality constant $C$ is also denoted by $\|u\|_{\CC^\alpha(\d)}$.
Alternatively, we may put the extra scaling factor to the test function: for any set $S$ with a well-defined (scaled) dimension $\frm$,
we set $\psi_y^{\lambda,S}=\lambda^{\frs-\frm}\psi_y^\lambda$, in which case the required bound takes the more natural form $u(\psi_y^{\lambda,\d})\leq C \lambda^{\alpha}$. 
The definition of $\CC^\alpha(\d)$ for $\alpha\in[0,1)$ is straightforward. Extending the scale to $\alpha\geq 1$ is more delicate, but we do not need that generality.
Note that it follows by definition that, for distributions $u$ supported on $\d$,
\begin{equ}\label{eq:extending from boundary easy}
\|u\|_{\CC^{(\alpha\wedge0)-1}(\R^d)}\leq\|u\|_{\CC^\alpha(\d)}.
\end{equ}
A particular special case is the `Dirac distribution' on $\d$ given by
\begin{equ}\label{eq:dirac}
\scal{\delta_\partial,\varphi}:=\int_{\partial }\varphi(x)\,dx.
\end{equ}
\begin{remark}
Although $\delta_\d$ is a distribution with H\"older regularity $-1$, it can be multiplied with functions $f$ well short of being $\CC^1$ by setting
\begin{equ}
\scal{\delta_\partial f,\varphi}:=\int_{\partial }f(x)\varphi(x)\,dx.
\end{equ}
This expression is perfectly meaningful if $f:\d\to\R$ is integrable or if $f: \R^d\to\R$ has an integrable trace on $\d$.
\end{remark}

For weighted H\"older spaces we take a domain $S \subset \R^d$ and a `boundary' $P \subset S$, with scaled codimension (as a subset of $S$) $k$.
The few cases of interest to us are the triples $(S,P,k)$ given by $(\R^d,\d,1)$, $(\R^d,\d^2,2)$, $(\d,\d^2,1)$, and in the spatiotemporal case $(\R^d,P_0,2)$.
Furthermore, as mentioned before, we often replace $\R^d$ by $Q$ and $\d$ by $Q_\d$, respectively.
For $\eta\leq\alpha\leq 0$ we define the space $\CC^{\alpha,\eta}_P(S)$ as the set of distributions on $\R^d\setminus P$ that vanish on test functions whose support does not intersect $S$ and satisfy the bounds
\begin{equ}
u(\psi_x^{\lambda,S})\leq C \lambda^\alpha |x|_P^{\eta-\alpha}
\end{equ}
uniformly in $x\in S\setminus P$, $\lambda\in(0,\half|x|_{P}]$, and normalised test functions $\psi$ with support in the unit ball.
For $\eta\leq\alpha\in(0,1)$ and $\eta\leq 0$, we set $\CC^{\alpha,\eta}_P(S)$ to be
the set of functions that belong to $\CC^{0,(\eta\wedge0)}_P(S)$ and satisfy the bounds
\begin{equ}
|u(x)-u(y)|\leq C \|x-y\|^\alpha|x|_P^{\eta-\alpha}
\end{equ}
uniformly in $x,y\in S\setminus P$ such that $\|x-y\|\leq \half|x|_P$.

The first important property of such spaces is that whenever the singularity is integrable, that is when $\eta>-k$, 
they are canonically included in the space of distributions on all of $\R^d$.
In our setting this can be formulated in the following way, which is a small modification of e.g.\ \cite[Prop~6.9]{H0}, \cite[Prop~2.15]{GH17}:
\begin{proposition}\label{prop:weighted holder extend}
Let $(S,P,k)$ be as above and let $-k<\eta\leq\alpha< 1$. Then the space $\CC^{\alpha,\eta}_P(S)$ canonically 
embeds into $\CC^{\eta}(S)$ in the sense that, for every $\zeta \in \CC^{\alpha,\eta}_P(S)$ there exists a unique 
$\hat \zeta \in \CC^{\eta}(S)$ such that $\zeta(\phi) = \hat \zeta(\phi)$ for all test functions $\phi$ with $\supp \phi \cap P = \emptyset$.
\end{proposition}
Note that in the `opposite' direction one has the trivial inclusion $\CC^{\eta}(S)\subset\CC^{\eta,\eta}_P(S)$.
When the singularity is non-integrable, there is in some cases still a 
natural way to obtain a distribution on $\R^d$.
\begin{proposition}\label{prop:boundary correction}
Let $\eta\in(-2,-1)$. For $u\in \CC_{\d}^{0,\eta}(Q)$ define the distribution $\scR u$ by
\begin{equ}\label{eq:boundary correction}
(\scR u)(\psi):=\int_{Q} u(x)\big(\psi(x)- \psi(\pi_\d x)\big)\,dx,\quad\psi\in\cC_c^\infty(Q).
\end{equ}
Then the mapping $u\mapsto\scR u$ is continuous from $\CC_{\d}^{0,\eta}(Q)$ to $\CC^{\eta}(Q)$.
\end{proposition}
\begin{remark}
Note that while the projection $\pi_P$ may be not be well-defined on a set of measure $0$, the integral in \eqref{eq:boundary correction} is well-defined.
\end{remark}
\begin{proof}
Since $\scR$ is linear, it suffices to show that it is bounded. 
Take $u\in \CC_{\d}^{0,\eta}(Q)$ with norm $1$, 
and a test function of the form $\psi_y^\lambda$.
We distinguish two cases, depending on whether $\lambda\lessgtr\half|y|_{\d}$. If $\lambda\leq \half|y|_{\d}$, then one simply has
\begin{equ}
|\scR u(\psi_y^\lambda)|=|u(\psi_y^\lambda)|\leq |y|_{\d}^{\eta}\lesssim \lambda^{\eta}.
\end{equ}
If $\lambda\geq \half|x|_{\d }$, then estimating $\psi_y^\lambda(x)- \psi_y^\lambda(\pi_\d x)$ by $|y|_\d|\nabla\psi_y^\lambda|$ yields the bound
\begin{equ}
|\scR u_\eps(\psi_y^\lambda)|\lesssim\int_{Q\cap\supp\psi_y^\lambda}
|x|_{\d }^{\eta}|x|_{\d }\lambda^{-\frs-1}\,dy
\lesssim  \lambda^{\eta}.
\end{equ}
using that $1+\eta>-1$ implies that $|x|_{\d}^{1+\eta}$ is integrable. This finishes the proof.
\end{proof}
\begin{remark}\label{rem:mini}
If $u$ happens to be a (uniformly) smooth function, then the extension
defined by \eqref{eq:boundary correction} differs from the `obvious' one by $\delta_P f$ with some smooth function $f$ on $P$. However, $u$ and $\scR u$ always coincide on test functions supported away from the boundary.
\end{remark}
Multiplying weighted distributions follows the usual rules in the regularity exponent, while the behaviour of the weight exponent can be read out from e.g. \cite[Prop~6.12]{H0}.
\begin{proposition}\label{prop:weighted holder mult}
Let $(S,P,k)$ be as above, $\eta\leq\alpha\leq 0$ and $\bar \eta\leq \bar\alpha<1$, such that $\alpha+\bar\alpha>0$.
Then the multiplication map is continuous from $\CC_P^{\alpha,\eta}(S)\times\CC_P^{\bar\alpha,\bar\eta}(S)$ to $\CC_P^{\alpha,(\alpha+(\bar\eta\wedge0))\wedge\eta}(S)$.
\end{proposition}

Finally, in some examples like $\Phi^4_3$ we have two boundaries with \emph{three} different singularities, one at each boundary and one at their intersection.
For such a setup we take $S$ as before but now with two boundaries $P_0$ and $P_1$. We only ever encounter situations when their codimensions are $2$ and $1$ respectively, and the codimension of their intersection is $3$.
We take a `weight triple' $w=(\eta,\sigma,\mu)$, and a regularity exponent $\alpha<1$. We always assume $\eta,\sigma,\mu\leq \alpha$, as well as $\mu\leq 0\wedge\eta\wedge\sigma$.
Then for $\alpha\leq 0$ we define the space $\CC^{\alpha,w}_{P_0,P_1}(S)$ as the set of distributions on $\R^d\setminus (P_0\cup P_1)$ that vanish on test functions whose support does not intersect $S$ and satisfy the bounds
\begin{equ}
u(\psi^{\lambda,S}_x)\leq C \lambda^{\alpha}|x|_{P_0}^{\eta-\alpha}|x|_{P_1}^{\mu-\eta}
\end{equ}
uniformly over $\lambda\leq \half|x|_{P_0}\leq \quarter|x|_{P_1}$ and
\begin{equ}
u(\psi^{\lambda,S}_x)\leq C \lambda^{\alpha}|x|_{P_1}^{\sigma-\alpha}|x|_{P_0}^{\mu-\sigma}
\end{equ}
uniformly over $\lambda\leq \half|x|_{P_1}\leq \quarter|x|_{P_0}$.
\begin{remark}
This definition is a slight refinement of the one in \cite[Def~4.7]{GH17}. Indeed, if we denote the (only $3$-parameter) spaces therein by $\tilde\CC^{w}_{P_0,P_1}$, then as long as $\eta>-2$ and $\sigma>-1$, $\CC^{\alpha,w}_{P_0,P_1}$ embeds into $\tilde\CC^{w}_{P_0,P_1}$.
\end{remark}
For $\alpha\in(0,1)$, 
we set $\CC^{\alpha,w}_{P_0,P_1}(S)$ to be
the set of functions that belong to $\CC^{0,(\eta\wedge0,\sigma\wedge0,\mu)}_P(S)$ and satisfy the bounds
\begin{equ}
|u(x)-u(y)|\leq C \|x-y\|^\alpha|x|_{P_0}^{\eta-\alpha}|x|_{P_1}^{\mu-\eta}
\end{equ}
uniformly in $x,y\in S\setminus (P_0\cup P_1)$ such that $\|x-y\|\leq\half |x|_{P_0}\leq\quarter|x|_{P_1}$ and the corresponding symmetric bounds near $P_1$.
The following properties either follow directly from the definition or are straightforward adaptations of some simple results of \cite{GH17}.
\begin{proposition}\label{prop:weighted holder2}
Consider the above setting of $S$, $P_0$, $P_1$, $w$, and $\alpha$.
\begin{enumerate}[(i)]
\item If $\sigma>0$, then the trace function $\Tr_{P_1}$ onto $P_1$ maps $\CC_{P_0,P_1}^{\alpha,w}(S)$ continuously into $\CC_{P_0\cap P_1}^{\sigma,\eta\wedge\mu}(P_1)$.
\item If $\eta>-2$, $\sigma>-1$, and $\mu>-3$, then the space $\CC^{\alpha,w}_{P_0,P_1}(S)$ continuously embeds into $\CC^{\eta\wedge\sigma\wedge\mu}(S)$.
\item If $\alpha,\bar\alpha\in(0,1)$ and $\eta,\sigma\leq0$, then the multiplication map is continuous from $\CC_{P_0}^{\alpha,\eta}(S)\times\CC_{P_1}^{\alpha,\sigma}(S)$ to $\CC_{P_0,P_1}^{\alpha,(\eta,\sigma,\eta+\sigma)}(S)$.
\end{enumerate}
\end{proposition}


\subsection{Kernel bounds}\label{sec:kernel computations}
Next we derive some general bounds on integrals involving singular kernels.
The two important quantities for our bounds are the scaled dimension $\frs$ and the ``blowup'' of the kernel that is denoted by $\frb>0$.
We are looking at a very specific blowup scenario in which we assume
\begin{equ}[e:constraint]
\frb\leq\frs-1,\qquad 2\frb-\frs=1. 
\end{equ}
In the two examples of the paper, we will have $\frs=3,\frb=2$ (PAM), and $\frs=5,\frb=3$ ($\Phi^4$).
Typically the kernels we work with are \emph{not} translation invariant, which motivates the following definition.
Let $\scG$ be the class of functions $G:\R^d\times\R^d\to\R$ that admit a decomposition
\begin{equ}
G(x,y)=\sum_{n\in\N}G^n(x,y),
\end{equ}
and such that there exists a reflection $T:\R^d\to\R^d$, two sets $A_1, A_2\subset \R^d$,
and a constant $C>0$ such that:
\begin{claim}
\item $G^n$ is supported in $\{(x,y) \in A_1 \times A_2\,:\, \|x-Ty\|\le C 2^{-n}\}$;
\item one has the bounds $|D_1^kD_2^\ell G^n(x,y)|\le C 2^{n(\frb+|k|+|\ell|)}$ for all 
$x\in\intr A_1$, $y\in \intr A_2$ and all multiindices $k,\ell$ with $|k|\leq 1$, $|\ell|\leq 1$.
\end{claim}
A trivial but important consequence of the first point is that $G^n$ is identically $0$ for $2^{-n}\lesssim d(x,TA_1)$.
We moreover assume $A_1$, $A_2$ to be sufficiently ``nice'': for our applications it will 
be more than sufficient if we assume them to consist of a finite union of direct products 
of intervals.
If we want to emphasise the choice of parameters, we write $T(G), A_1(G), A_2(G)$.
\begin{example}\label{example1}
Consider the $1+1$-dimensional heat kernel $\cK$. The homogeneous Neumann heat kernel on the positive half line is then given by 
\begin{equ}\label{eq:example1}
\cG(t,x,t',x')=\cK(t-t',x-x')\bone_{x,x'\geq 0}+\cK(t-t',x+x')\bone_{x,x'\geq 0}.
\end{equ}
Both terms belong to $\scG$: for the first one has $T=\id$, $A_1=A_2=\R\times\R_+$, while for the second one has $T(t',x')=(t',-x')$, $A_1=A_2=\R\times\R_+$. In such a simple situation the above formalism would be an overkill, and some of the calculations below were actually performed in \cite{GH17}. In higher dimensions however, there are several different reflections that need to be handled, hence the more generic setup of $\scG$.
\end{example}
For $G\in\scG$ we define
\begin{equ}[e:mollification]
G_\eps(x,y)=\int G(x,z)\rho_\eps(y-z)\,dz.
\end{equ}
We need three basic bounds for kernels of this type.
We remark the elementary bounds, for $\delta_1,\delta_2>0$, $\alpha_1<0<\alpha_2$,
\begin{equ}
\sum_{n\in\N:\, 2^{-n}\lesssim\delta_1}2^{n\alpha_1}\lesssim \delta_1^{-\alpha_1},\qquad \sum_{n\in\N:\, 2^{-n}\gtrsim\delta_2}2^{n\alpha_2}\lesssim \delta_2^{-\alpha_2},
\end{equ}
that are repeatedly used in the proofs.
\begin{lemma}\label{lem:kernels1}
Let $G,\hat G\in\scG$ with $A_1 = \hat A_1$ and such that \eqref{e:constraint} holds.
Take $\gamma\in(0,1)$. Then for all $x,x'\in A_1$ with $|x-x'|\lesssim d(x, \hat T\hat A_2)$ one has
\begin{equs}
\int G_\eps\hat G_\eps(x,y)\,dy&\lesssim |\eps\vee d(x,\hat T\hat A_2)|^{-1};\label{eq:kernels1-a}
\\
\int G_\eps\hat G_\eps(x,y)-G_\eps\hat G_\eps(x',y)\,dy&\lesssim \|x-x'\|^\gamma|\eps\vee d(x,\hat T\hat A_2)|^{-1-\gamma}.\label{eq:kernels1-b}
\end{equs}
\end{lemma}
\begin{proof}
Let us use the shorthand $[x]=d(x,\hat T\hat A_2)$.
We only prove the second bound since the first one is easier.
First note that one has the decomposition
\begin{equ}
\hat G_\eps(x,y)=\sum_{n\in\N} \hat G^{n}_{\eps}(x,y),
\end{equ}
where $\supp \hat G^n_\eps(x,\cdot) \subset\{y:\|x-\hat T y\|\lesssim\eps+2^{-n},d(y,\hat A_2)\leq\eps\}$
and one has bounds
\begin{equs}
|\hat G_\eps^n(x,y)|
&\lesssim (\eps\vee 2^{-n})^{-\frs}2^{n(\frb-\frs)}
\\
|\hat G_\eps^n(x,y)-\hat G_\eps^n(x',y)|
&\lesssim |x-x'|^\gamma (\eps\vee 2^{-n})^{-\frs}2^{n(\frb-\frs+\gamma)}
\end{equs}
for all $x,x'\in A_1$ and $y\in\R^d$, and similarly for $G$.
In particular, the volume of $(\supp G^n_\eps)\cap(\supp \hat G^m_\eps)$ is bounded by $(\eps\vee(2^{-n}\wedge 2^{-m}))^{\frs}$.
Therefore, the left-hand side of \eqref{eq:kernels1-b} is bounded by
\begin{equs}[eq:triv eps bound]
\,&|x-x'|^\gamma\sum_{n,m\in\N}(\eps\vee(2^{-n}\wedge 2^{-m}))^{\frs}(\eps\vee2^{-n})^{-\frs}2^{n(\frb-\frs)}
\\
&\qquad\qquad\times (\eps\vee2^{-m})^{-\frs}2^{m(\frb-\frs)}(2^{n\gamma}\vee 2^{m\gamma})\\
&=|x-x'|^\gamma\sum_{n,m\in\N}2^{(n+m)(\frb-\frs)}(\eps\vee2^{-n}\vee2^{-m})^{-\frs}(2^{n\gamma}\vee 2^{m\gamma})
\end{equs}
At this stage the roles of $n$ and $m$ are symmetric, so we can bound the above sum by
\begin{equs}[eq:triv eps bound2]
\,\sum_{2^{-n}\gtrsim 2^{-m}} & 2^{(n+m)(\frb-\frs)}(\eps\vee 2^{-n})^{-\frs}2^{m\gamma}
\lesssim
\sum_{n\in\N}2^{n(2\frb-2\frs+\gamma)}(\eps\vee 2^{-n})^{-\frs}
\\
&
\lesssim \eps^{-\frs}\sum_{2^{-n}\lesssim\eps}2^{n(2\frb-2\frs+\gamma)}
+\sum_{2^{-n}\gtrsim\eps}2^{n(2\frb-\frs+\gamma)}\lesssim \eps^{-1-\gamma},
\end{equs}
where we used that $\frb-\frs+\gamma<0$ and $2\frb-\frs=1$. This yields the required bound if $[x]\lesssim\eps$.
For $[x]\gtrsim\eps$ we make use of the property that $\hat G^m_\eps$ is identically $0$ for $2^{-m}\lesssim [x]$.
Therefore, instead of \eqref{eq:triv eps bound}, we now get the bound
\begin{equs}
\,&|x-y|^\gamma\sum_{\substack{2^{-m}\gtrsim [x]\\ n\in\N}}2^{(n+m)(\frb-\frs)}(2^{-n}\vee2^{-m})^{-\frs}(2^{n\gamma}\vee 2^{m\gamma}).
\end{equs}
For the part of the sum where $2^{-n}\gtrsim 2^{-m}$, we get a bound
\begin{equ}
\sum_{2^{-n}\gtrsim 2^{-m}\gtrsim [x]}2^{n\frb}2^{m(\frb-\frs+\gamma)}\lesssim\sum_{2^{-m}\gtrsim [x]}2^{m(2\frb-\frs+\gamma)}\lesssim [x]^{-1-\gamma}\;,
\end{equ}
as required. Concerning the $2^{-m}\gtrsim 2^{-n}$ regime, we can write
\begin{equ}
\sum_{2^{-m}\gtrsim (2^{-n}\vee [x])}2^{m\frb}2^{n(\frb-\frs+\gamma)}\lesssim\sum_{2^{-m}\gtrsim[x]}2^{m(2\frb-\frs+\gamma)}\lesssim[x]^{-1-\gamma}\;,
\end{equ}
which finishes the proof.
\end{proof}
\begin{lemma}\label{lem:kernels2}
Assume the setting of Lemma \ref{lem:kernels1}.
Then for all $x\in A_1$ with $d(x,\hat T\hat A_2)\gtrsim \eps$, one has
\begin{equ}\label{eq:kernels2}
\Big|\int \big(G_\eps\hat G_\eps-G\hat G\big)(x,y)\,dy\Big|\lesssim \eps^\gamma \big(d(x,\hat T\hat A_2)\big)^{-1-\gamma}.
\end{equ}
\end{lemma}
\begin{proof}
We write $G_\eps\hat G_\eps-G\hat G=(G_\eps-G)\hat G+G_\eps(\hat G_\eps-\hat G)$
and bound the two corresponding integrals separately (unfortunately the two cases are not exactly symmetric).
We first treat $(G_\eps-G)\hat G$. Writing $G^n_\Delta=G_\eps^n-G^n$, the quantity to bound is
\begin{equ}
\sum_{m,n\in\N}\int \big(G_{\Delta}^n\hat G^m(x,y)\big)\,dy.
\end{equ}
Using again the shorthand $[x]=d(x,\hat T\hat A_2)$, we see 
that the sum over $m$ can be restricted to the range $2^{-m}\gtrsim[x]$, since $\hat G^m$ 
vanishes otherwise.
The easiest case is $2^{-n}\lesssim\eps$, one can simply use the bounds $|{\supp G^n_\Delta}|\lesssim 2^{-n\frs}$, $\sup|G^n_\Delta|\lesssim 2^{n\frb}$, $\sup|G^m|\lesssim 2^{m\frb}$. This yields
\begin{equ}
\sum_{\substack{2^{-m}\gtrsim[x]\\2^{-n}\lesssim\eps}}
\Big|\int \big(G_{\Delta}^n\hat G^m(x,y)\big)\,dy\Big|
\lesssim\sum_{\substack{2^{-m}\gtrsim[x]\\2^{-n}\lesssim\eps}}2^{n(\frb-\frs)}2^{m\frb}\lesssim\eps^{\frs-\frb}[x]^{-\frb}.
\end{equ}
Recalling that $[x]\gtrsim\eps$, $\frs\geq\frb+1$, and $2\frb-\frs=1$, one sees that this is indeed bounded by $\eps[x]^{-2}$, as required.
For $2^{-n}\gtrsim\eps$ we split the integral into two regions, depending on the distance of $y$ to $\d A_2$.
Define $\d^\eps=\{y: d(y,\d A_2)\lesssim\eps\}$.
If $y\notin \d^\eps$, then one can use the differentiability of $G$ in the second variable to get the bound $|G_\Delta^n(x,y)|\lesssim \eps^\gamma 2^{n(\frb+\gamma)}$. Therefore,
\begin{equ}\label{eq:sum1}
\sum_{\substack{2^{-m}\gtrsim[x]\\2^{-n}\gtrsim\eps}}
\Big|\int_{(\d^\eps)^c} \big(G_{\Delta}^n\hat G^m(x,y)\big)\,dy\Big|
\lesssim
\sum_{\substack{2^{-m}\gtrsim[x]\\2^{-n}\gtrsim\eps}}
\big(2^{-n\frs}\wedge 2^{-m\frs}\big)\eps^\gamma 2^{n(\frb+\gamma)} 2^{m\frb}.
\end{equ}
For the regime $2^{-n}\gtrsim 2^{-m}$ one has
\begin{equ}
\sum_{2^{-n}\gtrsim 2^{-m}\gtrsim[x]}2^{-m\frs}\eps^\gamma 2^{n(\frb+\gamma)}2^{m\frb}\lesssim
\eps^\gamma\sum_{2^{-m}\gtrsim[x]}2^{m(2\frb-\frs+\gamma)}\lesssim
\eps^{\gamma}[x]^{-1-\gamma}.
\end{equ}
For $2^{-m}\gtrsim 2^{-n}\gtrsim [x]$ one gets
\begin{equ}
\sum_{2^{-m}\gtrsim 2^{-n}\gtrsim[x]}2^{-n\frs}\eps^\gamma 2^{n(\frb+\gamma)}2^{m\frb}
\lesssim\eps^\gamma\sum_{2^{-m}\gtrsim|x|}2^{m(2\frb-\frs+\gamma)}
\lesssim\eps^{\gamma}[x]^{\frs-2\frb-\gamma}=\eps^{\gamma}[x]^{-1-\gamma}.
\end{equ}
Finally, in the case $2^{-n}\lesssim [x]$ the sum becomes
\begin{equ}
\sum_{2^{-m}\gtrsim[x]\gtrsim 2^{-n}\gtrsim\eps}2^{-n\frs}\eps^\gamma 2^{n(\frb+\gamma)}2^{m\frb}\lesssim\eps^\gamma[x]^{-\frb}\sum_{[x]\gtrsim 2^{-n}} 2^{n(\frb+\gamma-\frs)}\lesssim\eps^{\gamma}[x]^{\frs-2\frb-\gamma}.
\end{equ}
All of these bounds are of the required order. It remains to treat the $\d^\eps$ portion of the integral.
Note that in this case the size of the region of integration is at most of order $\eps \big(2^{-n(\frs-1)}\wedge 2^{-m(\frs-1)}\big)$.
Combining this with the trivial supremum bounds one gets
\begin{equ}\label{eq:sum2}
\sum_{\substack{2^{-m}\gtrsim[x]\\2^{-n}\gtrsim\eps}}
\Big|\int_{\d^\eps} \big(G_{\Delta}^n\hat G^m(x,y)\big)\,dy\Big|
\lesssim
\sum_{\substack{2^{-m}\gtrsim[x]\\2^{-n}\gtrsim\eps}}
\eps \big(2^{-n(\frs-1)}\wedge 2^{-m(\frs-1)}\big) 2^{n\frb} 2^{m\frb}.
\end{equ}
Again, first bound the sum over $2^{-n}\gtrsim 2^{-m}$:
\begin{equ}
\sum_{2^{-n}\gtrsim 2^{-m}\gtrsim[x]}\eps 2^{-m(\frs-1)}2^{n\frb}2^{m\frb}\lesssim
\eps\sum_{2^{-m}\gtrsim[x]}2^{m(2\frb-\frs+1)}\lesssim
\eps[x]^{-2},
\end{equ}
as required.
Next, in the case $[x]\lesssim 2^{-n}\lesssim 2^{-m}$ one has
\begin{equ}
\sum_{2^{-m}\gtrsim 2^{-n}\gtrsim[x]}\eps 2^{-n(\frs-1)} 2^{n\frb}2^{m\frb}
\lesssim\eps \sum_{2^{-n}\gtrsim[x]}2^{n(2\frb-\frs-1)}\lesssim \eps[x]^{-2},
\end{equ}
as required.
Finally, for $2^{-n}\lesssim[x]$ the sum becomes
\begin{equs}
\sum_{2^{-m}\gtrsim[x]\gtrsim 2^{-n}\gtrsim\eps}
\eps 2^{-n(\frs-1)} 2^{n\frb}2^{m\frb}
&\lesssim\eps[x]^{-\frb}\sum_{[x]\gtrsim 2^{-n}\gtrsim\eps} 2^{n(\frb-\frs+1)}
\\
&\lesssim\eps^{\gamma}[x]^{-\frb}\sum_{[x]\gtrsim 2^{-n}\gtrsim\eps} 2^{n(\frb-\frs+\gamma)}\lesssim \eps^{\gamma}[x]^{-1-\gamma},
\end{equs}
using that $\frb-\frs+\gamma<0$.
Combining all the cases finishes the term $(G_\eps-G)\hat G$.

It now remains to do a similar calculation for $G_\eps(\hat G_\eps-\hat G)$.
Writing $\hat G_\Delta^m=G^m_\eps-G^m$, the quantity to bound is
\begin{equ}
\sum_{m,n\in\N}\int \big(G_\eps^n\hat G^m_\Delta(x,y)\big)\,dy.
\end{equ}
As before, the sum over $m$ can be restricted to the regime $2^{-m}\gtrsim[x]$, since both $\hat G^m$ and $\hat G^m_\eps$ vanish otherwise. For the former this is obvious and for the latter this follows from the assumption $\eps\lesssim[x]$.
To bound the sum over $2^{-n}\lesssim \eps$, one can use $\|G^n_\eps\|_{L^1}=\|G^n\|_{L^1}\lesssim 2^{n(\frb-\frs)}$, with the trivial bound $\sup|G^m_{\Delta}|\lesssim 2^{m\frb}$. This yields the same bound as before, namely
\begin{equ}
\sum_{\substack{2^{-m}\gtrsim[x]\\2^{-n}\lesssim\eps}}
\Big|\int \big(G_\eps^n\hat G^m_\Delta(x,y)\big)\,dy\Big|
\lesssim\sum_{\substack{2^{-m}\gtrsim[x]\\2^{-n}\lesssim\eps}}2^{n(\frb-\frs)}2^{m\frb}\lesssim\eps^{\frs-\frb}[x]^{-\frb}.
\end{equ}
For $2^{-n}\gtrsim\eps$ we split the integral to $\d^\eps$ and $(\d^\eps)^c$ as before, and use that if $y\notin \d^\eps$, 
then one due to the differentiability of $\hat G$ in the second variable one has the bound $|\hat G_\Delta^m(x,y)|\lesssim \eps^\gamma 2^{m(\frb+\gamma)}$.
Therefore,
\begin{equ}
\sum_{\substack{2^{-m}\gtrsim[x]\\2^{-n}\gtrsim\eps}}
\Big|\int_{(\d^\eps)^c} \big(G_{\eps}^n\hat G^m_\Delta(x,y)\big)\,dy\Big|
\lesssim
\sum_{\substack{2^{-m}\gtrsim[x]\\2^{-n}\gtrsim\eps}}
\big(2^{-n\frs}\wedge 2^{-m\frs}\big)2^{n\frb}\eps^\gamma 2^{m(\frb+\gamma)}.
\end{equ}
We leave it as an exercise to the reader to treat this sum similarly to the one in \eqref{eq:sum1}.
Finally, concerning the integral over $\d^\eps$ we get
\begin{equ}
\sum_{\substack{2^{-m}\gtrsim[x]\\2^{-n}\gtrsim\eps}}
\Big|\int_{\d^\eps} \big(G_{\eps}^n\hat G^m_\Delta(x,y)\big)\,dy\Big|
\lesssim
\sum_{\substack{2^{-m}\gtrsim[x]\\2^{-n}\gtrsim\eps}}
\eps \big(2^{-n(\frs-1)}\wedge 2^{-m(\frs-1)}\big) 2^{n\frb} 2^{m\frb}.
\end{equ}
The right-hand side is now precisely the same as in \eqref{eq:sum2}, and therefore using the already established bound the proof is finished.
\end{proof}

\begin{lemma}\label{lem:kernels3}
Let $G\in\scG$, let $A\subset\R^d$ be an open convex set, and define $\tilde G(x,y)=G(x,y)\bone_{y\in A}$.
Take $\gamma\in(0,1),\gamma'\in[-1,1]$.
Then for all $x,x'$ with $\|x-x'\|\lesssim d(x,\d A)$ one has
\begin{equs}[eq:kernels3]
\Big|\int \Big(\tilde G_\eps^2(x,y)  -\tilde G_\eps^2(x',y)\Big)- & \Big(G_\eps^2(x,y)\bone_{y\in A}  -G_\eps^2(x',y)\bone_{y\in A}\Big)\,dy\Big|
\\
&\lesssim \|x-y\|^\gamma\eps^{\gamma'} (d(x,\d A))^{-1-\gamma-\gamma'}.
\end{equs}
(Recall the notation introduced in \eqref{e:mollification}.)
\end{lemma}
\begin{proof}
Let us again use the shorthand $[x]=d(x,\d A)$.
Clearly it suffices to consider the extremal cases $\gamma'\in\{-1,1\}$. 
For $\gamma'=-1$ we can bound the integrals of the two terms in the big brackets as in \eqref{eq:triv eps bound}-\eqref{eq:triv eps bound2}.
In fact this gives the required bound not only for $\gamma'=-1$, but also for $\gamma'=1$ in case $[x]\lesssim\eps$.

In the case $\gamma'=1$, $[x]\gtrsim \eps$,
the integrand in \eqref{eq:kernels3} vanishes identically on $[y]\gtrsim\eps$.
On the remaining region we again bound the integrals of the two terms in the big brackets separately. 
They are essentially identical calculations, we only detail the first one.
One can write
\begin{equs}
\int_{[y]\lesssim\eps} & \tilde G_\eps^2(x,y)  -\tilde G_\eps(x',y)\,dy
\\&\lesssim\|x-x'\|^\gamma\sum_{n,m\in\N}\Big|(\supp \tilde G_\eps^n)\cap(\supp\tilde G_\eps^m)\cap \{[y]\lesssim\eps\}\Big|2^{n\frb}2^{m\frb}(2^{n\gamma}\vee 2^{m\gamma}).
\end{equs}
Since $[x]\gtrsim \eps$, only terms with $2^{-n},2^{-m}\gtrsim [x]$ contribute to the sum. In this case one has
\begin{equ}
\Big|(\supp \tilde G_\eps^n)\cap(\supp\tilde G_\eps^m)\cap \{[y]\lesssim\eps\}\Big|
\lesssim \eps (2^{-n}\wedge 2^{-m})^{\frs-1}.
\end{equ}
Since the above sum is symmetric under $n \leftrightarrow m$, we can bound it by
\begin{equ}
\eps\sum_{2^{-m}\gtrsim 2^{-n}\gtrsim[x]}2^{n(\frb+\gamma+1-\frs)}2^{m\frb}
\lesssim \eps\sum_{2^{-n}\gtrsim[x]}2^{n(2\frb+\gamma+1-\frs)}
\lesssim \eps[x]^{-2-\gamma},
\end{equ}
as required.
\end{proof}

\subsection{Convergence of the Robin kernels}\label{sec:robin}
The purpose of this section is to make the folklore fact
\[
\textit{``the $-\infty$ Robin boundary condition is the $0$ Dirichlet boundary condition''}
\]
precise in a form that suits the setup of \cite{GH17}.
The dimension plays no role here, but we stick to the $1+3$-dimensional setting that we will use later.
As a warm-up example, let us recall the construction of Robin heat kernels on the upper half space $\R^4_u=\{(x_0,x_1,x_2,x_3)\,:\,x_3>0\}$.
For $x\in\R^4$, $r\in\R$ denote $x^r=(x_0,x_1,x_2,-x_3-r)$. Then for any $a\in(-\infty,\infty)$ the function
\begin{equs}\label{eq:Robin half space}
\hat \cG_a(x,y)&=\cK(x-y)+\cK(x-y^0)-\int_0^\infty 2 a e^{-ar}\cK(x-y^r)\,dr
\end{equs}
is easily seen to satisfy, for each fixed $y\in\R^4_u$,
\begin{equs}
(\d_{x_0}-\Delta_x) \hat \cG_a(x,y)&=\delta_{x=y}&\quad&\text{on }\R^4_u;\\
\d_n \hat \cG_a(x,y)&=-\d_{x_3}\hat \cG_a(x,y)=-a\hat \cG_a(x,y)&\quad &\text{on }\d\R^4_u.
\end{equs}
Therefore $\hat \cG_a$ \dash more precisely, its product with the indicator of $(\R^4_u)^2$ \dash is indeed the Robin heat kernel.
Taking $a=\infty$ in \eqref{eq:Robin half space} the measure $ae^{-ar}\,dr$ becomes the Dirac mass at $0$, and we recover the Dirichlet heat kernel. Loosely speaking, to build the Robin kernels on the cube $D$, one needs to repeat the procedure of reflecting and averaging in \eqref{eq:Robin half space} for each face of the cube ad infinitum.

\begin{remark}\label{rem:uniform-in-a}
Note that for $a\in[0,\infty]$, the function $\cR_a(x,y)=\bone_{x,y\in\R^4_u}\int_0^\infty 2ae^{-ar}\cK(x-y^r)\,dr$ is twice a convex combination of the kernels
$\bone_{x,y\in\R^4_u}\cK(x-y^r)$ that each fit in the framework of Section \ref{sec:kernel computations} with $A_1=A_2=\R^4_u$ and $T(y)=y^r$ (hence also with $T(y)=y^0$). Therefore $\cR_a$ satisfy the bounds therein uniformly over $a\in[0,\infty]$.
\end{remark}

To formulate the result, recall the following concept of kernel remainders from \cite{GH17}.
For the present section it is more convenient to work on the cube $D=(-1,1)^3$ and recall that we denote by $\d$  the boundary of $\R\times D$.
\begin{definition}\label{def:Z}
Denote by $\scZ_{\beta,\d}$ the set of functions
$Z:(\R^4\setminus \d)^2 \to \R$ that can be written in the form
$
Z(z,z')=\sum_{n\geq0}Z_n(z,z')
$
where, for each $n$, $Z_n$ satisfies the following
\begin{claim}
\item $Z_n$ is supported on
$\{(z,z')=((t,x),(t',x')):\,|x|_{\d}+|x'|_{\d}+|t-t'|^{1/2}\leq 3 (2^{-n})\}$, 
where $C$ is a fixed constant depending only on the domain $D$.
\item For any multiindices $k$ and $\ell$ with $|k|,|\ell|\leq 2$,
\begin{equ}\label{eq:Z bound}
|D_1^kD_2^\ell Z_n(z,z')|\lesssim2^{n(\frs+|k+\ell|-\beta)},
\end{equ}
where the proportionality constant may depend on $k$ and $\ell$,
but not on $n$, $z$, $z'$.
\end{claim}
\end{definition}
Clearly $\scZ_{\beta,\d}$ is a vector space, on which
the best proportionality constant in \eqref{eq:Z bound}
defines a norm $\|\cdot\|_{\scZ_{\beta,\d}}$.
Let us decompose the $1+3$-dimensional heat kernel as $\cK=\bar\cK+R$ in such a way that
$\bar\cK=\cK$ on the ball of radius $1/2$ around the origin and  vanishes outside the ball of radius $1$.
Furthermore, $\bar\cK$ can be chosen to satisfy \cite[Ass.~5.1]{H0}
and $R$ is globally smooth with any derivatives having faster than polynomial spatial decay.

Whenever $a\in(-\infty,\infty)\setminus\{0\}$, the boundary conditions $\big(\frac{\d_n}{a}+1\big)f=0$ and $(\d_n+a)f=0$ are equivalent. For $a=0$ only the latter makes sense (and gives the homogeneous Neumann boundary conditions), while for $a=\infty$ only the former does (and gives the homogeneous Dirichlet). In the lemma below we use the former form, with the understanding of the obvious modification for $a=0$.

\begin{lemma}\label{lem:robin}
There exists a family of remainders
$(Z^{(a)})_{a\in(-\infty,\infty]}$ such that:
\begin{enumerate}[(i)]
\item $(\partial_{t}-\Delta)(\bar \cK+Z^{(a)})(t,x,t',y)=\delta_{t=t',x=y}$
on $\big([0,1]\times D\big)^2$;
\item $\big(\tfrac{\partial_{n}}{a}+1\big)(\bar\cK+Z^{(a)})(t,x,t',y)=0$
on $[0,1]\times \partial D\times[0,1]\times D$;
\item For all $\beta<2$, $(Z^{(a)})_{a\in(-\infty,\infty]}$ is continuous with respect to the natural topology of $(-\infty,\infty]$, as a function with values in $\scZ_{\beta,\d}$.
\end{enumerate}
\end{lemma}

We then denote the Robin heat kernels on $D$ by
\begin{equ}\label{eq:robin-HK-def}
\cG_a=\bar\cK+Z^{(a)}.
\end{equ}

\begin{proof}
Let us preface that since the continuity property is only easier in $(-\infty,\infty)$, we will only deal with it at the endpoint case $a=\infty$.
Denote by $S$ the group of transformations of $\R^3$
generated by the reflections $g_{\pm i}$ on the hyperplanes
$\R^{i-1}\times\{\pm 1\}\times\R^{3-i}$, $i=1,2,3$,
and the different pieces of the boundary by
$D^{\pm i}:=(-1,1)^{i-1}\times\{\pm1\}\times(-1,1)^{3-i}$. 
Also let $e_{\pm i}$ denote the outward normal vector on the boundary piece $D^{\pm i}$. 

Take $a\in[1,\infty]$ and a function $F:(\R^4)^2\rightarrow\R$
with all derivatives having faster than polynomial spatial decay,
which furthermore has a `sign' $b\in\{\pm1\}^3$ with the property
$\partial_{x_i}F(t,x,t',y)=b_i\partial_{y_i}F(t,x,t',y)$.
Let us denote by $T^a_{\pm i}F$ the function
\begin{equ}
(T^a_{\pm i}F)(t,x,t',y)
=\int_0^\infty b_i(-1+2e^{-as})
\partial_{e_{\pm i}}F(t,x,t',g_{\pm i}(y)+se_{\pm i})\,ds.
\end{equ}
Here and below the partial derivative $\partial_{e_{\pm i}}$
is understood to act onto the $x$ coordinate.
Note that $T^a_{\pm i}F$ also has the above mentioned properties,
with its `sign' switched in the $i$-th coordinate.
The construction is such that one has
\begin{equ}\label{eq:robin reflection}
\big(\tfrac{\partial_{\pm e_i}}{a}+1\big)(F+T^a_{\pm i} F)
=0\quad\text{on }\R\times D^{\pm i}\times \R\times \R^3.
\end{equ}
Notice also that for $i\neq j$, one has
$T_{\pm i}T_{\pm j}=T_{\pm j}T_{\pm i}$,
and so for any $g\in S$ for any minimal
(with respect to the length) representation
$g=h_1h_2\cdots h_n$, where $h_k=g_{\pm i}$,
the corresponding mapping $T^a_g = T^a_{h_1}\cdots T^a_{h_n}$
is well-defined.
We also write $T^a_{\text{id}}=\text{id}$.

Let $A$ be the set of elements of $S$  whose minimal representation contains at most
one of $g_{-i}$ and $g_{+i}$ for all $i$.
Consider
\begin{equ}
G^{1,a}:=\sum_{g\in A} T^a_{g}\bar \cK.
\end{equ}
Note that for any fixed $\pm i$, all $g\in A$ are of one of three types:
not containing either $g_{-i}$ or $g_{+i}$ for all $i$,
of the form $g_{\pm i}g$ with $g$ of the previous type,
or containing $g_{\mp i}$.
Since for those of the last type, $(T^a_g\bar \cK)(t,\cdot,t',y)$ vanishes
near $D^{\pm i}$ whenever $y\in D$,
one can `pair up' elements of the first two types,
and conclude by \eqref{eq:robin reflection} that
\begin{equ}
\big(\tfrac{\partial_{\pm e_i}}{a}+1\big)G^{1,a}
=0\quad\text{on }\R\times D^{\pm i}\times \R\times D.
\end{equ}
Fix $\beta<2$, for convenience and without loss of generality we also assume $\beta\geq1$.
Let us use the notation $f\sim \bar f$ for functions $f$, $\bar f$
on $(\R^4)^2$ whenever $f=\bar f$ on $([0,1]\times D)^2$.
Note for instance, one has $T_g^a\bar\cK\sim 0$ for all $g\in S\setminus A$.
We then claim that for all $g\in A\setminus\{\text{id}\}$,
there exists a $Z^a_g\in\scZ_{\beta,\d}$ such that $T^a_g\bar \cK\sim Z^a_g$,
and $\|Z^a_g-Z^\infty_g\|_{\scZ_{\beta,\d}}\rightarrow 0$
as $a\rightarrow\infty$.
For convenience let us illustrate the argument for $g=g_{+i}$:
first note that if $x,y\in D$ are such that
$(t,x,t',g_{+i}(y)+se_{+i})\in\text{supp }\bar \cK_n$
for some $t,t',s\geq 0$, then
$d(x,\partial D)\vee d(y,\partial D)\leq 2^{-n-1}$.
Take now smooth functions $\varphi_n$ on $(\R^3)^2$, which are $1$ on
$\{(x,y):d(x,\partial D)\vee d(y,\partial D)\leq 2^{-n-1}\}$,
supported on $\{(x,y):d(x,\partial D)\vee d(y,\partial D)\leq 2^{-n}\}$,
and for all multiindices $k$ and $\ell$, $D_1^kD_2^\ell\varphi$ is bounded by
$2^{n(|k+\ell|)}$, up to a constant uniform in $n$. One then has, on $([0,1]\times D)^2$,
\begin{equs}
T_{g_{+i}}^a\bar \cK&
=\sum_{n\geq 0}\int_0^\infty-(-1+2e^{-as})
\partial_{e_{+ i}}\bar \cK_n(t,x,t',g_{+i}(y)+se_{+ i})\,ds
\\
&\sim\sum_{n\geq 0}\varphi_n(x,y)\int_0^\infty-(-1+2e^{-as})
\partial_{e_{+ i}}\bar \cK_n(t,x,t',g_{+i}(y)+se_{+ i})\,ds
\\
&=:Z^a_{g_{+i}}=Z^\infty_{g_{+i}}+(Z^a_{g_{+i}}-Z^\infty_{g_{+i}}).
\end{equs}
Noticing that the values of $s$ with non-zero contribution to the integral above are of size at most 
$\cO(2^{-n})$, one has
\begin{equs}
2^{-n(\frs+|k+\ell|-\beta)}&|D_1^kD_2^\ell(Z^a_{g_{+i}}-Z^\infty_{g_{+i}})_n|\lesssim
2^{n(\beta-1)}\int_0^{C2^{-n}}e^{-as}\,ds \\
&\le 2^{n(\beta-1)}(2^{-n} \wedge a^{-1}) \label{eq:Zn}
\le a^{\beta-2}\;,
\end{equs}
which indeed converges to $0$
uniformly in $n$, as $a\rightarrow \infty$, yielding our claim.

Consider next 
\begin{equ}
G^{2,a}:=\varphi_{0}\sum_{g\in S}T^{a}_g R.
\end{equ}
Thanks to the spatial decay properties of $R$,
this sum converges as a smooth function on $([0,1]\times D)^2$,
uniformly in $a\in[0,\infty]$.
In particular, this is enough to infer from \eqref{eq:robin reflection}
that it also satisfies the boundary condition, and also that since for each $g$,
$\|\varphi_{-1}T^{a}_g R-\varphi_{-1}T^{\infty}_g R\|_{\scZ_{\beta,\d}}\rightarrow 0$,
one has $\|G^{2,a}-G^{2,\infty}\|_{\scZ_{\beta,\d}}\rightarrow 0$.

Since clearly $(\partial_{t}-\Delta)(T_g^a(\bar \cK+R))=0$ on
$([0,1]\times D)^2$ for $g\neq \text{id}$, setting
\begin{equ}
Z^{(a)}:=G^{2,a}+\sum_{g\in A\setminus\{\text{\scriptsize{id}}\}}Z^a_g,
\end{equ}
completes the proof of the lemma.
\end{proof}

\begin{remark}
As can be seen from the explicit expression \eqref{eq:Robin half space}, the limit $a\to -\infty$ does not exist,
contrarily to what (ii) may suggest.
\end{remark}

\begin{remark}
In the special cases $a=0,\infty$ the above construction coincides with the one in \cite[Ex.~4.15]{GH17} for the homogeneous Neumann and Dirichlet heat kernels, respectively.
\end{remark}

\section{Explicit boundary corrections}\label{sec:calculations}
In this section we perform the boundary renormalisation of expectations of some concrete stochastic objects.
Recall that in the simplest situation of translation invariant renormalisation without subdivergences 
one considers a sequence of random distributions $X_\eps$, where $\E X_\eps=C_\eps$ diverges but does \emph{not} 
depend on the space-time variable.
In the case with boundaries the prototypical situation will be
\begin{equ}\label{eq:decomposition}
\E X_\eps=C_\eps+c_\eps\delta_\partial+\bar R^1_\eps+\bar R^2_\eps+\bar R^3_\eps,
\end{equ}
with divergent $C_\eps$ and $c_\eps$ and three different types of remainder distributions $\bar R^i_\eps$
that each converge to a finite limit.

\subsection{PAM}\label{sec:PAM square}
We denote $\Psi_\eps=\nabla Y_\eps$, where $Y_\eps$ is defined in \eqref{eq:Y}.
Let $K(x)=\tfrac{1}{4\pi|x|}$ be the Green's function of the $3$-dimensional Poisson equation
and fix a compactly supported function $\bar K$ such that $K-\bar K$ is smooth and vanishes
in a neighbourhood of the origin.
Set
\begin{equ}\label{eq:RC def 1}
\ell_\eps(\<PAM-2a-Small>)
=\E|\nabla\bar K\ast\xi_\eps|^2.
\end{equ}
As mentioned in Remark \ref{rem:reno-PAM}, the quantity $\ell_\eps(\<PAM-2a-Small>)$ arises from the BPHZ renormalisation of the regularity structure associated to \eqref{e:PAM}.
We will give more details in Section \ref{sec:reg str} below, but the reader may freely take 
\eqref{eq:RC def 1} as a definition for now.

We will show a decomposition of the type \eqref{eq:decomposition}
on the tetrahedron $Q$.		
\begin{lemma}\label{lem:reno Psi^2 PAM}
For all $\eps>0$, on $Q$ one has the decomposition
\begin{equ}
\E|\Psi_\eps|^2=\ell_\eps(\<PAM-2a-Small>)+\big(a_\rho+\tfrac{|{\log \eps}|}{8\pi}\big)\delta_\partial+\delta_\d  R^1_\eps+\scR R^2_\eps+R^3_\eps,
\end{equ}
where $a_\rho$ is a constant, and the remainders satisfy:
\begin{enumerate}[(i)]
\item $R^1_\eps\to R^1_0$ in $\CC^{1-\kappa,-\kappa}_{\d^2}(Q_\d)$;\label{i}
\item $R^2_\eps\to R^2_0$ in $\CC^{1-\kappa,-1-\kappa}_{\d }(Q)$ and $D_{x_1}R^2_\eps=D_{x_2}R^2_\eps=0$;\label{ii}
\item $R^3_\eps\to R^3_0$ in $\CC^{1-\kappa,-1-\kappa}_{\d^2}(Q)$;\label{iii}
\end{enumerate} 
where $\kappa>0$ is arbitrarily small and the limits do not depend on $\rho$.
\end{lemma}
\begin{proof}
Introduce the shorthand, for functions $f$ on $(\R^3)^2$, $\eps\in[0,1]$, and
$S\subset\R^3$,
\begin{equ}
(\frD^\eps_S f)(x,x'):=\big(\rho_\eps\ast(\nabla_xf(x,\cdot)\bone_{S}(\cdot))\big)(x')
\end{equ}
with the convention that for $\eps=0$ we replace the convolution
with $\rho_\eps$ by the identity.
With this notation we have 
\begin{equ}
\ell_\eps(\<PAM-2a-Small>) = \int |\frD^\eps_{\R^3}\bar K(x,x')|^2\, dx',
\end{equ}
which of course does not actually depend on $x$ since $\bar K$ depends only on the difference of its arguments.
One can then write, for $\eps>0$,
\begin{equ}
\E\Psi_\eps^2(x)-\ell_\eps(\<PAM-2a-Small>)=\int \bigl(|\frD^\eps_D G(x,x')|^2-|\frD^\eps_{\R^3}\bar K(x,x')|^2\bigr)\, dx'.
\end{equ}
As a first step, we truncate the infinite sum in $G$ and remove the truncation of $K$.
Let $B=\{\pm1\}^3$, and for $b\in B$ and $x\in\R^3$, let $x^{b}$ be
the vector obtained by switching the signs of the coordinates of $x$
according to $b$. Denote furthermore $K^b(x,\bar x):=K(x,\bar x^b)$.
We then write
\begin{equ}\label{eq:sim1}
\E\Psi_\eps^2(x)-\ell_\eps(\<PAM-2a-Small>)-R^{3,1}_\eps(x)= \int\Big(\Big|\frD^\eps_D \Big(\sum_{b\in B}K^b\Big)(x,x')\Big|^2-|\frD^\eps_{\R^3} K(x,x')|^2\Big)\, dx',
\end{equ}
interpreting this as a definition of $R^{3,1}_\eps$.
It follows from the reflection principle that, on $Q$, $R^{3,1}_\eps$ converges 
as a smooth function to a $\rho$-independent limit.
Next we claim that with $\R^3_u=\{(x_1,x_2,x_3):\,x_3>0\}$,
\begin{equ}
R^{3,2}_\eps(x)=
\int \Big|\frD^\eps_D \Big(\sum_{b\in B}K^b\Big)(x,x')\Big|^2	-
\big|\frD^\eps_{\R^3_u} 
\big(K^{(1,1,1)}+K^{(1,1,-1)}\big)(x,x')\big|^2\,dx'
\end{equ}
converges in $\CC^{1-\kappa,-1-\kappa}_{\d^2}(Q)$ to a $\rho$-independent limit.
The function  $R^{3,2}_\eps$ can be written as a finite linear combination of terms of the type
\begin{equs}
\tilde R_\eps(x):=\int\big(\frD_{S}^\eps K^b\cdot\frD_{\hat S}^\eps K^{\hat b}\big)(x,x')\,dx'\;,
\end{equs}
where either $\hat b\in B\setminus\{(1,1,1),(1,1,-1)\}$ and $\hat S=D$, or $\hat b\in \{(1,1,1),(1,1,-1)\}$ and $\hat S=\R^3_u\setminus D$. The choice of $b$ and $S$ will not play a role.
Notice that we are in the setting of Section \ref{sec:prepare}: 
$\nabla_x K^b(x,y)\bone_{y\in S}\in\scG$, with $\frs=3$, $\frb=2$, $T(z)=z^b$, $A_x=\R^3$, $A_y=S$.
Moreover, for $x\in Q$, one has $d(x,\hat S^{\hat b})\geq x_1=|x|_{\d^2}$ for each choice of $\hat b$ and $\hat S$ as above,
and therefore Lemma \ref{lem:kernels1} provides a bound for $\tilde R_\eps$ in $\CC^{1-\kappa,-1}_{\d^2}(Q)$, uniformly in $\eps$.
Since on $Q\setminus \d^2$, $\tilde R^\eps$ converges locally in $\CC^{1-\kappa}$, this proves the convergence in $\CC^{1-\kappa,-1-\kappa}_{\d^2}(Q)$ for each $\tilde R_\eps$, and consequently for $R^{3,2}_\eps$ as well.
The function $R^3_\eps=R^{3,1}_\eps+R^{3,2}_\eps$ therefore satisfies \eqref{iii} and we have so far proved the following decomposition on $Q$:
\begin{equ}
\E\Psi_\eps^2(x)-\ell_\eps(\<PAM-2a-Small>)-R^{3}_\eps(x)= \int|\frD^\eps_{\R^3_u} 
\Big(K^{(1,1,1)}+K^{(1,1,-1)}\Big)(x,x')|^2-|\frD^\eps_{\R^3} K(x,x')|^2\, dx'.
\end{equ}
Define $R^2_\eps$ as the right-hand side of the above equality.
It is clear that $R^2_\eps$ does not depend on $x_1$ and $x_2$, so it remains to check its convergence in $\CC^{1-\kappa,-1-\kappa}_\d(Q)$.
By Lemma~\ref{lem:kernels3}, both of the functions
\begin{equs}
\,&\int \big|\frD^\eps_{\R^3_u} K^{(1,1,1)}(x,x')\big|^2-|\frD^\eps_{\R^3} K(x,x')|^2\bone_{x'\in\R^3_u}\,dx',
\\
&\int \big|\frD^\eps_{\R^3_u} K^{(1,1,-1)}(x,x')\big|^2-|\frD^\eps_{\R^3} K(x,x')|^2\bone_{x'\in(\R^3\setminus \R^3_u)}\,dx',
\end{equs}
converge to $0$ in $\CC^{1-\kappa,-1-\kappa}_\d(Q)$.
Concerning the cross term in $R^2_\eps$, its convergence to
\begin{equ}
R^2_0(x)=2\int \bone_{x'\in \R^3_u}\big(\nabla_1 K\cdot\nabla_1 K^{(1,1,-1)}\big)(x,x')\,dx'
\end{equ}
follows as above: Lemma~\ref{lem:kernels1} yields a uniform bound in $\CC^{1-\kappa,-1}_{\d}(Q)$, and away from  $\d$ the convergence in $\CC^{1-\kappa}$ is quite clear.
Therefore, $R^2_\eps$ satisfies \eqref{ii}.
It remains to show that the difference $R^2_\eps-\scR R^2_\eps$ is of the claimed form.

As noted in Remark \ref{rem:mini}, this difference is of the form $\delta_\d  m_\eps$, and on $Q_\delta$ one can express the function $m_\eps$ by
\begin{equ}
m_\eps(y)=\int_0^{y_1}R^2_\eps(y_1,y_2,s)\,ds=\int_0^{y_1}R^2_\eps(0,0,s)\,ds.
\end{equ}
Let us use the shorthand $R^2_\eps(0,0,s)=I_\eps(s)$,
for which we have the bound $\eps^{-1}$ from \eqref{eq:triv eps bound}.
One can rewrite the above integral as
\begin{equ}
m_\eps(y)=\int_0^\eps I_\eps(s)\,ds+\int_\eps^\infty I_\eps(s)-I_0(s)\,ds-\int_{y_1}^\infty I_\eps(s)-I_0(s)\,ds+\int_\eps^{y_1}I_0(s)\,ds.
\end{equ}
By Lemma \ref{lem:kernels2} and \ref{lem:kernels3}, one has the bound
$|I_\eps(s)-I_0(s)|\lesssim\eps^{1-\kappa}/s^{2-\kappa}$.
Therefore, the second term above is finite and by scaling invariance,
is independent of $\eps$, so it is just a ($\rho$-dependent) constant. 
The first term is also independent of $\eps$, also by scaling invariance.
Therefore,
\begin{equ}
m_\eps(y)=a_\rho+\int_{y_1}^\infty I_\eps(s)-I_0(s)\,ds+\int_\eps^{y_1}I_0(s)\,ds.
\end{equ}
Denote the second term by $R^{1,1}_\eps$.
Invoking Lemma \ref{lem:kernels2} again, we have
$|R^{1,1}_\eps(y)|\leq  \eps^{1-\kappa}/y_1^{1-\kappa}$,
and by \eqref{eq:kernels1-a} we have $\nabla_y R^{1,1}_\eps(y)\leq 1/y_1$. This is enough to conclude $R^{1,1}_\eps\to 0$ in $\CC^{1-\kappa,-\kappa}_{\d^2}(Q_\d)$.
Moving on to the third term on the right-hand side 
we write, with $\us=(0,0,s)$,
\begin{equs}
I_0(s)
&=\int_{\R^3}\big(\nabla_1 K\cdot\nabla_1 K^{(1,1,-1)}\big)(\us,x')\,dx'
\\
&=\int_{\R^3}\big((-\d_{x_1'},-\d_{x_2'},-\d_{x_3'})K\cdot(-\d_{x_1'},-\d_{x_2'},\d_{x_3'})K^{(1,1,-1)}\big)(\us,x')\,dx'
\\
&=-\int_{\R^3}\big((\Delta K)K^{(1,1,-1)}\big)(\us,x')\,dx'
-2\int_{\R^3}\big(\d_{x_3'}K\d_{x_3'}K^{(1,1,-1)}\big)(\us,x')\,dx'
\\
&=:I_0^1(s)+\frac{1}{16\pi^2}I_0^2(s).
\end{equs}
Since
$-\Delta K=\delta_0$, one easily gets
\begin{equ}
I_0^1(s)=\int_{\R^3}\delta_{x'=\us}K^{(1,1,-1)}(\us,x')\,dx'=\frac{1}{8\pi s}.
\end{equ}
Next we rewrite $I_0^2(s)$ by change of variables: first by setting $x_i=x_i'/s$ and then $\bar x_3=1/x_3$, $\bar x_1=x_1/x_3,$ $\bar x_2=x_2/x_3$, one gets
\begin{equs}
I_0^2(s)&=\int_{\R^3_u}\frac{(s-x_3')(s+x_3')}{|(x_1',x_2',x_3'-s)|^3|(x_1',x_2',x_3'+s)|^3}\,dx'
\\
&=\frac{1}{s}\int_{\R^3_u}\frac{1-x_3^2}{|(x_1,x_2,x_3-1)|^3|(x_1,x_2,x_3+1)|^3}\,dx
\\
&=\frac{1}{s}\int_{\R^3_u}\frac{\bar x_3^2-1}{|(\bar x_1,\bar x_2,\bar x_3-1)|^3|(\bar x_1,\bar x_2,\bar x_3+1)|^3}\,d\bar x=-I_0^2(s),
\end{equs}
and so $I_0^2(s)=0$.
We can conclude that
\begin{equ}
m^\eps(y)= a_\rho+R^{1,1}_\eps(y)+\tfrac{\log y_1}{8\pi}-\tfrac{\log \eps}{8\pi}\;,
\end{equ}
and, setting $R^1_\eps(y)=R^{1,1}_\eps(y)+\tfrac{\log y_1}{8\pi}$, this completes the proof.
\end{proof}

\subsection{\texorpdfstring{$\Phi^4_3$}{Phi\^4\_3} - the quadratic term}\label{sec:Phi4 square}
Let $\cK$ be the heat kernel on $\R\times\R^3$
and fix a compactly supported function $\bar\cK$ such that $\cK-\bar\cK$ is smooth and vanishes in a neighborhood of the origin. We then define
\begin{equ}
\ell_\eps(\<Phi-2-Small>)
=\E(\bar \cK\ast\xi_\eps)^2.
\end{equ}

As before, we will show a decomposition of the type \eqref{eq:decomposition}
on $Q$. 

\begin{lemma}\label{lem:reno Phi43 square}
Let $(b_\eps)_{\eps\in(0,1]}\subset \R$ be a sequence such that $\eps b_\eps\to 0$ and
\begin{equ}
\lim_{\eps\to 0}\Big(\frac{|{\log \eps}|}{32\pi}-b_\eps \Big)
=b\in[0,\infty].
\end{equ}
Then there exist a sequence $(c_\eps)_{\eps\in(0,1]}$ such that $c_\eps\to b$ and 
such that on $Q$ one has the decomposition
\begin{equ}
\E|\Psi_{\eps,c_\eps}|^2=\ell_\eps(\<Phi-2-Small>)+\big(a_\rho+b_\eps+c_\eps\big)\delta_\partial+\delta_\d R^1_\eps+\scR R^2_\eps+R^3_\eps,
\end{equ}
where $a_\rho$ is a constant and the remainders $R^i_\eps$ satisfy the properties in Lemma \ref{lem:reno Psi^2 PAM} \eqref{i}-\eqref{iii}.
\end{lemma}
\begin{remark}
While for  $b<\infty$ one may take $c_\eps\equiv b$, for $b=\infty$
the sequence $c_\eps$ is \emph{not} obtained in the trivial way $c_\eps=\tfrac{|{\log \eps}|}{32\pi}-b_\eps$. For example, when $b_\eps\equiv 0$, the difference $ c_\eps - \tfrac{|{\log \eps}|}{32\pi}$ should actually be chosen to diverge at order
$\log|{\log \eps}|$. This is left as an exercise to the interested reader.
\end{remark}
\begin{proof}
Let us first take an arbitrary sequence $\tilde c_\eps\to b$ such that $\eps\tilde c_\eps\to 0$.
The first part of the argument is then virtually identical to that in the proof of Lemma \ref{lem:reno Psi^2 PAM}.
By following the same steps, we can conclude that on $Q$ one has
\begin{equ}
\E\Psi_{\eps,\tilde c_\eps}^2-\ell_\eps(\<Phi-2-Small>)-R^3_\eps-\scR R^2_\eps=\delta_\d m_\eps
\end{equ}
with $R^2_\eps$ and $R^3_\eps$ satisfying \eqref{ii} and \eqref{iii}, respectively.
It is also clear that $R^2_0$ and $R^3_0$ do not depend on $\tilde c_\eps$ but only on $b$
(since they can be expressed from heat kernels for the $-3b$-Robin boundary condition).
The function
$m_\eps$ is given on $Q_\d$ by
\begin{equ}\label{eq:m-eps 1}
m_\eps(y)=\int_0^{y_1}
\int\big|\big(\rho_\eps\ast\hat\cG_{3\tilde c_\eps}(\us,\cdot)\big)(x)\big|^2
-\big|\big(\rho_\eps\ast\cK(\us,\cdot)\big)(x)\big|^2
\,dx\,ds,
\end{equ}
where $\us=(0,0,0,s)$ and $\hat\cG_{3\tilde c_\eps}$ is the Robin heat kernel on the upper half space
$\R^4_u=\{(x_0,x_1,x_2,x_3):\,x_3>0\}$.
Recall from \eqref{eq:Robin half space} that
it is given, with the notation $z^r=(z_0,z_1,z_2,-z_3-r)$, by
\begin{equs}
\hat \cG_a(z,\tilde z)&= \bone_{\tilde z_3\geq 0}\cK(z-\tilde z)+\bone_{\tilde z_3\geq 0}\cK(z-\tilde z^0)-\cR_a(z,\tilde z)
\\
&=\bone_{\tilde z_3\geq 0}\cK(z-\tilde z)+\bone_{\tilde z_3\geq 0}\cK(z-\tilde z^0)-\bone_{\tilde z_3\geq 0}\int_0^\infty 2 a e^{-ar}\cK(z-\tilde z^r)\,dr.
\end{equs}
Now we would like to proceed similarly to Lemma \ref{lem:reno Psi^2 PAM} by simplifying the integral in \eqref{eq:m-eps 1}.
However, since the kernels $\tilde\cG_{3\tilde c_\eps}$ themselves depend on $\eps$, some of the scaling arguments break down.
Therefore let us separate the $\eps$-dependent part from the kernel.
Let $\cK^\eps(z,\tilde z)=\big(\rho_\eps\ast(\cK(z-\cdot))\big)(\tilde z)$,
$\cK^\eps_+(z,\tilde z)=\big(\rho_\eps\ast(\bone_{\cdot\in\R^4_u}\cK(z-\cdot))\big)(\tilde z)$,
$\cK^\eps_-(z,\tilde z)=\big(\rho_\eps\ast(\bone_{\cdot\in\R^4_u}\cK(z-(\cdot)^0))\big)(\tilde z)$,
and $\cR_a^\eps(z,\tilde z)=\big(\rho_\eps\ast(\cR_a(z,\cdot))\big)(\tilde z)$.
and define
\begin{equ}\label{eq:cJ 0}
\cJ^\eps_a(s)=\int_{\R^4_u}
-2\cK^\eps_+(\us,z)\cR_{3a}^\eps(\us,z)
-2\cK^\eps_-(\us,z)\cR_{3a}^\eps(\us,z)
+|\cR_{3a}^\eps(\us,z)|^2\,dz.
\end{equ}
We then have
\begin{equ}
m_\eps(y)=
\int_0^{y_1}\int|(\cK^\eps_+ + \cK^\eps_-)(\us,z)|^2-|\cK^\eps(\us,z)|^2\,dz\,ds
+\int_0^{y_1}\cJ_{\tilde c_\eps}^\eps(s)\,ds.
\end{equ}
Now the first integral can be treated by the scaling arguments as in Lemma \ref{lem:reno Psi^2 PAM}.
That is, there exists a constant $a_\rho$ and a sequence of functions $R^{1,1}_\eps$ converging to $0$ in $\CC^{1-\kappa,-\kappa}_{\d^2}(Q_\d)$ such that
\begin{equ}
m_\eps(y)=a_\rho+R^{1,1}_\eps(y)+\int_\eps^{y_1}\cI_0(s)\,ds+\int_0^{y_1}\cJ_{\tilde c_\eps}^\eps(s)\,ds,
\end{equ}
where
the function $\cI_0$ is given by
\begin{equ}
\cI_0(s)=\int_{\R^4_u}2\cK(\us-z)\cK(\us-z^0)\,dz.
\end{equ}
By scaling invariance once again, we have $\cI_0(s)=\tfrac{1}{s}\cI_0(1)$.
The value of $\cI_0(1)$ can be found by an explicit computation, which we perform in a more general setting below,
see \eqref{eq:big integral}, which yields the value 
\begin{equ}\label{eq:I0-value}
\cI_0(1)=\frac{1}{32\pi}.
\end{equ}
Therefore,
\begin{equ}\label{eq:m-eps 2}
m_\eps(y)= a_\rho+\tfrac{|{\log \eps}|}{32\pi}+R^{1,1}_\eps(y)+\tfrac{\log y_1}{32\pi}+\int_0^{y_1}\cJ^\eps_{\tilde c_\eps}(s)\,ds.
\end{equ}
Clearly, $R^{1,2}_\eps(y):=\tfrac{\log y_1}{32\pi}$ satisfies \eqref{i}. Moving on the last term,
we rewrite it as
\begin{equ}\label{eq:integrals}
\int_0^{\eps}\cJ^\eps_{\tilde c_\eps}(s)\,ds
+\int_\eps^{y_1}\big(\cJ^\eps_{\tilde c_\eps}(s)-\cJ_{\tilde c_\eps}^0(s)\big)\,ds
-\int_{y_1}^{1}\cJ^0_{\tilde c_\eps}(s)\,ds
+\int_{\eps}^1\cJ^0_{\tilde c_\eps}(s)\,ds.
\end{equ}
By the scaling relation $\cR^\eps_a(z,\tilde z)=\cR^{\lambda\eps}_{\lambda^{-1}a}(\lambda z,\lambda\tilde z)\lambda^3$, we have
$\cJ^\eps_a(s)=\cJ^{\lambda\eps}_{\lambda^{-1}a}(\lambda s)\lambda$
and therefore
\begin{equ}\label{eq:cJ 1}
\int_0^\eps\cJ_{\tilde c_\eps}^\eps(s)\,ds =
\int_0^1\cJ^1_{\eps\tilde c_\eps}(s)\,ds\;.
\end{equ}
Denote
$\cK^\eps_r(z,\tilde z)=\big(\rho_\eps\ast(\bone_{\cdot\in\R^4_u}\cK(z-(\cdot)^r))\big)(\tilde z)$,
which is just the generalisation of $\cK^\eps_-=\cK^\eps_0$.
From \eqref{eq:kernels1-a} one has the bounds
\begin{equ}
\Big|\int\cK^1_+(\us,z)\cK^1_r(\us,z)\,dz\Big|\, ,\,
\Big|\int\cK^1_{r'}(\us,z)\cK^1_r(\us,z)\,dz\Big|
\lesssim 1\wedge r^{-1}
\end{equ}
uniformly in $s,r,r'\geq 0$.
Since $\eps\tilde c_\eps\to0$, this implies $\int_0^1\cJ^1_{\eps\tilde c_\eps}(s)\,ds\to0$.
Denote the second integral in \eqref{eq:integrals} by $R^{1,3}_\eps(y)$.
From Lemma \ref{lem:kernels2}
one has the bound $|\cJ^\eps_a(s)-\cJ^0_a(s)|\lesssim \eps^{1-\kappa}/ s^{2-\kappa}$,
uniformly in $a$
(see Remark \ref{rem:uniform-in-a}).
Therefore, if similarly to \eqref{eq:cJ 1} we write
\begin{equ}
R^{1,3}_\eps(y)=\int_{\eps}^{y_1} \big(\cJ^\eps_{\tilde c_\eps}(s)-\cJ^0_{\tilde c_\eps}(s)\big)\,ds=
\int_{1}^{\eps^{-1}y_1} \big(\cJ^1_{\eps\tilde c_\eps}(s)-\cJ^0_{\eps\tilde c_\eps}(s)\big)\,ds,
\end{equ}
then by the dominated convergence theorem, the right-hand side goes to $0$. Moreover, one has the bound $D_{y_1} R^{1,3}(y)\lesssim 1/y_1$ from \eqref{eq:kernels1-a}. Therefore $R^{1,3}_\eps\to0$ in $\CC^{1-\kappa,-\kappa}_{\d^2}(Q_\d)$.
Next, denote the third integral in \eqref{eq:integrals} by $R^{1,4}_\eps(y)$.
Its convergence away from $\{y_1=0\}$ is clear, and therefore a uniform in $\eps$ bound in $\CC^{1,0}_{\d^2}(Q_\d)$ suffices to conclude the property \eqref{i}.
This is quite immediate: $R^{1,4}_\eps$ vanishes on the hyperplane $\{y_1=1\}$, and one has the bound $D_{y_1}R^{1,3}_\eps(y)\lesssim 1/y_1$ from \eqref{eq:kernels1-a} as before.

So $R^1_\eps=R^{1,1}_\eps+R^{1,2}_\eps+R^{1,3}_\eps+ R^{1,4}_\eps$ satisfies \eqref{i}, and one can rewrite \eqref{eq:m-eps 2} as
\begin{equs}
m_\eps(y)&= a_\rho+\tfrac{|{\log \eps}|}{32\pi}+R^{1}_\eps(y)+\int_\eps^{1}\cJ^0_{\tilde c_\eps}(s)\,ds
\\
&=a_\rho+\tfrac{|{\log \eps}|}{32\pi}+R^{1}_\eps(y)+\int_\eps^{1}\frac{1}{s}\cJ^0_{s\tilde c_\eps}(1)\,ds.
\end{equs}
We are therefore finished as soon as we show that there exist solutions $c_\eps,d_\eps$ to
\begin{equ}
b_\eps+c_\eps+d_\eps=\tfrac{|{\log \eps}|}{32\pi}+\int_\eps^{1}\frac{1}{s}\cJ^0_{sc_\eps}(1)\,ds
\end{equ}
that furthermore satisfy $c_\eps\to b$, $\eps c_\eps\to 0$, and $d_\eps\to d_0$ for some finite and $\rho$-independent $d_0$.
We claim the following properties of the function $\cJ^0_a(1)$, whose proof we postpone so that the present proof can be concluded.
%

\begin{proposition}\label{prop:bound for cJ}
The function $a \mapsto \cJ_a^0(1)$ is continuous on $(0,\infty)$,
$\lim_{a\to\infty}\cJ_a^0(1)=-2\cI_0(1)$,
and the bound
\begin{equs}
|\cJ_a^0(1)|\lesssim  a|{\log a}|\label{eq:bound for cJ}
\end{equs}
holds for $a\in(0,1/3]$.
\end{proposition}
When $b<\infty$, we simply choose $c_\eps=b$, and so 
\begin{equ}
d_\eps\to d_0=\int_0^1\frac{1}{s}\cJ_{sb}^0(1)\,ds
\end{equ}
which by Proposition \ref{prop:bound for cJ} is finite.
In the case $b=\infty$, first choose a $K>0$ such that $\cJ_a^0(1)\in[-3\cI_0(1), -\cI_0(1)]=\big[-\tfrac{3}{32\pi},-\tfrac{1}{32\pi}\big]$ for $a\geq K$, which is possible thanks to Proposition \ref{prop:bound for cJ}.
Define the map, for $c\geq 1 $,
\begin{equ}
f(c)=c-\int_{Kc^{-1}}^1\frac{1}{s}\cJ^0_{s c}(1)\,ds.
\end{equ}
Clearly $f$ is continuous and $f(c)\geq c$, in fact one has the bounds $c+\lambda\log c\leq f(c)\leq c+\lambda^{-1}\log c$
with some $\lambda>0$ for large enough $c$.
Therefore, there exists a function $\bar f$ so that $f(\bar f(c))=c$ for all sufficiently large $c$.
We then set
$c_\eps=\bar f\big(\tfrac{|{\log \eps}|}{32\pi}-b_\eps\big).$
Clearly $c_\eps\to\infty$, from $c_\eps\leq\tfrac{|{\log \eps}|}{32\pi}-b_\eps$ we have $\eps c_\eps\to0$, and
\begin{equ}
d_0=\lim_{\eps\to0}\int_\eps^{Kc_\eps^{-1}}\frac{1}{s}\cJ^0_{sc_\eps}(1)\,ds=0,
\end{equ}
using Proposition \ref{prop:bound for cJ} once more.
\end{proof}

\begin{proof}[Proof of Proposition \ref{prop:bound for cJ}]
For this proof we denote points in $\R^4$ as $z=(t,x_1,x_2,x_3)$.
The first two claims of the proposition are obvious, the third one requires some calculation.
Denote
\begin{equ}
N_t(x)=\frac{\one_{t>0}}{\sqrt{\pi t}}\exp\Big(-\frac{x^2}{t}\Big)
\end{equ}
and note the identities
\begin{equ}
N_t(x)=\lambda N_{\lambda^{2}t}(\lambda x),\qquad N_t(x)N_t(y)=N_{2t}(x+y)N_{t/2}((x-y)/2).
\end{equ}
We will also use the complementary error function
$\Ercf(s)=2\int_s^\infty N_1(x)\,dx$.
With these notation one has the following identity for $a\geq 0$:
\begin{equ}\label{eq:random identity}
\int_0^\infty \Ercf\Big(\frac{1}{\sqrt{t}}\Big)N_t(a)t^{-1}\,dt=\frac{2\tan^{-1}(a)}{\pi a}.
\end{equ}
Here we use the standard branch of $\tan^{-1}$, that is, $\tan^{-1}(0)=0$, and so $\frac{\tan^{-1}(a)}{a}\to 1$ as $a\to 0$.
To see \eqref{eq:random identity}, denote $\pi$ times the left-hand side by $S(a)$. One then sees that
\begin{equ}
S'(a)=\int_0^\infty\sqrt{\pi}\Ercf\Big(\frac{1}{\sqrt{t}}\Big)t^{-1/2}\Big[\frac{a^2}{t^2}\exp\Big(-\frac{a^2}{t}\Big)\Big]\,dt\Big(-\frac{2a}{a^2}\Big).
\end{equ}
Since the quantity in $[\cdot]$ is a total derivative, one can integrate by parts and find the relation
\begin{equ}\label{eq:ode1}
S'(a)=\frac{2}{a+a^3}-\frac{1}{a}S(a).
\end{equ}
One furthermore notices that
\begin{equ}\label{eq:ode2}
S(1)=\int_0^\infty\frac{1}{2}\d_t\Big(2\int_{1/\sqrt{t}}^\infty e^{-x^2}\,dx\Big)^2\,dt=\frac{\pi}{2}.
\end{equ}
The differential equation \eqref{eq:ode1}-\eqref{eq:ode2} defines $S$ uniquely and one can easily verify that $S(a)=\frac{2\tan^{-1}(a)}{a}$ is a solution. This proves \eqref{eq:random identity} and by scaling one gets \begin{equ}
\int_0^\infty \Ercf\Big(\frac{b}{\sqrt{t}}\Big)N_t(a)t^{-1}\,dt=\frac{2\tan^{-1}(a/b)}{\pi a},
\end{equ}
which also holds in the limiting case $b=0$.

Note now that we can write
\begin{equs}[eq:messy-calc]
-\cJ^0_{a/3}(1)&=\int_{\R^4_u}2\cK(\uone-z)\cR_{a}(\uone,z)+2\cK(\uone-z^0)\cR_{a}(\uone,z)-|\cR_{a}(\uone,z)|^2\,dz
\\
&=4\int_0^\infty a e^{-ar}\scJ(-1,1+r)\,dr+4\int_0^\infty a e^{-ar}\scJ(1,1+r)\,dr
\\
&\qquad-4\int_0^\infty\int_0^\infty a^2e^{-a(r+\bar r)}\scJ(1+r,1+\bar r)\,dr\,d\bar r\;,
\end{equs}
where $\scJ$ is defined by
\begin{equ}
\scJ(a,b)
:=\int_{\R^4_u}N_{4t}(x_1)N_{4t}(x_2)N_{4t}(x_3+a)N_{4t}(x_1)N_{4t}(x_2)N_{4t}(x_3+b)\,dz\;.
\end{equ}
For $a,b$ such that $a+b\geq0$, we then have
\begin{equs}[eq:big integral]
\scJ(a,b)&=\int_{\R^4_u}N_{8t}(2x_1)N_{8t}(2x_2)N_{8t}(2x_3+a+b)(N_{2t}(0))^{2}N_{2t}(a-b)\,dz
\\
&=\int_{\R^4_u}\frac{1}{2^{3}\pi t}N_{2t}(x_1)N_{2t}(x_2)N_{8t}(2x_3+a+b)N_{2t}(a-b)\,dz
\\
&=\int_{t,x_3>0}\frac{1}{2^{5}\pi t^{3/2}}N_1\Big(\frac{2x_3+a+b}{\sqrt{8t}}\Big)N_t\Big(\frac{a-b}{\sqrt{2}}\Big)\,dx_3\,dt
\\
&=\int_{t,\tilde x_3>0}\frac{1}{2^{9/2}\pi t}N_1\Big(\tilde x_3+\frac{a+b}{\sqrt{8t}}\Big)N_t\Big(\frac{a-b}{\sqrt{2}}\Big)\,d\tilde x_3\,dt
\\
&=\int_{t>0}\frac{1}{2^{11/2}\pi t} \Ercf\Big(\frac{a+b}{\sqrt{8t}}\Big)N_{t}\Big(\frac{a-b}{\sqrt{2}}\Big)\,dt
\\
&=\frac{1}{2^4\pi^2}\frac{\tan^{-1}\Big(2\frac{a-b}{a+b}\Big)}{a-b}\;,
\end{equs}
and in particular $\scJ(1,-1)=\frac{1}{64\pi}$,
which proves \eqref{eq:I0-value} as promised.

To show \eqref{eq:bound for cJ}, we bound each integral appearing in \eqref{eq:messy-calc} separately.
To bound the first two, notice that $\scJ(-1,1+r),\scJ(1,1+r)\lesssim 1\wedge r^{-1}$.
One can then decompose the integral as
\begin{equ}
a\Big(\int_0^1e^{-ar}\,dr+\int_1^{a^{-1}}e^{-ar}r^{-1}\,dr+\int_{a^{-1}}^\infty e^{-ar}r^{-1}\,dr\Big).
\end{equ}
Here the first integral is clearly bounded by $1$,
the second one by $|{\log a}|$, and
the third one is independent of $a$, so it is also of order $1$.
Moving on to the last term in \eqref{eq:messy-calc}, we use the the bound $\scJ(1+r,1+\bar r)\lesssim 1\wedge(r+\bar r)^{-1}$. 
One can then write
\begin{equs}
\int_0^\infty \int_0^\infty a^2 e^{-a(r+\bar r)}\big(1\wedge(r+\bar r)^{-1}\big)\,dr\,d\bar r
&=\int_0^\infty\int_0^{\tilde r} a^2 e^{-a \tilde r}\big(1\wedge \tilde r^{-1}\big)\,dr\,d\tilde r
\\
&=a^2\Big(\int_0^1e^{-a\tilde r}\tilde r\,d\tilde r+\int_1^{\infty}e^{-a\tilde r}\,d\tilde r\Big).
\end{equs}
Now the first integral is bounded by $1$ and the second by $a^{-1}$. This completes the proof.
\end{proof}

\subsection{\texorpdfstring{$\Phi^4_3$}{Phi\^4\_3} - the cubic term}\label{sec:restrict}
For the $\Phi^4_3$ equation, there is one more term that is well below regularity $-1$, which is the cube of $\Psi$.
While it does not require additional boundary renormalisation, the fact that it is `compatible' with the boundary renormalisation of the square is far from obvious.
This `compatibility' is formulated in the following lemma, whose proof is the goal of this section.
\begin{lemma}\label{lem:restrict2}
In the setting of Lemma \ref{lem:reno Phi43 square}, for any sufficiently small $\kappa>0$ the sequences $\delta_\d R^1_\eps\Psi_{\eps,c_\eps}$, $(\scR R^2_\eps)\Psi_{\eps,c_\eps}$, and $R^3_\eps\Psi_{\eps,c_\eps}$ converge in $\CC^{-3/2-\kappa}(\R^4)$ to limits that do not depend on $\rho$.
\end{lemma}
The first step is to show that the free field $\Psi_{0,a}$ can be restricted to the boundary $\partial D$ -- even though it is only a distribution.
This is not unlike the temporal restriction in \cite[Sec.~9.4]{H0}.
However, the present setting gets somewhat involved as we will need continuity of this restriction not only with respect to the mollification, but also with respect to the kernel as well as the hyperplane on which we restrict.

We start by introducing a few notations.
Recall that $\R^4_u=\{(x_0,x_1,x_2,x_3)\in\R^4:\,x_3>0\}$.
Take $\beta\in(0,5)$. Let $\mathscr{G}_\beta$ be the set of functions $G$ of the form
\begin{equ}\label{eq:dec1}
G=\sum_{n\in \N}G^n,
\end{equ}
where $G^n:\R^4\times\R^4\to\R$ such that $\supp G^n\subset\{(x,y):\,\|x-y\|\lesssim 2^{-n}\}$ and $|D_x^k D_y^\ell G^n(x,y)|\lesssim C 2^{n(5+|k|+|\ell|-\beta)}$ for $(x,y)\in(\R^4_u\cup-\R^4_u)^2$ with some $C$ for all $|k|\leq K$ and $\ell\leq L$.
The optimal choice of $C$ in the optimal choice of the decomposition \eqref{eq:dec1} yields a norm on $\scG_\beta$.
 The values $K$ and $L$ will be occasionally relevant, in this case we use the notation $\scG_\beta^{(K,L)}$.

For $y\in\R^4_u$ denote by $\bar y$ its projection to its first three coordinates.
Let $\scG_\beta^\d$ be the set of functions $G$ of the form \eqref{eq:dec1}
where this time $G^n:\R^3\times\R^4_u\to\R$ such that $\supp G^n\subset\{(x,y):\,\| x-\bar y\|\lesssim 2^{-n},\,y_3\in I_n\}$, where $I_n\subset\R$ are of size $2^{-n}$, and $| G^n|\lesssim C 2^{n(5-\beta)}$ with some $C$ for all $n\in\N$.
The optimal choice of $C$ in the optimal choice of the decomposition \eqref{eq:dec1} yields a `norm' on $\scG_\beta^\d$\footnote{Note that $\scG_\beta^\d$ is not actually a vector space! Scalar multiplication however is well-defined and our
`norm' is positive and one-homogeneous.}. On both  $\scG_\beta$ and  $\scG_\beta^\d$ we denote by $W^\eta$ the multiplication by $(|y_3|\wedge 1)^{\eta}$.
The relevant properties of these spaces are summarised below.
\begin{lemma}\label{lem:restriction}
\begin{enumerate}[(i)]
\item Define $G_{(\eps)}$ by
\begin{equ}
G_{(\eps)}(x,y)=\int G(x,z)\rho_\eps(y-z)\,dz,\qquad x,y,\in\R^4_u.
\end{equ}
Then for all $\eta>0$, $\beta'\in(\beta-1,\beta)$, $\eps>0$, and $G\in\scG_\beta^{(1,1)}$ one has
\begin{equ}
\|W^\eta \big(G- G_{(\eps)}\big)\|_{\scG_{\beta'}^{(1,0)}}\lesssim 
\eps^{\eta\wedge(\beta-\beta')}\|G\|_{\scG_{\beta}^{(1,1)}}.
\end{equ}
\item Define $G^{(r)}$ by
\begin{equ}
G^{(r)}(x,y)=G((x,r),y),\qquad x\in\R^3,y\in\R^4_u.
\end{equ}
Then for all $\beta'\in(\beta-1,\beta)$, $r,r'\geq 0$, and $G\in\scG_\beta^{(1,0)}$ one has
\begin{equ}
\|G^{(r)}-G^{(r')}\|_{\scG_{\beta'}^\d}\lesssim |r-r'|^{\beta-\beta'} \|G\|_{\scG_\beta^{(1,0)}}.
\end{equ}
\item Define the random variables $\Phi_G(\varphi)$ by
\begin{equ}
\Phi_G(\varphi)=\xi\Big(y\mapsto\int_{\R^3}G(x,y)\varphi(x)\,dx\Big),\qquad\varphi\in\CC^\infty(\R^3).
\end{equ}
Then for all $\eta\in(0,1/2)$, $\beta\in(0,5/2)$, $\kappa>0$ and $G\in W^{-\eta}\scG_\beta^\d$ there is a random distribution $\hat\Phi_G$ such that for all test functions $\varphi$ one has $\hat\Phi_G(\varphi)=\Phi_G(\varphi)$ almost surely. Furthermore for all $p>0$ and compact $\frK\subset\R^3$ one has the bound
\begin{equ}
\E\|\hat\Phi_G\|_{\CC^{-5/2+\beta-\eta-\kappa}(\frK)}^p\lesssim\|G\|_{W^{-\eta}\scG_\beta^\d}^p.
\end{equ}
\end{enumerate}
\end{lemma}
\begin{proof}
In all of the proofs by homogeneity we may and will assume that the norms appearing on the right-hand side equal to $1$.

(i)
Take $\eps\in(0,1]$ and write the trivial bound
\begin{equs}\label{eq:trivi0}
|\d_x^k (G^n- G^n_{(\eps)})(x,y)|&\leq 2^{n(5+|k|-\beta)}
\end{equs}
for $|k|\leq 1$.
If $\eps\geq |y_3|$, then this yields
\begin{equ}
|\d_x^k(G^n-G^n_{(\eps)})(x,y)|\leq \eps^\eta |y_3|^{-\eta} 2^{n(5+|k|-\beta)}.
\end{equ}
If $\eps<|y_3|$, then we can also write
\begin{equ}
|\d_x^k(G^n- G^n_{(\eps)})(x,y)|\lesssim \eps \sup_{y'}|\d_x^k\nabla_y G^n(x,y')|\leq \eps 2^{n(5+|k|+1-\beta)}.
\end{equ}
Interpolating between this and \eqref{eq:trivi0} gives the required bound.

(ii) The definitions immediately yield that for $|\alpha|\leq1$ one has
\begin{equ}
|(G^{(r),n}-G^{(r'),n})(x,y)|\leq  |r-r'|^\alpha\sup_{x'}|\d_{x_3}^{|\alpha|} G^n(x',y)|\leq |r-r'|^{|\alpha|} 2^{n(5+\alpha-\beta)}.
\end{equ}
By interpolation it also holds for $|\alpha|=\beta-\beta'$, giving the desired bounds on $G^{(r)}-G^{(r')}$.
As for the support,
$(G^{(r),n}-G^{(r'),n})(x,y)$ is supported on $\{(x,y):\,\| x-\bar y\|\lesssim 2^{-n},\,y_3\in I_n\}$ with $I_n=[r-2^{-n},r+2^{-n}]\cup[r'-2^{-n},r'+2^{-n}]$.

(iii) Let $\varphi^\lambda$ be a test function on $\R^3$ on scale $\lambda$. It follows from Gaussianity and Kolmogorov's H\"older estimate (for negative exponents) that it suffices to show that
\begin{equ}\label{eq:cube-Kolmogorov}
\Big\|y\mapsto \int_{\R^3}G(x,y)\varphi^\lambda(x)\,dx\Big\|_{L_2(\R^4)}\lesssim \lambda^{-5/2+\beta-\eta}.
\end{equ}
Set $\tilde Q^n=\{y:\,d(\bar y,\supp \varphi^\lambda)\lesssim 2^{-n}, y_3\in I_n\}$ and $Q^n=\tilde Q^n\setminus\cup_{j={n+1}}^\infty\tilde Q^n$.
It is clear that one has
\begin{equ}[eq:something0]
\Big|\int_{Q^n} |y_3|^{-2\eta}\,dy\Big|\lesssim \begin{cases}
2^{n(-5+2\eta)} &\text{if $2^{-n}\geq\lambda$},\\
\lambda^42^{n(-1+2\eta)} &\text{if $2^{-n}<\lambda$}.
\end{cases}
\end{equ}
We also claim that one has on $Q^n$,
\begin{equ}[eq:something]
\Big||y_3|^\eta\int_{\R^3}G(x,y)\varphi^\lambda(x)\,dx\Big|\lesssim
\begin{cases}
 2^{n(5-\beta)} &\text{if $2^{-n}\geq\lambda$},\\
 \lambda^{-(5-\beta)} &\text{if $2^{-n}<\lambda$}.
\end{cases}
\end{equ}
To see the first case in \eqref{eq:something}, notice that in the sum
\begin{equ}\label{eq:some sum}
y_3^\eta\sum_{m\in\N}\int_{\R^3} G^m(x,y)\varphi^\lambda(x)\,dx
\end{equ}
only terms with $2^{-m}\gtrsim 2^{-n}$ contribute. For each of these terms we use supremum bound on $W^\eta G^{n}$, and recall that $\varphi^\lambda$ is normalised in $L_1(\R^3)$.
Therefore we get a bound of order
\begin{equ}
\sum_{2^{-m}\gtrsim 2^{-n}}2^{m(5-\beta)}\lesssim 2^{n(5-\beta)}
\end{equ}
as claimed.
To see the second case in \eqref{eq:something}, the sum over the terms $2^{-m}\gtrsim\lambda$ yields the required bound precisely as above.
On the terms $2^{-m}\lesssim\lambda$
we bound $\varphi^\lambda$ by its supremum and use that $|{\supp G^m(\cdot,y)}|\lesssim 2^{-4m}$, and therefore
\begin{equ}
\sum_{2^{-m}\lesssim\lambda}\lambda^{-4} 2^{(5-\beta)m}2^{-4m}\lesssim\lambda^{-(5-\beta)}
\end{equ}
as claimed.
From \eqref{eq:something0}-\eqref{eq:something}
we can write
\begin{equs}
\Big\|y\mapsto &\int_{\R^3}G(x,y)\varphi^\lambda(x)\,dx\Big\|_{L_2(\R^4)}^2
\\
&\lesssim\sum_{n\in\N}
\Big|\int_{Q^n} |y_3|^{-2\eta}\,dy\Big|
\sup_{y\in Q^n}\Big||y_3|^\eta\int_{\R^3}G(x,y)\varphi^\lambda(x)\,dx\Big|^2
\\
&\lesssim\sum_{2^{-n}\geq\lambda} 2^{(5-2\beta+2\eta) n}+\sum_{2^{-n}<\lambda}\lambda^{-(6-2\beta)}2^{n(-1+2\eta)}
\lesssim\lambda^{-5+2\beta-2\eta},
\end{equs}
which is precisely \eqref{eq:cube-Kolmogorov}.
\end{proof}
Notice that for $\eps>0$, one has
\begin{equ}\label{eq:restriction identity}
\hat\Phi_{\cG_{3 c_\eps,(\eps)}^{(r)}}=\Psi_{\eps,c_\eps}(\cdot,r).
\end{equ}
For $\eps=0$, the right-hand side looks like the restriction of a distribution to a hyperplane, which is in general not allowed.
However, the left-hand side is perfectly meaningful thanks to Lemma \ref{lem:restriction}, so we can use \eqref{eq:restriction identity} to \emph{define} the restriction of $\Psi_{0,b}$ to the hyperplanes $\R^4_u+(0,0,0,r)$.
\begin{corollary}\label{cor:restrict1}
In the setting of Lemma \ref{lem:reno Phi43 square}, for all $\kappa>0$  the functions
$
r\to\Psi_{\eps,c_\eps}(\cdot,r)
$
and
$
r\to\Psi_{\eps,\infty}(\cdot,r)
$
converge in $\CC^{\kappa}\big([0,1],\CC^{-1/2-6\kappa}(\R^3)\big)$ in probability as $\eps\to 0$.
\end{corollary}
\begin{proof}
We only provide the argument for the first case since the second is easier.
It follows from Lemma \ref{lem:robin}
that the kernels $\cG_{3 c_\eps,(\eps)}$ converge in $\scG_{2-\kappa}^{(1,1)}$.
By Lemma~\ref{lem:restriction} (i) the convergence also holds in $W^{-\kappa}\scG_{2-2\kappa}^{(1,0)}$.
By Lemma~\ref{lem:restriction} (ii)--(iii), the functions
$r\mapsto \hat\Phi_{\cG_{3 c_\eps,(\eps)}^{(r)}}=\Psi_{\eps,c_\eps}(\cdot,r)$ converge in
$\CC^{\kappa}\big([0,1],L_p(\Omega,\CC^{-1/2-5\kappa}(Q_1))\big)$.
Kolmogorov's continuity theorem
shows that if $p$ is large enough, then the convergence holds in
$L_p\big(\Omega,\cC^{(5/6)\kappa}([0,1],\CC^{1/2-5\kappa}(Q_1))\big)$ which is as required,
provided we substitute $(5/6)\kappa\to\kappa$.
\end{proof}
Now we have all we need for the proof of Lemma \ref{lem:restrict2}.
Recall from Section \ref{sec:prepare} the notation $Q$ and $Q_\d$ and denote
$Q_r=\{(x_0,x_1,x_2):\,(x_0,x_1,x_2,r)\in Q\}$.
\begin{proof}[Proof of Lemma \ref{lem:restrict2}]
It follows from Proposition \ref{prop:weighted holder mult} and Corollary \ref{cor:restrict1} that the products $R^1_\eps\Tr_{\d}\Psi_{\eps,c_\eps}$ converge in $\CC^{-1/2-\kappa,-1/2-2\kappa}_{\d^2}(Q_\d)$. By Proposition \ref{prop:weighted holder extend}, this implies the convergence in $\CC^{-1/2-2\kappa}(\d)$.
It remains to use the basic fact that $\delta_\d  \CC^{\alpha}(\d)$ continuously embeds into $\CC^{\alpha-1-\kappa}(\R^4)$.

The last term is even easier: Proposition \ref{prop:weighted holder mult} and the convergence of $\Psi_{\eps,c_\eps}$ in $\CC^{-1/2-\kappa}(\R^d)$ implies the convergence of $R^3_\eps\Psi_{\eps,c_\eps}$ in $\CC^{1/2-\kappa,-3/2-2\kappa}_{\d^2}(Q)$, which by Proposition \ref{prop:weighted holder extend} implies the claim.

The statement concerning $R^2_\eps$ is more involved. 
First recall that $R^2_\eps$ does not depend on the variables $x_0,x_1,x_2$.
Take a test function $\varphi^\lambda$ on $\R^4$ on scale $\lambda$.
Our goal is then to show
\begin{equs}[some label]
\Big|\int_{0}^1 & \int_{\R^3} \bone_{y\in Q_r}\Big( R^2_\eps(r)\big(
\Psi_{\eps,c_\eps}(y,r)\varphi^\lambda(y,r)
-
\Psi_{\eps,c_\eps}(y,0)\varphi^\lambda(y,0)\big)
\\
&
-R^2_0(r)\big(
\Psi_{0,b}(y,r)\varphi^\lambda(y,r)
-
\Psi_{0,b}(y,0)\varphi^\lambda(y,0)\big)\Big)\,dy \,dr\Big|\lesssim o(1)\lambda^{-3/2-\kappa},
\end{equs}
where $o(1)\to0$ in the $\eps\to 0$ limit.
Note that the inner integral has to be understood in a distributional sense.
This understanding is justified by Corollary \ref{cor:restrict1} and by the fact that muliplying with $\bone_{Q_r}$ is a well-defined and continuous operation on $\CC^{-1/2-6\kappa}$.
First we treat the easy case when the support of $\varphi^\lambda$ is separated from $Q_\d$ by at least $\lambda$.
In this case the integrand simplifies to
\begin{equ}
\bone_{y\in Q_r}\big(R^2_\eps(r)(\Psi_{\eps,c_\eps}(y,r)-\Psi_{0,b}(y,r))+(R^2_\eps(r)-R^2_0(r))\Psi_{0,b}(y,r)\big)\varphi^\lambda(y,r)
\end{equ}
and keep in mind that $\lambda\leq r$.
By Corollary \ref{cor:restrict1} we have the bounds
$\|\Psi_{\eps,c_\eps}(\cdot,r)-\Psi_{0,b}(\cdot,r)\|_{\CC^{-1/2-\kappa}}\lesssim o(1)$ (in $\eps$) uniformly in $r$ as well as
$\|\Psi_{0,b}(\cdot,r)\|_{\CC^{-1/2-\kappa}}\lesssim 1$.
By Lemma \ref{lem:reno Phi43 square} (ii) we further have
$|R_\eps^2(r)|\lesssim r^{-1-\kappa}$ uniformly in $\eps,r$ as well as
$|R_\eps^2(r)-R_0^\eps(r)|\lesssim o(1)r^{-1-\kappa}$ uniformly in $r$.
Finally, notice that $\varphi^\lambda(\cdot,r)$ can be seen as $\lambda^{-1}$ times a test function on scale $\lambda$ on $\R^3$.
Combining these bounds show that the left-hand side of \eqref{some label} is bounded by
\begin{equ}
\int_0^1o(1)\lambda^{-3/2-\kappa}r^{-1-\kappa}\bone_{\lambda\leq r}\,dr
\leq \int_0^1 o(1)\lambda^{-3/2-3\kappa}r^{-1+\kappa}\,dr\lesssim o(1)\lambda^{-3/2-3\kappa}.
\end{equ}
This is the required bound with $3\kappa$ in place of $\kappa$, yielding the claim
in the case when the support of $\varphi^\lambda$ is separated from $Q_\d$ by at least $\lambda$.

In the alternative case the boundary terms in \eqref{some label} have to be taken into account.
Recall the elementary identity
\begin{equs}
a_1b_1-a_2b_2-a_3b_3+a_4b_4&=(a_1-a_2-a_3+a_4)b_1+(a_3-a_4)(b_1-b_3)
\\&\quad +(a_2-a_4)(b_1-b_2)+a_4(b_1-b_2-b_3+b_4),
\end{equs}
which holds for any "product" that is bilinear. In our situation the "product" will be the action of $\CC^{-1/2-\kappa}(\R^3)$ on $\CC^{1/2+\kappa}(\R^3)$, and the terms will be
$a_1=\Psi_{\eps,c_\eps}(\cdot,r)$,
$a_2=\Psi_{\eps,c_\eps}(\cdot,0)$,
$a_3=\Psi_{0,b}(\cdot,r)$,
$a_4=\Psi_{0,b}(\cdot,0)$, and
$b_1=R^2_\eps(r)\varphi^\lambda(\cdot,r)$,
$b_2=R^2_\eps(r)\varphi^\lambda(\cdot,0)$,
$b_3=R^2_0(r)\varphi^\lambda(\cdot,r)$,
$b_4=R^2_0(r)\varphi^\lambda(\cdot,0)$.
From Corollary \ref{cor:restrict1} we have the uniform bounds
\begin{equs}
\|a_1-a_2-a_3+a_4\|_{\CC^{-1/2-\kappa}} & \lesssim r^{\kappa/6}o(1),
\\
\|a_3-a_4\|_{\CC^{-1/2-\kappa}} & \lesssim r^{\kappa/6},
\\
\|a_2-a_4\|_{\CC^{-1/2-\kappa}}	& \lesssim o(1),
\\
\|a_4\|_{\CC^{-1/2-\kappa}} & \lesssim 1.
\end{equs}
Furthermore, not only can $\varphi^\lambda(\cdot,r)$ be seen as $\lambda^{-1}$ times a test function on scale $\lambda$ on $\R^3$, also
$\varphi^\lambda(\cdot,r)-\varphi^\lambda(\cdot,0)$ can be seen as $r^{\kappa/6}\lambda^{-1-\kappa/6}$ times a test function on scale $\lambda$ on $\R^3$, for sufficiently small $\kappa>0$.
Hence using Lemma \ref{lem:reno Phi43 square} with $\kappa/12$ in place of $\kappa$,
\begin{equs}
\|b_1\|_{\CC^{1/2+\kappa}} & \lesssim r^{-1-\kappa/12}\lambda^{-3/2-\kappa},
\\
\|b_1-b_3\|_{\CC^{1/2+\kappa}} & \lesssim r^{-1-\kappa/12}o(1)\lambda^{-3/2-\kappa},
\\
\|b_1-b_2\|_{\CC^{1/2+\kappa}}	& \lesssim r^{-1-\kappa/12+\kappa/6}\lambda^{-3/2-\kappa-\kappa/6},
\\
\|b_1-b_2-b_3+b_4\|_{\CC^{1/2+\kappa}}	& \lesssim r^{-1-\kappa/12+\kappa/12}o(1)\lambda^{-3/2-\kappa-\kappa/6}.
\end{equs}
Since the exponent of $r$ is greater than $-1$ in each of these terms,
the outer integration in \eqref{some label} can be performed as before and and the claimed bound holds.
\end{proof}

\section{Regularity structures and models}\label{sec:reg str}

\subsection{General remarks}
We will use a mild modification of the general black box theory of regularity structures. The first tweak is a slight relaxation of the required bounds on models. Its formulation is somewhat technical, but the moral of it is simply that it is sufficient to assume the bounds of the correct order from the models on test functions supported away from the boundary for any symbol of degree above the codimension of a given boundary.
\begin{proposition}\label{prop:model boundary extension}
Let $P$ be a boundary of codimension $k$ and let $\scT=(A,T,G)$ be a regularity structure. Assume that we are given mappings $\Pi$, $\bar \Pi$, and $\Gamma$ such that:
\begin{claim}
\item $\Gamma:\R^d\times\R^d\to G$ is continuous and satisfies $\Gamma_{xx}=1$ and $\Gamma_{xy}\Gamma_{yz}=\Gamma_{xz}$;
\item For all $x\in \R^d$, $\Pi_x$ maps $T_{\leq -k}$ to $\cS'(\R^d)$, while for all $x\in \R^d\setminus P$, $\bar\Pi_x$ maps $T$ to $\cS'(\R^d\setminus P)$, such that as elements of $\cS'(\R^d\setminus P)$, $\Pi_x\tau=\bar\Pi_x\tau$ for all $\tau\in T_{\leq-k}$;
\item the identities $\Pi_y=\Pi_x\Gamma_{xy}$ and $\bar\Pi_y=\bar\Pi_x\Gamma_{xy}$ hold;
\item on $T_{\leq -k}$, $(\Pi,\Gamma)$ is a model with norm bounded by $1$;
\item the following bounds hold:
\begin{equ}
|(\bar\Pi_z\tau)(\varphi_z^\lambda)|\leq \lambda^\alpha,\qquad|\Gamma_{xy}\tau|_\beta\leq\|x-y\|^{\beta-\alpha},
\end{equ}
for all $\alpha\in A$, $\tau\in T_\alpha$ with $|\tau|_\alpha=1$, $z\in \R^d\setminus P$, $\lambda\in(0,1]$ such that $\lambda\leq (1/2) |z|_P$, $y\in \R^d\setminus P$ such that $\|z-y\|\leq(1/2)|z|_P$, and $\beta<\alpha$.
\end{claim}
Then there exists a unique model of the form $(\hat\Pi, \Gamma)$ such that, as elements of $\cS'(\R^d)$, 
$\hat\Pi_x\tau=\Pi_x\tau$ for all $\tau\in T_{\leq-k}$ and as elements of $\cS'(\R^d\setminus P)$, $\hat\Pi_x\tau=\bar\Pi_x\tau$ for all $\tau\in T$. Furthermore, the norm of $(\hat\Pi, \Gamma)$ is bounded by a constant depending only on $\scT$ and $P$
and the map $(\Pi,\bar \Pi,\Gamma) \mapsto (\hat \Pi,\Gamma)$ is continuous in its natural topology. 
\end{proposition}

\begin{proof}
The proof is virtually identical to that of \cite[Thm~C.5]{HP19}.
\end{proof}

\begin{remark}
In this statement, the various `norms' are taken over the entire space $\R^d$, but since the operation is local
this can clearly be localised to compact regions.
\end{remark}

The convergence of models in our setting do not directly follow from \cite{CH}, but we aim to minimise the additional arguments.
Let us first very briefly summarise how the convergence results are obtained in \cite[Sec~10]{H0}, loosely following the notation therein.
To each basis symbol $\tau$ one associates functions
$\cW^{(\eps;k)}_i\tau(z;x;y_1,\ldots,y_k)$ in $k+2$ variables. Here $\eps\in[0,1]$, $k$ is a natural number, and $i$ runs over some finite set. 
For any fixed $z$, by Wiener's isometry $I_k$, any such function yields a distribution in the variable $x$, living in the $k$-th homogeneous Wiener chaos. The distribution $\Pi_z^\eps\tau$ is then defined as the sum of all these random distributions over all the indices $k,i$.
Let us point out here the the first slight difference to \cite{H0}: therein, due to the translation invariance, the dependence on $z$ can freely be ignored.

By \cite[Thm~10.7]{H0}, the convergence of a sequence of models $(\Pi^\eps,\Gamma^\eps)_{\eps\in[0,1]}$ of the above form follow from the bounds
\begin{equs}
\E\big|(\Pi_z^0\tau)(\varphi_z^\lambda)\big|^2
&\lesssim\lambda^{2|\tau|+\kappa},\label{eq:model-conv1}
\\
\E\big|(\Pi_z^0\tau-\Pi_z^\eps\tau)(\varphi_z^\lambda)\big|^2
&\lesssim\eps^{2\theta}\lambda^{2|\tau|+\kappa},\label{eq:model-conv2}
\end{equs}
where $\kappa,\theta>0$ are arbitrary. 
By \cite[Prop~10.11]{H0}, these bounds follow from
\begin{equs}
\big|\scal{ &\cW^{(0;k)}_i  \tau(z;x;\cdot), \cW^{(0;k)}_i\tau(z;\bar x;\cdot)}\big|
\\
&\lesssim
\sum_{\zeta}\big(\|x-z\|+\|\bar x-z\|\big)^\zeta\|x-\bar x\|^{2|\tau|+\kappa-\zeta}, \label{eq:model-conv3}
\\
\big|\scal{&(\cW^{(0;k)}_i-\cW^{(\eps;k)}_i) \tau(z;x;\cdot),(\cW^{(0;k)}_i-\cW^{(\eps;k)}_i)\tau(z;\bar x;\cdot)}\big|
\\
&\lesssim
\eps^{2\theta}\sum_{\zeta}\big(\|x-z\|+\|\bar x-z\|\big)^\zeta\|x-\bar x\|^{2|\tau|+\kappa-\zeta},\label{eq:model-conv4}
\end{equs}
where the scalar product is understood in the $k$-fold tensor product of $L_2(\R^{d_1})$ and the sum ranges over a finite set of values of $\zeta\in[0,2|\tau|+\kappa+|\frs|)$.
Moreover, due to Proposition \ref{prop:model boundary extension}, whenever $|\tau|>-1$, one only needs these bounds to hold for $x,\bar x$ with $|x|_\d\lesssim \|x-z\|, |\bar x|_\d\lesssim \|\bar x-z\|$.

Unfortunately, even with this modification in mind, \eqref{eq:model-conv4} will not always hold in our setting. It will be complemented with the following criterion.
\begin{proposition}
Let $z\in\R^d$, $P\subset\R^d$ a boundary, $\eps>0$, and $\tilde\cW:\R^d\times\R^{kd_1}\to\R$ be a function that satisfies the bounds
\begin{equ}
\big|\scal{\tilde\cW(x,\cdot),\tilde\cW(\bar x,\cdot)}\big|\lesssim \bone_{|x|_P\leq\eps} \sum_{\zeta}\big(\|x-z\|+\|\bar x-z\|\big)^\zeta\|x-\bar x\|^{2|\tau|+\kappa-\zeta},
\end{equ}
with the sum as above. 
Then one has the bounds, with any sufficiently small $\theta>0$,
\begin{equ}\label{eq:modified-crit}
\E\Big|I_k\Big(\int\varphi_z^\lambda(x)\tilde\cW(x,\cdot)\,dx\Big)\Big|^2\lesssim \eps^{2\theta}\lambda^{2|\tau|+\kappa-2\theta}.
\end{equ}
\end{proposition}
\begin{proof}
First note that the $\lambda\lesssim\eps$ case is straightforward, even without using the indicator function $\bone_{|x|_P\leq\eps}$.
Indeed, just as one proves \eqref{eq:model-conv1} from \eqref{eq:model-conv3}, one can bound the left-hand side of \eqref{eq:modified-crit} by $\lambda^{2|\tau|+\kappa}\lesssim\eps^{2\theta}\lambda^{2|\tau|+\kappa-\theta}$ as required.

We therefore may assume $\eps\lesssim \lambda$. The left-hand side of \eqref{eq:modified-crit} can clearly be bounded by
\begin{equs}
\,&\Big|\int  \int\varphi_z^\lambda(x)\varphi_z^\lambda(\bar x)
\scal{\tilde\cW(x,\cdot),\tilde\cW(\bar x,\cdot)}\,dx\,d\bar x\Big|
\\
&\lesssim \lambda^{-2|\frs|}\sum_{\zeta}\int\int\bone_{|x|_P\leq\eps,\|x-z\|\leq\lambda,\|\bar x-z\|\leq \lambda}
\big(\|x-z\|+\|\bar x-z\|\big)^\zeta\|x-\bar x\|^{2|\tau|+\kappa-\zeta}\,dx\,d\bar x
\\
&\lesssim \lambda^{-2|\frs|}\sum_{\zeta}\lambda^\zeta
\int\int\bone_{|x|_P\leq \eps,\|x-z\|\leq\lambda,\|\bar x-z\|\leq \lambda}
\|x-\bar x\|^{2|\tau|+\kappa-\zeta}\,dx\,d\bar x
\\
&\lesssim\lambda^{-2|\frs|}\sum_{\zeta}\lambda^\zeta
\int\int\bone_{|x|_P\leq \eps,\|x-z\|\leq\lambda,\|\hat x\|\leq \lambda}
\|\hat x\|^{2|\tau|+\kappa-\zeta}\,dx\,d\hat x.
\end{equs}
By the assumptions on $\zeta$, the integral in $\hat x$ is finite and bounded by $\lambda^{2|\tau|+\kappa-\zeta+|\frs|}$. The integral in $x$ is trivially bounded by $\eps\lambda^{|\frs|-1}$, and hence we get the desired bound.
\end{proof}
Applying the proposition with $\tilde\cW=\delta\cW^{(\eps;k)}_i\tau(z;\cdot;\cdot)$, we see that if instead of \eqref{eq:model-conv4} one has the bounds
\begin{equs}
\big|\scal{&(\cW^{(0;k)}_i-\cW^{(\eps;k)}_i) \tau(z;x;\cdot),(\cW^{(0;k)}_i-\cW^{(\eps;k)}_i)\tau(z;\bar x;\cdot)}\big|
\\
&\lesssim
\bone_{|x|_P\leq \eps}\sum_{\zeta}\big(\|x-z\|+\|\bar x-z\|\big)^\zeta\|x-\bar x\|^{2|\tau|+\kappa-\zeta},\label{eq:model-conv5}
\end{equs}
then \eqref{eq:model-conv2} holds.
\begin{remark}\label{rem:indicator}
Clearly the convergence criteria \eqref{eq:model-conv3}, \eqref{eq:model-conv4} and \eqref{eq:model-conv5} have the property that if they hold for $\cW^{(\eps;k)}_i\tau$, then they also hold for $\bone_{x\in A}\cW^{(\eps;k)}_i\tau$, for any measurable set $A$.
Informally speaking, the multiplication of Gaussian models with indicator functions is straightforward.
\end{remark}

Next, we need some analogues of \cite[Lem~10.14]{H0} in the case where
the blowup of a kernel is not only controlled by the distance to the diagonal but also to the boundary.
For $\alpha,\gamma\leq 0$, denote by $\|K\|_{\alpha,\gamma}$ the best proportionality constant in the bound
\begin{equ}\label{eq:kernel-singu}
|K(x,y)|\lesssim \big(|x|_\d+\|x-y\|)^\alpha\|x-y\|^\gamma.
\end{equ}
When $\alpha=0$, then $\|\cdot\|_{0,\gamma}$ coincides with $\|\cdot\|_{\gamma;0}$ from \cite[Sec~10.3]{H0}, with the latter $0$ indicating that no derivatives are involved in the bounds.
Note also the trivial property that negative powers can be transferred from the first component to the second, that is, $\|K\|_{\alpha+\alpha',\gamma-\alpha'}\leq\|K\|_{\alpha,\gamma}$ for $\alpha'\in[0,-\alpha]$.
\begin{example}
Consider the $1+1$-dimensional homogeneous Neumann heat kernel on the positive half line as in Example \ref{example1} and let the two terms in \eqref{eq:example1} be denoted by $K_1$ and $K_2$. Then $\|K_1\|_{0,-1}\lesssim 1$ and $\|K_2\|_{-1,0}\lesssim 1$.
\end{example}
\begin{lemma}\label{lem:kernel-convolution}
(i) One has the bounds
\begin{equ}
\|K\bar K\|_{\alpha+\bar\alpha,\gamma+\bar\gamma}\leq
\|K\|_{\alpha,\gamma}
\|\bar K\|_{\bar\alpha,\bar\gamma}.
\end{equ}
(ii) Set $\tilde{K}(x,z)=\int K(x,y)\bar K(y,z)\,dy$.
Then in case
\begin{equ}
\gamma+\bar\gamma>-|\frs|,\quad\alpha+\gamma>-|\frs|,\quad\alpha+\gamma+\bar\gamma<-|\frs|
\end{equ}
one has the bounds
\begin{equ}
\|\tilde K\|_{\alpha+\gamma+\bar\gamma+|\frs|,0}\lesssim
\|K\|_{\alpha,\gamma}
\|\bar K\|_{0,\bar\gamma}.
\end{equ}
(iii) In case
\begin{equ}
\gamma,\bar\gamma>-|\frs|,\quad\gamma+\bar\gamma<-|\frs|,\quad\alpha+\gamma<-|\frs|
\end{equ}
one has the bounds
\begin{equ}
\big\||x|_\d^{-\alpha-\gamma-|\frs|}\tilde K\big\|_{0,\bar\gamma}\lesssim
\|K\|_{\alpha,\gamma}
\|\bar K\|_{0,\bar\gamma}.
\end{equ}
(iv) In case
\begin{equ}
\gamma,\bar\gamma>-|\frs|,\quad\gamma+\bar\gamma<-|\frs|,
\end{equ}
one has the bounds
\begin{equ}
\|\tilde K\|_{0,\gamma+\bar\gamma+|\frs|}\lesssim\|K\|_{0,\gamma}
\|\bar K\|_{0,\bar\gamma}.
\end{equ}
\end{lemma}
\begin{proof}
Claim (i) is trivial. Claim (iv) follows from \cite[Lem~10.14]{H0}.
Concerning (ii) and (iii), we divide the integral defining $\tilde K$ into separate regions. In the case $2|x|_\d\leq\|x-z\|$, we set
$A_1=\{\|x-y\|\leq|x|_\d\}$,
$A_2=\{|x|_\d\leq\|x-y\|\leq(1/2)\|x-z\|\}$,
$A_3=\{\|z-y\|\leq (1/2)\|x-z\|\}$,
$A_4=\R^d\setminus(A_1\cup A_2\cup A_3)$.
In the case $2|x|_\d\geq\|x-z\|$, we set
$B_1=\{\|x-y\|\leq(1/2)\|x-z\|\}$,
$B_2=\{\|z-y\|\leq(1/2)\|x-z\|\}$,
$B_3=\{\|x-y\|\leq 4|x|_\d\}\setminus(B_1\cup B_2)$,
$B_4=\R^d\setminus B_3$.
A schematic picture of the different regions is given below, where thick lines denote balls with radii of order $|x|_\d$ and thin lines denote balls with radii of order $\|x-z\|$.
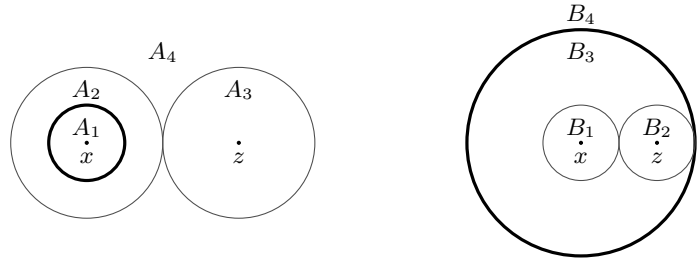
\begin{center}
\begin{tikzpicture}
\draw[black!70!white] (0,0) circle (1cm);
\draw[very thick] (0,0) circle (0.5cm);
\draw[black!70!white] (2,0) circle (1cm);
\draw[fill=black] (0,0) circle (0.2mm);
\draw[fill=black] (2,0) circle (0.2mm);
\draw (0,-0.2) node {\footnotesize $x$};
\draw (2,-0.2) node {\footnotesize $z$};
\draw (0,0.2) node {\footnotesize $A_1$};
\draw (0,0.7) node {\footnotesize $A_2$};
\draw (2,0.7) node {\footnotesize $A_3$};
\draw (1,1.2) node {\footnotesize $A_4$};

\draw[very thick](6.5,0) circle (1.5cm);
\draw[black!70!white] (6.5,0) circle (0.5cm);
\draw[black!70!white] (7.5,0) circle (0.5cm);
\draw[fill=black] (6.5,0) circle (0.2mm);
\draw[fill=black] (7.5,0) circle (0.2mm);
\draw (6.5,-0.2) node {\footnotesize $x$};
\draw (7.5,-0.2) node {\footnotesize $z$};
\draw (6.5,0.2) node {\footnotesize $B_1$};
\draw (7.5,0.2) node {\footnotesize $B_2$};
\draw (6.5,1.2) node {\footnotesize $B_3$};
\draw (6.5,1.7) node {\footnotesize $B_4$};
\end{tikzpicture}
\end{center}
First, we have
\begin{equ}
\Big|\int_{A_1}K(x,y)\bar K(y,z)\,dy\Big|\lesssim
|x|_\d^\alpha\int_{A_1}\|x-y\|^\gamma\|y-z\|^{\bar \gamma}\,dy\lesssim|x|_\d^{\alpha+\gamma+|\frs|}\|x-z\|^{\bar\gamma},
\end{equ}
where we have used that $\gamma>-|\frs|$ in both of (ii) and (iii). This bound is clearly the right order for (iii), while for (ii) the condition $\alpha+\gamma>-|\frs|$ implies $|x|_\d^{\alpha+\gamma+|\frs|}\leq\|x-z\|^{\alpha+\gamma+|\frs|}$. Next,
\begin{equ}
\Big|\int_{A_2}K(x,y)\bar K(y,z)\,dy\Big|\lesssim
\|x-z\|^{\bar \gamma}\int_{A_2}\|x-y\|^{\alpha+\gamma}\,dy.
\end{equ}
For (ii), the condition $\alpha+\gamma>-|\frs|$ implies that the integral is of order $\|x-z\|^{\alpha+\gamma+|\frs|}$, yielding the required bound.
For (iii), the integral is of order $|x|_{\d}^{\alpha+\gamma+|\frs|}$, also as required. Further,
\begin{equ}
\Big|\int_{A_3}K(x,y)\bar K(y,z)\,dy\Big|\lesssim
\|x-z\|^{\alpha+\gamma}\int_{A_3}\|y-z\|^{\bar \gamma}\,dy\lesssim\|x-z\|^{\alpha+\gamma+\bar\gamma+|\frs|},
\end{equ}
using $\bar \gamma>-|\frs|$. For (ii), this is as desired. For (iii)
the condition $\alpha+\gamma<-|\frs|$ implies $\|x-z\|^{\alpha+\gamma+|\frs|}\lesssim|x|_\d^{\alpha+\gamma+|\frs|}$, yielding also a bound of the correct order. Finally,
\begin{equ}
\Big|\int_{A_4}K(x,y)\bar K(y,z)\,dy\Big|\lesssim
\int_{A_4}\|x-y\|^{\alpha+\gamma}\|y-z\|^{\bar \gamma}\,dy\lesssim\|x-z\|^{\alpha+\gamma+\bar\gamma+|\frs|},
\end{equ}
where we have used that $\alpha+\gamma+\bar\gamma<-|\frs|$ in both of (ii) and (iii). This is the same bound as in the case of $A_3$, and so it is of the right order.

We now move on the case $2|x|_\d\geq\|x-z\|$.
First, we have
\begin{equ}
\Big|\int_{B_1}K(x,y)\bar K(y,z)\,dy\Big|\lesssim
|x|_\d^\alpha\int_{B_1}\|x-y\|^\gamma\|y-z\|^{\bar \gamma}\,dy
\lesssim|x|_\d^{\alpha}\|x-z\|^{\gamma+\bar\gamma+|\frs|}.
\end{equ}
For (ii), one has
$\|x-z\|^{\gamma+\bar\gamma+|\frs|}\lesssim|x|_\d^{\gamma+\bar\gamma+|\frs|}$, giving the required bound.
For (iii), we bound $\|x-z\|^{\gamma+|\frs|}\lesssim|x|_\d^{\gamma+|\frs|}$.
The integral over $B_2$ is treated in exactly the same way.
Further,
\begin{equ}
\Big|\int_{B_3}K(x,y)\bar K(y,z)\,dy\Big|\lesssim
|x|_\d^\alpha\int_{B_3}\|x-y\|^\gamma\|y-z\|^{\bar \gamma}\,dy
\lesssim|x|_\d^\alpha\int_{B_3}\|x-y\|^{\gamma+\bar\gamma}\,dy.
\end{equ}
For (ii), the condition $\gamma+\bar\gamma>-|\frs|$ implies that the integral is of order $|x|_\d^{\gamma+\bar\gamma+|\frs|}$, yielding the required bound.
For (iii), the integral is of order $\|x-z\|^{\gamma+\bar\gamma+|\frs|}\lesssim |x|_\d^{\gamma+|\frs|}\|x-z\|^{\bar\gamma}$, also as required.
Finally, 
\begin{equ}
\Big|\int_{B_4}K(x,y)\bar K(y,z)\,dy\Big|\lesssim
\int_{B_4}\|x-y\|^{\alpha+\gamma}\|y-z\|^{\bar \gamma}\,dy
\lesssim|x|_\d^{\alpha+\gamma+\bar\gamma+|\frs|},
\end{equ}
where we have used that $\alpha+\gamma+\bar\gamma<-|\frs|$ in both of (ii) and (iii).
For (ii), this is the required bound.
For (iii), one can write $|x|_\d^{\bar\gamma}\lesssim\|x-z\|^{\bar\gamma}$. This exhausts all cases and finishes the proof.
\end{proof}
We will also use the following fact frequently used in the singular SPDE literature (see e.g. \cite{HS15, HM18}).
It is a simple tool to switch between convergences with respect to different parameters, which can be useful in situations where convergence of an auxiliary approximation is much easier to show.
\begin{proposition}\label{prop:triv-but-nice}Let $(a^{\eps,\delta})_{\eps,\delta\in[0,1]}$ be a two-parameter family in a metric space $(A,d)$. Suppose that $a_{\eps,\delta}\to a_{\eps,0}$ as $\delta\to 0$ uniformly in $\eps\in[0,1]$, and that $a_{\eps,\delta}\to a_{0,\delta}$ as $\eps\to 0$ for any $\delta\in(0,1]$. Then $a_{\eps,0}\to a_{0,0}$.
A quantitative version also holds: if for some $C_1,C_2,\gamma_1,\gamma_2>0$, $\gamma_0\in\R$ one has $d(a_{\eps,\delta},a_{\eps,0})\leq C_1\delta^{-\gamma_1}$ and $d(a_{\eps,\delta},a_{0,\delta})\leq C_2\eps^{-\gamma_2}\delta^{\gamma_0}$,
then for some $C_3,\gamma_3>0$ one has $d(a_{\eps,0},a_{0,0})\leq C_3\eps^{-\gamma_3}$.\end{proposition}

\subsection{PAM}
In the case of the \eqref{e:PAM}, the model is time-homogeneous.
By the general machinery developed in \cite{BHZ}, the system formally given by
\begin{equ}
\Delta Y=\xi,\qquad
\Delta v=v|\nabla Y|^2-2\nabla v\cdot\nabla Y,
\end{equ}
determines a regularity structure $\scT_0=(T_0,G_0,A_0)$, with
$|\Xi|=-3/2-\kappa$ for some sufficiently small $\kappa>0$.
We use common pictorial representations of some elements of $T_0$ by writing $\<0>$ for $\Xi$, thin lines for $\cI$, thick red lines for the abstract gradient of $\cI$, and when joining red lines at their root we understand the scalar product of the terms.
For example, $\<PAM-2I>=\sum_{i=1}^3\cI(\cD_i \cI\Xi)^2$

We then take a new regularity structure $\scT=(T,G,A)$ by adding three new symbols: we set 
$T=T_0\oplus \scal{\<PAM-2new>,\<PAM-4new>,\<PAM-2new>X}$. The homogeneities of the new symbols are the same 
as for the corresponding `red' symbols, namely $|\,\<PAM-2new>\,|=-1-2\kappa$, $|\,\<PAM-4new>\,|=-4\kappa$, 
$|\,\<PAM-2new>X\,|=-2\kappa$, so $A=A_0$.
The group $G$ is isomorphic to $G_0$, the action of its elements on the new symbols being uniquely determined 
by setting $G\<PAM-2new>=\<PAM-2new>$ and requiring the product to be regular.
We will also use the notation $\cF_\d=\{\<PAM-2new>,\<PAM-4new>,\<PAM-2new>X\}$.

Recall from the setup of Section \ref{sec:PAM square} that $\bar K$ stands for the truncated Green's function of the $3$-dimensional Poisson equation and satisfies \cite[Ass.~5.1]{H0}.
One can then use the results of \cite{CH} to build the BPHZ models
$(\Pi^\eps,\Gamma^\eps)_{\eps\in[0,1]}$ for $\scT_0$,
which converge in probability as $\eps\rightarrow0$,
are admissible with respect to $\bar K$,
and satisfy $\Pi_x^\eps\Xi=\xi_\eps$, where $\xi^\eps$ denotes a mollified stationary white noise
defined on all of $\R^3$.
One easily sees that $\Pi^\eps_x\<PAM-2a>=|\nabla \bar K\ast\xi_\eps|^2-\ell_\eps(\<PAM-2a-Small>)$ with $\ell_\eps(\<PAM-2a-Small>)$ satisfying \eqref{eq:RC def 1}.
Let us furthermore recall the notation $G$ for the Green's function of the $3$-dimensional Neumann Green's function on $D$. As such, $(x,y)\mapsto G(x,y)$ is $0$ outside $D\times D$.
As a convention, by $\nabla_1 G$ we mean the function that equals to the gradient in the $x$ direction of $G$ on $D\times D$
and $0$ outside $D\times D$
(as opposed to the distributional derivative of $x\mapsto G(x,y)$, which has an additional boundary term).

As mentioned above, the extension of the $\Gamma^\eps$ component to $\scT$ is automatic. Concerning the extension of $\Pi^\eps$, we set for $\eps>0$
\begin{equs}
\big(\Pi^\eps_x\<PAM-2new>\big)(y)&=|\Psi_\eps|^2(y)-\E|\Psi_\eps|^2(y),
\\
\big(\Pi^\eps_x\<PAM-4new>\big)(y)&=\Big(\big(\Pi^\eps_x\<PAM-I2a>\big)(y)\Big)\Big(\big(\Pi^\eps_x\<PAM-2new>\big)(y)\Big)- \ell_\eps(\<PAM-4a-Small>),
\\
\big(\Pi^\eps_x\<PAM-2new>X\big)(y)&=\Big(\big(\Pi^\eps_x\<PAM-2new>\big)(y)\Big)(y-x).
\end{equs}
As in \cite[App C]{HP19}, we also consider the models $(\Pi^{+,\eps},\Gamma^\eps)$ and $(\Pi^{-,\eps},\Gamma^\eps)$ defined by $(\Pi^{+,\eps}_x\tau)(y)=(\Pi^\eps_x\tau)(y)\bone_{\R\times D}(y)$ and $(\Pi^{-,\eps}_x\tau)(y)=(\Pi^\eps_x\tau)(y)\bone_{\R\times D^c}(y)$. These models are of course \emph{not} admissible, but will be very handy in the analysis of the reconstruction operator in Section \ref{sec:proofs} below.

\begin{lemma}\label{lem:PAM model convergence}
The three sequences of models $(\Pi^{\eps},\Gamma^\eps)$, $(\Pi^{+,\eps},\Gamma^\eps)$, and $(\Pi^{-,\eps},\Gamma^\eps)$, converge.
\end{lemma}
\begin{proof}
We start by establishing a priori bounds on the models $(\Pi^{\eps},\Gamma^\eps)$, including a candidate limit model for $\eps=0$.
As discussed above, we phrase everything in terms of the functions $\cW$ and borrow some some graphical notations commonplace in the singular SPDE literature. 
Each function $\cW$ is given as an integral, which we represent by a graph as follows.
Each vertex of the graph represents a variable, and each edge a function of the two variables corresponding to the two vertices it connects. 
The node
\begin{tikzpicture}[baseline={(0,-0.1)}]
\draw (0,0) node[base]{};
\end{tikzpicture}
denotes the basepoint $z$;
\begin{tikzpicture}[baseline={(0,-0.1)}]
\draw (0,0) node[vari]{};
\end{tikzpicture}
denotes the running variable $x$;
\begin{tikzpicture}[baseline={(0,-0.1)}]
\draw (0,0) node[noise]{};
\end{tikzpicture}
denote the variables $y_1,\ldots,y_k$ (on which Wiener's isometry acts);
\begin{tikzpicture}[baseline={(0,-0.1)}]
\draw (0,0) node[int]{};
\end{tikzpicture}
denotes an auxiliary variable that is integrated out.
The edge
\begin{tikzpicture}[baseline={(0,-0.1)}]
\draw (-0.5,0) -- (0,0);
\end{tikzpicture}
denotes a factor $\bar K(u-v)$, where $u$ and $v$ are the two variables corresponding to the the two endpoints;
\begin{tikzpicture}[baseline={(0,-0.1)}]
\draw[kernels2] (-0.5,0) -- (0,0);
\end{tikzpicture}
denotes a factor $\nabla\bar K(u-v)$;
\begin{tikzpicture}[baseline={(0,-0.1)}]
\draw[rho] (-0.5,0) -- (0,0);
\end{tikzpicture}
denotes a factor $\rho_\eps(u-v)$;
\begin{tikzpicture}[baseline={(0,-0.1)}]
\draw[kernels3] (-0.5,0) -- (0,0);
\end{tikzpicture}
denotes a factor $\nabla_1 G(u,v)$.
For example, $\Pi^\eps_{z}\<PAM-2a>$ has only one component in its chaos decomposition, which is the image under Wiener's isometry of the function $\cW^{(\eps;2)}\<PAM-2a>$ given by
\begin{equ}
\cW^{(\eps;2)}\<PAM-2a>(z,x,y_1,y_2)=\bar K^\eps(x-y_1)\bar K^\eps(x-y_2)=
\begin{tikzpicture}[scale=0.7,baseline={(0,0.3)}]
\draw[kernels2] (-1,1) -- (0,0) node[vari]{} -- (1,1) ;
\draw[rho] (-1,1) node[int]{} -- (-1,1.5);
\draw[rho] (1,1) node[int]{} -- (1,1.5);
\draw (-1,1.5) node[noise]{};
\draw (1,1.5) node[noise]{};
\end{tikzpicture}\quad.
\end{equ}
In particular, $\cW^{(\eps;2)}\<PAM-2a>$ does not depend on $z$.
Similarly, for the first `new' symbol we have
\begin{equ}
\cW^{(\eps;2)}\<PAM-2new>=
\begin{tikzpicture}[scale=0.7,baseline={(0,0.3)}]
\draw[kernels3] (-1,1) -- (0,0) node[vari]{} -- (1,1) ;
\draw[rho] (-1,1) node[int]{} -- (-1,1.5);
\draw[rho] (1,1) node[int]{} -- (1,1.5);
\draw (-1,1.5) node[noise]{};
\draw (1,1.5) node[noise]{};
\end{tikzpicture}\quad.
\end{equ}
Note that in the $\eps=0$ case the edge
\begin{tikzpicture}[baseline={(0,-0.1)}]
\draw[rho] (-0.5,0) -- (0,0) node[int]{};
\end{tikzpicture}
together with its endpoint becomes an integration against a Dirac-$\delta$ so it can be dropped.
This then \emph{defines} $\cW^{(0;2)}\<PAM-2a>$ and $\cW^{(0;2)}\<PAM-2new>$.

Later on we will also use the shorthand 
$\begin{tikzpicture}[baseline={(0,-0.1)}]
\draw (-0.5,0) node[int]{} -- (0,0) node[vari]{};
\draw (-0.25,0.07) -- (-0.25,-0.07);
\end{tikzpicture}=
\begin{tikzpicture}[baseline={(0,-0.1)}]
\draw (-0.5,0) node[int]{} -- (0,0) node[vari]{};
\end{tikzpicture}-\begin{tikzpicture}[baseline={(0,-0.1)}]
\draw (-0.5,0) node[int]{} -- (0,0) node[base]{};
\end{tikzpicture}$
for the recentering of the kernel $\cK$.
A lot of our simplification relies on the following property: if a given edge appears without any recentering in a graph, then the bound obtained in \cite{CH} on that graph only uses the (pointwise) upper bound on that kernel.
For example, the bound \eqref{eq:model-conv3} on the function $\cW^{(\eps;2)}\<PAM-2new>$ follow exactly in the same way as in the translation invariant case for $\cW^{(\eps;2)}\<PAM-2a>$, since the kernels $G^\eps$ and $\bar K^\eps$ both satisfy the required bound of order $-2$.
This also immediately implies the required bounds for the symbol $\<PAM-2new>X$.

Concerning the last remaining `new' symbol $\<PAM-4new>$ and its translation invariant counterpart $\<PAM-4a>$, we have
\begin{equ}
\cW^{(\eps;4)}\<PAM-4a>=
\begin{tikzpicture}[scale=0.7,baseline={(0,0.3)}]
\draw[kernels2] (-1,1) -- (0,0) -- (1,1) ;
\draw[kernels2] (-1,2.5) -- (0,1.5) -- (1,2.5) ;
\draw (0,0) -- (0,1.5);
\draw[rho] (-1,1) node[int]{} -- (-1,1.5);
\draw[rho] (1,1) node[int]{} -- (1,1.5);
\draw[rho] (-1,2.5) node[int]{} -- (-1,3);
\draw[rho] (1,2.5) node[int]{} -- (1,3);
\draw (-0.1,0.75)--(0.1,0.75);
\draw (0,0) node[vari]{};
\draw (-1,1.5) node[noise]{};
\draw (1,1.5) node[noise]{};
\draw (0,1.5) node[int]{};
\draw (-1,3) node[noise]{};
\draw (1,3) node[noise]{};
\end{tikzpicture}
\quad,\quad
\cW^{(\eps;2)}\<PAM-4a>=2\,
\begin{tikzpicture}[scale=0.7,baseline={(0,0.3)}]
\draw[kernels2] (-1,1) -- (0,0) -- (1,1) ;
\draw[kernels2] (-1,2.5) -- (0,1.5) -- (1,2.5) ;
\draw (0,0) -- (0,1.5);
\draw[rho] (-1,1) node[int]{} -- (-1,1.75);
\draw[rho] (1,1) node[int]{} -- (1,1.5);
\draw[rho] (-1,2.5) node[int]{} -- (-1,1.75);
\draw[rho] (1,2.5) node[int]{} -- (1,3);
\draw (-0.1,0.75)--(0.1,0.75);
\draw (0,0) node[vari]{};
\draw (1,1.5) node[noise]{};
\draw (0,1.5) node[int]{};
\draw (1,3) node[noise]{};
\draw (-1,1.75) node[int]{};
\end{tikzpicture}
\quad,\quad
\cW^{(\eps;0)}\<PAM-4a>=-2\,
\begin{tikzpicture}[scale=0.7,baseline={(0,0.3)}]
\draw[kernels2] (-1,1) -- (0,0) -- (1,1) ;
\draw[kernels2] (-1,2.5) -- (0,1.5) -- (1,2.5) ;
\draw[rho] (-1,1) node[int]{} -- (-1,1.75);
\draw[rho] (1,1) node[int]{} -- (1,1.75);
\draw[rho] (-1,2.5) node[int]{} -- (-1,1.75);
\draw[rho] (1,2.5) node[int]{} -- (1,1.75);
\draw (0,1.5) -- (0,3);
\draw (0,3) node[base]{};
\draw (0,0) node[vari]{};
\draw (0,1.5) node[int]{};
\draw (-1,1.75) node[int]{};
\draw (1,1.75) node[int]{};
\end{tikzpicture}
\quad,
\end{equ}
\begin{equ}
\cW^{(\eps;4)}\<PAM-4new>=
\begin{tikzpicture}[scale=0.7,baseline={(0,0.3)}]
\draw[kernels3] (-1,1) -- (0,0) -- (1,1) ;
\draw[kernels2] (-1,2.5) -- (0,1.5) -- (1,2.5) ;
\draw (0,0) -- (0,1.5);
\draw[rho] (-1,1) node[int]{} -- (-1,1.5);
\draw[rho] (1,1) node[int]{} -- (1,1.5);
\draw[rho] (-1,2.5) node[int]{} -- (-1,3);
\draw[rho] (1,2.5) node[int]{} -- (1,3);
\draw (-0.1,0.75)--(0.1,0.75);
\draw (0,0) node[vari]{};
\draw (-1,1.5) node[noise]{};
\draw (1,1.5) node[noise]{};
\draw (0,1.5) node[int]{};
\draw (-1,3) node[noise]{};
\draw (1,3) node[noise]{};
\end{tikzpicture}
\quad,\quad
\cW^{(\eps;2)}\<PAM-4new>=2\,
\begin{tikzpicture}[scale=0.7,baseline={(0,0.3)}]
\draw[kernels3] (-1,1) -- (0,0) -- (1,1) ;
\draw[kernels2] (-1,2.5) -- (0,1.5) -- (1,2.5) ;
\draw (0,0) -- (0,1.5);
\draw[rho] (-1,1) node[int]{} -- (-1,1.75);
\draw[rho] (1,1) node[int]{} -- (1,1.5);
\draw[rho] (-1,2.5) node[int]{} -- (-1,1.75);
\draw[rho] (1,2.5) node[int]{} -- (1,3);
\draw (-0.1,0.75)--(0.1,0.75);
\draw (0,0) node[vari]{};
\draw (1,1.5) node[noise]{};
\draw (0,1.5) node[int]{};
\draw (1,3) node[noise]{};
\draw (-1,1.75) node[int]{};
\end{tikzpicture}
\quad,\quad
\cW^{(\eps;0)}_1\<PAM-4new>=-2\,
\begin{tikzpicture}[scale=0.7,baseline={(0,0.3)}]
\draw[kernels3] (-1,1) -- (0,0) -- (1,1) ;
\draw[kernels2] (-1,2.5) -- (0,1.5) -- (1,2.5) ;
\draw[rho] (-1,1) node[int]{} -- (-1,1.75);
\draw[rho] (1,1) node[int]{} -- (1,1.75);
\draw[rho] (-1,2.5) node[int]{} -- (-1,1.75);
\draw[rho] (1,2.5) node[int]{} -- (1,1.75);
\draw (0,1.5) -- (0,3);
\draw (0,3) node[base]{};
\draw (0,0) node[vari]{};
\draw (0,1.5) node[int]{};
\draw (-1,1.75) node[int]{};
\draw (1,1.75) node[int]{};
\end{tikzpicture}
\quad,
\end{equ}
\begin{equ}\label{eq:model building 1}
\cW^{(\eps;0)}_2\<PAM-4new>=2\,
\begin{tikzpicture}[scale=0.7,baseline={(0,0.3)}]
\draw[kernels3] (-1,1) -- (0,0) -- (1,1) ;
\draw[kernels2] (-1,2.5) -- (0,1.5) -- (1,2.5) ;
\draw[rho] (-1,1) node[int]{} -- (-1,1.75);
\draw[rho] (1,1) node[int]{} -- (1,1.75);
\draw[rho] (-1,2.5) node[int]{} -- (-1,1.75);
\draw[rho] (1,2.5) node[int]{} -- (1,1.75);
\draw (0,0) -- (0,1.5);
\draw (0,0) node[vari]{};
\draw (0,1.5) node[int]{};
\draw (-1,1.75) node[int]{};
\draw (1,1.75) node[int]{};
\end{tikzpicture}
\quad - \quad
2\,\begin{tikzpicture}[scale=0.7,baseline={(0,0.3)}]
\draw[kernels2] (-1,1) -- (0,0) -- (1,1) ;
\draw[kernels2] (-1,2.5) -- (0,1.5) -- (1,2.5) ;
\draw[rho] (-1,1) node[int]{} -- (-1,1.75);
\draw[rho] (1,1) node[int]{} -- (1,1.75);
\draw[rho] (-1,2.5) node[int]{} -- (-1,1.75);
\draw[rho] (1,2.5) node[int]{} -- (1,1.75);
\draw (0,0) -- (0,1.5);
\draw (0,0) node[vari]{};
\draw (0,1.5) node[int]{};
\draw (-1,1.75) node[int]{};
\draw (1,1.75) node[int]{};
\end{tikzpicture}
\quad.
\end{equ}
The same argument as above provides the bound \eqref{eq:model-conv3} for $\cW^{(\eps;4)}\<PAM-4new>$, $\cW^{(\eps;2)}\<PAM-4new>$, and $\cW^{(\eps;0)}_1\<PAM-4new>$.
However, for $\cW^{(\eps;0)}_2\<PAM-4new>$ there is no corresponding term in the translation invariant case and one has to work a bit more. First of all, the $\eps=0$ meaning of \eqref{eq:model building 1} is not obvious, since both terms on their own diverge in the $\eps\to 0$ limit.
Introducing
$\begin{tikzpicture}[baseline={(0,-0.1)}]
\draw[kerneldiff] (-0.5,0) -- (0,0) ;
\end{tikzpicture}=
\begin{tikzpicture}[baseline={(0,-0.1)}]
\draw[kernels3] (-0.5,0) -- (0,0);
\end{tikzpicture}-\begin{tikzpicture}[baseline={(0,-0.1)}]
\draw[kernels2] (-0.5,0) -- (0,0) ;
\end{tikzpicture}$, we can rewrite \eqref{eq:model building 1} as
\begin{equ}\label{eq:model building 2}
\cW^{(\eps;0)}_2\<PAM-4new>=2\,
\begin{tikzpicture}[scale=0.7,baseline={(0,0.3)}]
\draw[kerneldiff] (-1,1) -- (0,0);
\draw[kernels3] (0,0) -- (1,1) ;
\draw[kernels2] (-1,2.5) -- (0,1.5) -- (1,2.5) ;
\draw[rho] (-1,1) node[int]{} -- (-1,1.75);
\draw[rho] (1,1) node[int]{} -- (1,1.75);
\draw[rho] (-1,2.5) node[int]{} -- (-1,1.75);
\draw[rho] (1,2.5) node[int]{} -- (1,1.75);
\draw (0,0) -- (0,1.5);
\draw (0,0) node[vari]{};
\draw (0,1.5) node[int]{};
\draw (-1,1.75) node[int]{};
\draw (1,1.75) node[int]{};
\end{tikzpicture}
\quad + \quad
2\,\begin{tikzpicture}[scale=0.7,baseline={(0,0.3)}]
\draw[kerneldiff] (-1,1) -- (0,0);
\draw[kernels2] (0,0) -- (1,1) ;
\draw[kernels2] (-1,2.5) -- (0,1.5) -- (1,2.5) ;
\draw[rho] (-1,1) node[int]{} -- (-1,1.75);
\draw[rho] (1,1) node[int]{} -- (1,1.75);
\draw[rho] (-1,2.5) node[int]{} -- (-1,1.75);
\draw[rho] (1,2.5) node[int]{} -- (1,1.75);
\draw (0,0) -- (0,1.5);
\draw (0,0) node[vari]{};
\draw (0,1.5) node[int]{};
\draw (-1,1.75) node[int]{};
\draw (1,1.75) node[int]{};
\end{tikzpicture}
\quad.
\end{equ}
We now argue that these terms are each well-defined in the $\eps\to0$ limit and satisfy the bounds \eqref{eq:model-conv2}.
In this computation we simplify the notation by only keeping track of the orders of the singularities as defined in \eqref{eq:kernel-singu} of (combination of) edges, which we denote by writing the order on the given edge. For example, instead of 
\begin{tikzpicture}[baseline={(0,-0.1)}]
\draw (-0.5,0) -- (0,0);
\end{tikzpicture}
we write
\begin{tikzpicture}[baseline={(0,-0.1)}]
\draw (-0.5,0) -- (1,0);
\draw (0.25,0) node[labl] {\tiny $0,-1$};
\end{tikzpicture},
or instead of
\begin{tikzpicture}[baseline={(0,-0.1)}]
\draw[kerneldiff] (0,0) -- (1,0);
\draw[rho] (-0.5,0) -- (0,0) node[int]{};
\end{tikzpicture}
we write
\begin{tikzpicture}[baseline={(0,-0.1)}]
\draw (-0.5,0) -- (1,0);
\draw (0.25,0) node[labl] {\tiny $-2,0$};
\end{tikzpicture}.
The graphs are then manipulated by using the rules in Lemma \ref{lem:kernel-convolution}.
Both terms in \eqref{eq:model building 2} can be bounded by using Lemma \ref{lem:kernel-convolution} (ii) and (iv) in the first inequality and (i) in the second one: 
\begin{equ}
\begin{tikzpicture}[scale=0.9,baseline={(0,0.8)}]
\draw (-1,1) -- (0,0) -- (1,1) node[int] {} -- (0,2) -- (-1,1) node[int] {};
\draw (0,0) node[vari]{} -- (0,2) node[int]{} ;
\draw (0,1) node[labl] {\tiny $0,-1$};
\draw (-0.6,0.5) node[labl] {\tiny $-2,0$};
\draw (-0.6,1.5) node[labl] {\tiny $0,-2$};
\draw (0.6,1.5) node[labl] {\tiny $0,-2$};
\draw (0.6,0.5) node[labl] {\tiny $0,-2$};
\end{tikzpicture}
\quad\lesssim
\begin{tikzpicture}[scale=0.9,baseline={(0,0.8)}]
\draw (0,0) to[out=45,in=-45,distance=1.5cm] (0,2);
\draw (0,0) to[out=135,in=-135,distance=1.5cm] (0,2) node[int]{};
\draw (0,0) node[vari]{} -- (0,2) node[int]{} ;
\draw (0,1) node[labl] {\tiny $0,-1$};
\draw (-0.6,0.5) node[labl] {\tiny $-1,0$};
\draw (0.6,0.5) node[labl] {\tiny $0,-1$};
\end{tikzpicture}
\quad\lesssim
\begin{tikzpicture}[scale=0.9,baseline={(0,0.8)}]
\draw (0,0) node[vari]{} -- (0,2) node[int]{} ;
\draw (0,1) node[labl] {\tiny $-1,-2$};
\end{tikzpicture}
\quad\lesssim
\big|{\log|x|_\d}\big|.
\end{equ}
Keeping in mind that since $|\,\<PAM-4new>\,|>-1$ , it suffices to consider the case $|x|_\d\lesssim\|x-z\|$, this verifies \eqref{eq:model-conv3} (notice that since $k=0$ in this case, the $L^2$-scalar product simplifies to product of real numbers).

We now turn to the proof of convergence, for which we use Proposition \ref{prop:triv-but-nice}. The auxiliary approximations $\cW^{(\eps,\delta)}$ (and from them the models $(\Pi^{\eps,\delta},\Gamma^{\eps,\delta})$) are built as follows.
The translation invariant part is simply the BPHZ model associated to the system
\begin{equ}
Y=\rho_\delta\ast\bar K\ast\rho_\eps\ast\xi,\qquad v=\rho_\delta\ast\bar K\ast\big(v|\nabla Y|^2-2\nabla v\cdot\nabla Y\big).
\end{equ}
On `new' symbols we proceed similarly as above, by simply replacing the appropriate red edges in the trees describing the functions $\cW^{(\eps,\delta)}$ by green ones.

For any fixed $\delta$, the functions $\cW^{(\eps,\delta)}$ are uniformly smooth over $\eps\in[0,1]$ in the $x$ variable (in fact with any $\cC^k$ norm bounded by a power of $\delta$).
As a consequence, the condition $a_{\eps,\delta}\to a_{0,\delta}$ as $\eps\to 0$ of Proposition \ref{prop:triv-but-nice} is met.
The other condition will be an easy consequence of the already established a priori estimates and the bound
\begin{equ}
\Big|\int \rho_\delta(x-u)\big(\nabla_1 G(u,v)-\nabla_1 G(x,v)\big)\,du\Big|\lesssim\delta^{\theta}\|x-v\|^{-2-\theta}+\bone_{|x|_\d\leq\delta}\|x-v\|^{-2}
\end{equ}
for any $\theta\in(0,1)$.
Recall that in the translation invariant case the analogous bounds take the simpler form
\begin{equ}
\Big|\int \rho_\delta(u-v)\big(\d_1^k\bar K(v,w)-\d_1^k\bar K(u,w)\big)\,dv\Big|\lesssim\delta^{\theta}\|x-w\|^{-1-|k|-\theta}
\end{equ}
for any $\theta\in(0,1)$ and $|k|=0,1$.
Therefore, for any $\tau$, $k$, and $i$ the function $\cW^{(\eps,\delta;k)}_i\tau-\cW^{(\eps,0;k)}_i\tau$ has the following form: it can be represented by a tree similar to $\cW^{(0,0;k)}_i\tau$, with each edge admitting the same bound as the corresponding edge in $\cW^{(0,0;k)}_i\tau$, and at least one edge is further multiplied with either a factor $\delta^{-\theta}$ or $\bone_{|x|_\d\leq\delta}$.
Therefore a bound similar of the form \eqref{eq:model-conv4} or \eqref{eq:model-conv5} is satisfied, and the convergence $(\Pi^{\eps,\delta},\Gamma^{\eps,\delta})\to(\Pi^{\eps,0},\Gamma^{\eps,0})$ as $\delta\to 0$, uniformly over $\eps\in[0,1]$, follows.
The convergence of the models $(\Pi^{+,\eps},\Gamma^\eps)$, and $(\Pi^{-,\eps},\Gamma^\eps)$ then immediately follows from Remark \ref{rem:indicator}.
\end{proof}
 
\subsection{\texorpdfstring{$\Phi^4_3$}{Phi\^4\_3}} 
We proceed similarly to the above.
Take the regularity structure $\scT_0=(T_0,G_0,A_0)$ and the corresponding sequence of renormalised models $(\Pi^\eps,\Gamma^\eps)$ built for the
$\Phi^4_3$ equation in \cite{H0}.
We again use the periodic models with fundamental domain $(-2,2)^3$ and assume the integration kernel being $\bar\cK$, the truncated heat kernel in $1+3$ dimensions, satisfying \cite[Ass.~5.1]{H0}.
We then take a new regularity structure $\scT=(T,G,A)$ by adding five new symbols: we set $T=T_0\oplus\scal{\cF_\d}$ with $\cF_\d=\{\<PAM-2new>,\<Phi-5new>,\<Phi-3new>,\<Phi-4new>,\<PAM-2new>X\}$.
The homogeneities of the new symbols are given by $|\,\<PAM-2new>\,|=-1-2\kappa$, $|\,\<Phi-5new>\,|=-1/2-5\kappa$, $|\,\<Phi-3new>\,|=-1/2-3\kappa$, $|\,\<Phi-4new>\,|=-4\kappa$, $|\,\<PAM-2new>X\,|=-2\kappa$.

For $\eps>0$ we then extend $(\Pi^\eps,\Gamma^\eps)$ to a family of models $(\Pi^{\eps,a},\Gamma^\eps)$, $a\in[0,\infty]$, by setting $\Pi^{\eps,a}\tau=\Pi^\eps\tau$ for $\tau\in T_0$ and
\begin{equs}
\big(\Pi^{\eps,a}_x\<PAM-2new>\big)(y)&=\Psi_{\eps,a}^2(y)-\E\Psi_{\eps,a}^2(y),
\\
\big(\Pi^{\eps,a}_x\<Phi-5new>\big)(y)&=\Big(\big(\Pi^\eps_x\<Phi-I3>\big)(y)\Big)\Big(\big(\Pi^{\eps,a}_x\<PAM-2new>\big)(y)\Big)-3 \ell_\eps(\<Phi-4-Small>)\big(\Pi^\eps_x\<lolly>\big)(y),
\\
\big(\Pi^{\eps,a}_x\<Phi-3new>\big)(y)&=\Psi_{\eps,a}^3(y)-3\big(\E\Psi_{\eps,a}^2(y)\big)\Psi_{\eps,a},
\\
\big(\Pi^{\eps,a}_x\<Phi-4new>\big)(y)&=\Big(\big(\Pi^{\eps}_x\<Phi-I2>\big)(y)\Big)\Big(\big(\Pi^{\eps,a}_x\<PAM-2new>\big)(y)\Big)- \ell_\eps(\<Phi-4-Small>),
\\
\big(\Pi^{\eps,a}_x\<PAM-2new>X\big)(y)&=\Big(\big(\Pi^{\eps,a}_x\<PAM-2new>\big)(y)\Big)(y-x).
\end{equs}
Here the functions $\Psi_{\eps,a}$ are defined as in \eqref{eq:Psi2}, using the Robin heat kernels $\cG_{3a}$ that are discussed in Section \ref{sec:robin}.
The extension of $\Gamma^\eps$ to the new symbols is again straightforward.
We also define the models $(\Pi^{+,\eps,a},\Gamma^\eps)$, $(\Pi^{-,\eps,a},\Gamma^\eps)$ exactly as before.
 
\begin{lemma}\label{lem:Phi model convergence}
Take a sequence $(a_\eps)_{\eps\in[0,1]}\subset[0,\infty]$ that is convergent in the natural topology.
Then the three sequences of models $(\Pi^{\eps,a_\eps},\Gamma^\eps)$, $(\Pi^{+,\eps,a_\eps},\Gamma^\eps)$, and $(\Pi^{-,\eps,a_\eps},\Gamma^\eps)$, converge.
\end{lemma}
 
\begin{proof}
Following the argument in the proof of Lemma \ref{lem:PAM model convergence}, we only really need to establish a priori bounds for components of $\Pi^{\eps,a}\tau$
whose counterpart for the corresponding translation invariant symbol
is completely cancelled by the renormalisation procedure.

For the $\Phi^4_3$ equation there are altogether two such terms: the non-recentered parts of the $0$-th chaos component of $\<Phi-4new>$ and the first chaos component of $\<Phi-5new>$.
For the first one, the computation is very similar to the previous proof, and we use the graphical notation from therein, but this time
\begin{tikzpicture}[baseline={(0,-0.1)}]
\draw (-0.5,0) -- (0,0);
\end{tikzpicture}
denoting a factor $\bar \cK(u-v)$ and
\begin{tikzpicture}[baseline={(0,-0.1)}]
\draw[kernels3] (-0.5,0) -- (0,0);
\end{tikzpicture}
denoting a factor $\cG_{3a_\eps}(u,v)$.
The function $\cW^{(\eps;0)}_2\<Phi-4new>$ is then given by
\begin{equ}\label{eq:model building 3}
\cW^{(\eps;0)}_2\<Phi-4new>=2\,
\begin{tikzpicture}[scale=0.7,baseline={(0,0.3)}]
\draw[kerneldiff] (-1,1) -- (0,0);
\draw[kernels3] (0,0) -- (1,1) ;
\draw (-1,2.5) -- (0,1.5) -- (1,2.5) ;
\draw[rho] (-1,1) node[int]{} -- (-1,1.75);
\draw[rho] (1,1) node[int]{} -- (1,1.75);
\draw[rho] (-1,2.5) node[int]{} -- (-1,1.75);
\draw[rho] (1,2.5) node[int]{} -- (1,1.75);
\draw (0,0) -- (0,1.5);
\draw (0,0) node[vari]{};
\draw (0,1.5) node[int]{};
\draw (-1,1.75) node[int]{};
\draw (1,1.75) node[int]{};
\end{tikzpicture}
\quad + \quad
2\,\begin{tikzpicture}[scale=0.7,baseline={(0,0.3)}]
\draw[kerneldiff] (-1,1) -- (0,0);
\draw (0,0) -- (1,1) ;
\draw (-1,2.5) -- (0,1.5) -- (1,2.5) ;
\draw[rho] (-1,1) node[int]{} -- (-1,1.75);
\draw[rho] (1,1) node[int]{} -- (1,1.75);
\draw[rho] (-1,2.5) node[int]{} -- (-1,1.75);
\draw[rho] (1,2.5) node[int]{} -- (1,1.75);
\draw (0,0) -- (0,1.5);
\draw (0,0) node[vari]{};
\draw (0,1.5) node[int]{};
\draw (-1,1.75) node[int]{};
\draw (1,1.75) node[int]{};
\end{tikzpicture}
\quad.
\end{equ}
The power counting also changes as now we have $|\frs|=5$ and \begin{tikzpicture}[baseline={(0,-0.1)}]
\draw[kerneldiff] (0,0) -- (1,0);
\draw[rho] (-0.5,0) -- (0,0) node[int]{};
\end{tikzpicture}
now admits the bound
\begin{tikzpicture}[baseline={(0,-0.1)}]
\draw (-0.5,0) -- (1,0);
\draw (0.25,0) node[labl] {\tiny $-3,0$};
\end{tikzpicture}.
Consequently, we can bound both terms in \eqref{eq:model building 3} using Lemma \ref{lem:kernel-convolution} (ii) and (iv) in the first inequality and (i) in the second one: 
\begin{equ}
\begin{tikzpicture}[scale=0.9,baseline={(0,0.8)}]
\draw (-1,1) -- (0,0) -- (1,1) node[int] {} -- (0,2) -- (-1,1) node[int] {};
\draw (0,0) node[vari]{} -- (0,2) node[int]{} ;
\draw (0,1) node[labl] {\tiny $0,-3$};
\draw (-0.6,0.5) node[labl] {\tiny $-3,0$};
\draw (-0.6,1.5) node[labl] {\tiny $0,-3$};
\draw (0.6,1.5) node[labl] {\tiny $0,-3$};
\draw (0.6,0.5) node[labl] {\tiny $0,-3$};
\end{tikzpicture}
\quad\lesssim
\begin{tikzpicture}[scale=0.9,baseline={(0,0.8)}]
\draw (0,0) to[out=45,in=-45,distance=1.5cm] (0,2);
\draw (0,0) to[out=135,in=-135,distance=1.5cm] (0,2) node[int]{};
\draw (0,0) node[vari]{} -- (0,2) node[int]{} ;
\draw (0,1) node[labl] {\tiny $0,-3$};
\draw (-0.6,0.5) node[labl] {\tiny $-1,0$};
\draw (0.6,0.5) node[labl] {\tiny $0,-1$};
\end{tikzpicture}
\quad\lesssim
\begin{tikzpicture}[scale=0.9,baseline={(0,0.8)}]
\draw (0,0) node[vari]{} -- (0,2) node[int]{} ;
\draw (0,1) node[labl] {\tiny $-1,-4$};
\end{tikzpicture}
\quad\lesssim
\big|{\log|x|_\d}\big|.
\end{equ}
Concerning the non-recentered term in the first chaos component of $\<Phi-5new>$, it can be written as
\begin{equ}
\cW^{(\eps;1)}_2\<Phi-5new>=6\,
\begin{tikzpicture}[scale=0.7,baseline={(0,0.3)}]
\draw[kerneldiff] (-1,1) -- (0,0);
\draw[kernels3] (0,0) -- (1,1) ;
\draw (-1,2.5) -- (0,1.5) -- (1,2.5) ;
\draw (0,1.5) -- (0,3) node [noise]{};
\draw[rho] (-1,1) node[int]{} -- (-1,1.75);
\draw[rho] (1,1) node[int]{} -- (1,1.75);
\draw[rho] (-1,2.5) node[int]{} -- (-1,1.75);
\draw[rho] (1,2.5) node[int]{} -- (1,1.75);
\draw (0,0) -- (0,1.5);
\draw (0,0) node[vari]{};
\draw (0,1.5) node[int]{};
\draw (-1,1.75) node[int]{};
\draw (1,1.75) node[int]{};
\end{tikzpicture}
\quad + \quad
6\,\begin{tikzpicture}[scale=0.7,baseline={(0,0.3)}]
\draw[kerneldiff] (-1,1) -- (0,0);
\draw (0,0) -- (1,1) ;
\draw (-1,2.5) -- (0,1.5) -- (1,2.5) ;
\draw (0,1.5) -- (0,3) node [noise]{};
\draw[rho] (-1,1) node[int]{} -- (-1,1.75);
\draw[rho] (1,1) node[int]{} -- (1,1.75);
\draw[rho] (-1,2.5) node[int]{} -- (-1,1.75);
\draw[rho] (1,2.5) node[int]{} -- (1,1.75);
\draw (0,0) -- (0,1.5);
\draw (0,0) node[vari]{};
\draw (0,1.5) node[int]{};
\draw (-1,1.75) node[int]{};
\draw (1,1.75) node[int]{};
\end{tikzpicture}
\quad.
\end{equ}
Therefore we have (using two instances of
\begin{tikzpicture}[baseline={(0,-0.1)}]
\draw (0,0) node[vari]{};
\end{tikzpicture}
to denote the two different running variables $x$ and $\bar x$)
by applying Lemma \ref{lem:kernel-convolution} (ii) and (iv) in the second inequality, (i) in the third one, and (iii) in the fourth and fifth one
\begin{equs}\label{eq:graph-calc-2}
\big|\scal{\cW^{(\eps;1)}_0\<Phi-5new>\,(x;\cdot),\cW^{(\eps;1)}_0\<Phi-5new>\,(\bar x;\cdot)}\big|  &  \lesssim\,
\begin{tikzpicture}[scale=0.9,baseline={(0,0.8)}]
\draw (-1,1) -- (0,0) -- (1,1) node[int] {} -- (0,2) -- (-1,1) node[int] {};
\draw (0,0) node[vari]{} -- (0,2) node[int]{} ;
\draw (0,1) node[labl] {\tiny $0,-3$};
\draw (-0.6,0.5) node[labl] {\tiny $-3,0$};
\draw (-0.6,1.5) node[labl] {\tiny $0,-3$};
\draw (0.6,1.5) node[labl] {\tiny $0,-3$};
\draw (0.6,0.5) node[labl] {\tiny $0,-3$};
\draw (2,1) -- (3,0) -- (4,1) node[int] {} -- (3,2) -- (2,1) node[int] {};
\draw (0,2) -- (1.5,2) node[int] {} -- (3,2);
\draw (0.75,2) node[labl] {\tiny $0,-3$};
\draw (2.25,2) node[labl] {\tiny $0,-3$};
\draw (3,0) node[vari]{} -- (3,2) node[int]{} ;
\draw (3,1) node[labl] {\tiny $0,-3$};
\draw (2.4,0.5) node[labl] {\tiny $-3,0$};
\draw (2.4,1.5) node[labl] {\tiny $0,-3$};
\draw (3.6,1.5) node[labl] {\tiny $0,-3$};
\draw (3.6,0.5) node[labl] {\tiny $0,-3$};
\end{tikzpicture}
\\&\lesssim\begin{tikzpicture}[scale=0.9,baseline={(0,0.8)}]
\draw (0,0) to[out=45,in=-45,distance=1.5cm] (0,2);
\draw (0,0) to[out=135,in=-135,distance=1.5cm] (0,2) node[int]{};
\draw (0,0) node[vari]{} -- (0,2) node[int]{} ;
\draw (0,1) node[labl] {\tiny $0,-3$};
\draw (-0.6,0.5) node[labl] {\tiny $-1,0$};
\draw (0.6,0.5) node[labl] {\tiny $0,-1$};
\draw (0,2) -- (3,2);
\draw (1.5,2) node[labl] {\tiny $0,-1$};
\draw (3,0) to[out=45,in=-45,distance=1.5cm] (3,2);
\draw (3,0) to[out=135,in=-135,distance=1.5cm] (3,2) node[int]{};
\draw (3,0) node[vari]{} -- (3,2) node[int]{} ;
\draw (3,1) node[labl] {\tiny $0,-3$};
\draw (2.4,0.5) node[labl] {\tiny $-1,0$};
\draw (3.6,0.5) node[labl] {\tiny $0,-1$};
\end{tikzpicture}
\\
&\lesssim
\quad
\begin{tikzpicture}[scale=0.9,baseline={(0,0.4)}]
\draw (0,0) node[vari]{} -- (0,1) node[int]{} ;
\draw (0,0.5) node[labl] {\tiny $-1+\eta,-4$};
\draw (0,1) -- (3,1);
\draw (1.5,1) node[labl] {\tiny $0,-1$};
\draw (3,0) node[vari]{} -- (3,1) node[int]{} ;
\draw (3,0.5) node[labl] {\tiny $-1+\eta,-4$};
\end{tikzpicture}
\\
&\lesssim |x|_\d^{\eta}\quad
\begin{tikzpicture}[scale=0.9,baseline={(0,-0.1)}]
\draw (0,0) node[vari]{} -- (2,0) node[int]{} -- (4,0) node[vari]{};
\draw (1,0) node[labl] {\tiny $0,-1$};
\draw (3,0) node[labl] {\tiny $-1+\eta,-4$};
\end{tikzpicture}
\\
&\lesssim |x|_\d^{\eta}|\bar x|_\d^\eta\quad
\begin{tikzpicture}[scale=0.9,baseline={(0,-0.1)}]
\draw (0,0) node[vari]{} -- (2,0) node[vari]{};
\draw (1,0) node[labl] {\tiny $0,-1$};
\end{tikzpicture}
\\
&=|x|_\d^{\eta}|\bar x|^\eta_\d\|x-\bar x\|^{-1}\;.
\end{equs}
Here $\eta<0$ is arbitrary and introduced only to avoid equalities in the conditions of Lemma~\ref{lem:kernel-convolution}.
Since $|x|_\d\lesssim\|x-z\|$, $|\bar x|_\d\lesssim\|\bar x-z\|$ can be assumed, this proves \eqref{eq:model-conv3}. The remaining arguments are then exactly as in the proof of Lemma \ref{lem:PAM model convergence}.
\end{proof}

\section{Proofs of the main results}\label{sec:proofs}

The convergence results are obtained via \cite[Thm.~5.6]{GH17}
(or a slight modification of it),
the notation of which we shall use without repeating them here.
Let us however summarise the step which is the main difference
to most of the literature and the source of the boundary renormalisation.
This concerns setting up the convolution operator of a fixed point problem.
Given a modelled distribution $f$ -- think of the right-hand side of an abstract equation -- we will define an abstract convolution $\scP_\eps(f),$
where the operator $\scP_\eps$ is built from the following ingredients:
\begin{claim}
\item A kernel $\bar \cK$ that is $2$-smoothing in the sense of \cite[Ass.~5.1]{H0} -- this will always be the truncated (translation invariant) heat kernel for us;
\item A model $(\Pi^\eps,\Gamma^\eps)$ admissible for $\bar\cK$ -- we will always choose from the models discussed above;
\item A remainder kernel $Z_\eps$ that is smooth locally, but not necessarily globally up to the boundary;
\item A distribution $\zeta$ that agrees with $\cR^\eps f$ away from the spatial boundary $\R\times\d D$.
\end{claim}
Given these ingredients, \cite[Lem~4.12 \& 4.16]{GH17} yields a modelled distribution $\scP_\eps(f)$ that satisfies 
\begin{equ}
\cR^\eps\big(\scP_\eps(f)\big)(x)=\scal{(K+Z_\eps)(x,\cdot),\zeta}.
\end{equ}
Moreover, $\scP_\eps$ is locally Lipschitz continuous in the natural topologies.
In particular, to get convergence results, one has to ensure that the choice of $\zeta$ is sufficiently stable in the $\eps\to 0$ limit.
This is far from obvious: 
for the type of weighted spaces we wish to use, the reconstruction theorem
\cite[Thm~3.10]{H0} can only be applied locally in the interior of 
$\R\times D$, which therefore yields a distribution $\cR^\eps f$ only on the interior.
Extending it to a distribution $\zeta$ \emph{on the whole space},
may not be obvious or indeed even unique.
Of course as long as $\eps>0$, there is trivial canonical extension, but it is \emph{not} true in general
that this choice will converge to a limit.
To ensure convergence, we modify the canonical choice by certain Dirac masses on the boundary, which in turn has the effect of changing the boundary conditions (see Section \ref{sec:lin} or \cite[Rem~1.5]{HP19}).
In line with \cite{GH17}, we will also use the notation $\zeta=\hat\cR^\eps f$, to emphasise its relation to the `usual' reconstruction operator from \cite{H0}.

Let us recall that in the weighted spaces of modelled distribution
used in \cite{GH17}
weights are given by triples $w=(\eta,\sigma,\mu)\in\R^3$ describing singularities near the temporal and spatial boundaries and their intersection, respectively.
It will often be convenient to use the notation $(\eta)_3=(\eta,\eta,\eta)$.
The lowest homogeneity symbol appearing in the target space of a given space of modelled distributions will usually be denoted by $\alpha$.
There are two general situations when the issue of choosing $\zeta$ can be avoided.
First, as long as $\sigma\wedge\alpha>-1$, there \emph{is} a canonical choice for $\hat\cR$, see \cite[Thm~4.9]{GH17}, which furthermore agrees with the natural extension for smooth models.
Second, if $\sigma\wedge\alpha>-2$ and $\bar\cK+Z_\eps$ is a Dirichlet kernel, then $\scP_\eps (f)$ can be defined \emph{without} defining $\hat\cR f$, see Lemma \ref{lem:Dirichlet-convolution} below.

Finally, we address the above mentioned `slight modification' of \cite[Thm~5.6]{GH17} that we will use in the case of $\Phi^4_3$,
namely that we do not consider the remainder part of the kernel fixed
as $\eps$ varies.
It is straightforward to check that the integration operation against smooth remainders with singularities at the boundary \cite[Lem~4.16]{GH17} is Lipschitz continuous with respect to the natural norm $\|\cdot\|_{\scZ_{\beta,P}}$ on the space of remainder kernels.
One then gets the following simple extension of \cite[Thm.~5.6]{GH17}.
\begin{proposition}\label{prop:tiny extension}
Assume the setting of \cite[Thm.~5.6]{GH17},
with however allowing for $Z\in\scZ_{\beta,P}$ to depend on $\eps\in[0,1]$.
Suppose that 
\begin{equ}
\|Z^\eps-Z^0\|_{\scZ_{\beta,P}}\rightarrow0.
\end{equ}
Then the conclusions of \cite[Thm.~5.6]{GH17} hold.
\end{proposition}

\subsection{PAM} 
As a warm-up example of the above setup without boundary renormalisation, define
the modelled distribution
\begin{equ}
\bY_\eps=-\scP_\eps(\bone_D\<0>),
\end{equ}
with $Z_\eps=Z$ being the Neumann remainder (see either \cite[Ex~4.15]{GH17}, \cite[App~A]{HP19}, or Section~\ref{sec:robin} above) and $\zeta=\bone_D\xi_\eps$.
Here and below $\bone_D$ is the indicator function of $D$.
Note that one can easily make sense (probabilistically) of the product $\bone_D\xi$,
and one moreover has $\bone_D\xi_\eps\rightarrow\bone_D\xi$ in $\cC^{-3/2-\kappa}$.
It is also easy to see that $\bone_D\<0>\in\cD_P^{\kappa,(-3/2-\kappa)_3}$.
Therefore,
setting $\bw_\eps=\scD\bY_\eps$,
we have that $\bw_\eps \to\bw_0\in \cD_P^{1+\kappa,(-1/2-\kappa)_3}$ and $\cR^\eps\bw_\eps=\Psi_\eps$, for all $\eps\in[0,1]$.

This allows us to set up the abstract fixed point problem corresponding to \eqref{e:etPAM} as follows.
Let $\bone_D^+$ denote the indicator function of $\R_+\times D$.
Let $Z_\eps=Z$ be the either the Dirichlet remainder (for Theorem \ref{thm:PAM-D})
or the Neumann remainder (for Theorem \ref{thm:PAM-N}).
Let furthermore $\bi_\eps$ be the modelled distribution
responsible for the contribution of the initial condition:
it is simply obtained by taking the Taylor polynomial lift of
the function $t,x\mapsto\scal{(\bar\cK+Z)_t(x,\cdot),u_0e^{Y_\eps}}$.
The abstract counterpart of \eqref{e:etPAM} then reads as
\begin{equ}\label{eq:PAM-abstract}
\bv_\eps=\scP_\eps\big(\bone_D^+ \bv_\eps|\bw_\eps|^2\big)
-2\scP_\eps'\big(\bone_D^+\scD\bv_\eps\cdot\bw_\eps\big)+\bi_\eps.
\end{equ}
The operators $\scP_\eps$ and $\scP_\eps'$ will only differ in the choice of $\hat\cR^\eps$. In fact as we will shortly see, for the latter this choice will be trivial.

We now specify the spaces of modelled distributions on which \eqref{eq:PAM-abstract} can be shown to be a contraction and therefore to have a unique solution.
Denote by $T_0^{\geq \alpha}$ the subspace of $T_0$ generated by basis vectors of homogeneity greater or equal to $\alpha$
(note that no new symbols are included in these spaces)
and by $\bar T$ the polynomial subspace.
Fix the exponents $\gamma=3/2+3\kappa$, $\gamma_0=\kappa$,
$\alpha=0$, $\alpha_0=-1-2\kappa$, $\alpha_1=-1/2-3\kappa$, the sectors
\begin{equ}
V=\cI(T_0^{\geq \alpha_0})+\bar T,\quad V_0=T_0^{\geq \alpha_0},\quad V_1=T_0^{\geq \alpha_1},
\end{equ}
as well as
\begin{equs}
\eta &=-1/2+3\kappa&\quad \sigma &=1-3\kappa&\quad \mu &=-1/2+3\kappa;\\
\eta_0 &=-3/2+\kappa,&\quad \sigma_0 &=-1-2\kappa,&\quad \mu_0 &=-3/2+\kappa;\\
\eta_1 &=-2+\kappa,&\quad \sigma_1 &=-1/2-4\kappa,&\quad \mu_1 &=-2+2\kappa.
\end{equs}
We use the shorthand $w=(\eta,\sigma,\mu)$ and similarly for $w_1$ and $w_2$.

It is easy to verify that, by \cite[Lem~5.4]{GH17}, $\bi_\eps\to\bi_0$ in $\cD^{\gamma,w}(V)$, provided $\kappa$ is sufficiently small.
It also follows from \cite[Lem.~4.3]{GH17} that the mappings
\begin{equ}
f\mapsto f|\bw_\eps|^2,\quad\quad f\mapsto \scD f\cdot\bw_\eps
\end{equ}
are globally Lipschitz continuous from $\cD^{\gamma,w}(V)$
to $\cD^{\gamma_0,w_0}(V_0)$ and
$\cD^{\gamma_0,w_1}(V_1)$, respectively.
Furthermore, since $\sigma_1\wedge\alpha_1>-1$,
\cite[Thm~4.9]{GH17} provides a continuous extension $\hat \cR^\eps$ of the reconstruction operator on $\cD^{\gamma_0,w_1}(V_1)$,
which agrees with the canonical extension for $\eps>0$.
This now completes the definition of $\scP_\eps'$ in \eqref{eq:PAM-abstract}.

Furthermore, in the case of Dirichlet boundary conditions, $\scP_\eps$ is also well-defined, thanks to Lemma \ref{lem:Dirichlet-convolution}.
By \cite[Thm~5.6]{GH17} one can conclude that,
provided $\kappa>0$ is sufficiently small, it has a unique local in time solution $\bv_\eps$.
Furthermore, since the right-hand side is globally Lipschitz continuous (in fact, linear), one can easily extend the local solutions to the whole time interval $[0,1]$.
Now the proof can be easily concluded.

\begin{proof}[Proof of Theorem \ref{thm:PAM-D}]
Invoking the usual arguments concerning the renormalisation in the bulk (see e.g. \cite[Sec~9]{H0}),
one sees that the function $v_\eps:=\cR^\eps\bv_\eps$ solves \eqref{e:etPAM} with $0$ Dirichlet boundary conditions.
Hence $v_\eps e^{-Y_\eps}$ solves solves \eqref{e:ePAM},
trivially also with $0$ Dirichlet boundary conditions,
and therefore it coincides with $u_\eps^{\Dir}$.
By \cite[Thm.~5.6]{GH17}, $\bv_\eps\to\bv_0$, the latter
clearly does not depend on $\rho$,
and so the limit $\lim_{\eps\to0}u_\eps^{\Dir}=(\cR^0\bv_0)e^{-Y_0}=:u^{\Dir}$ is also independent of $\rho$, yielding the claim.
\end{proof}

In the Neumann case, there is no easy way to construct (or avoid constructing, as above) $\hat\cR^\eps$ for the integration operator $\scP_\eps$.
The main remaining step is therefore to build $\hat\cR^\eps$ as a continuous map from
$\cV_\eps:=\cD^{\gamma,w}(V)|\bw_\eps|^2$ to $\CC^{-3/2+\kappa}$,
which is where all previous steps in the article will be used.
First of all, notice that $\bw_\eps$ is of the form
\begin{equ}\label{eq:qwe}
|\bw_\eps|^2(x)=\<PAM-2a>+2\varphi_\eps(x)\cdot\<PAM-1>+|\varphi_\eps(x)|^2
\end{equ}
for some $(\bar T)^3$-valued $\varphi_\eps$.
One however can use the newly added symbols to out regularity structure to write an equivalent function
\begin{equ}\label{eq:rty}
|\tilde\bw_\eps|^2(x):=\<PAM-2new>+\tilde\varphi_\eps(x),
\end{equ}
where $\tilde\varphi_\eps(x)=\E|\Psi_\eps|^2(x)-\ell_\eps(\<PAM-2a-Small>)$.
Although the latter expression seems ill-defined for $\eps=0$,
it is clear that $\tilde\varphi_\eps$ has pointwise limits away from the boundary.
In fact, in the notation of Lemma \ref{lem:reno Psi^2 PAM} we have the expression $\tilde \varphi_\eps(x)=R^2_\eps+R^3_\eps$.
The $|\cdot|^2$ in \eqref{eq:rty} is merely notational and serves only as an analogy with \eqref{eq:qwe}.
One then trivially has $\cR^\eps|\bw_\eps|^2=\cR^\eps|\tilde \bw_\eps|^2$,
in fact one has $\cR^\eps f|\bw_\eps|^2=\cR^\eps f|\tilde\bw_\eps|^2$
for all $f\in\cD^{\gamma,w}(V)$.
Let us therefore define the isometry $\iota_\eps:\cV_\eps\to\tilde \cV_\eps:=\cD^{\gamma,w}(V)|\tilde\bw_\eps|^2$ by replacing the $|\bw_\eps|^2$ factor with $|\tilde\bw_\eps|^2$.
We can conclude that it suffices to construct $\hat \cR^\eps$ on $\tilde \cV_\eps$, and then it can be pulled back to $\cV_\eps$ by $\iota_\eps$.

\begin{lemma}\label{lem:hatR-PAM}
For all $\eps\in[0,1]$ there exist maps $\hat \cR^\eps:\tilde\cV^\eps\to\CC^{-3/2+\kappa}$ with the following properties.
\begin{claim}
\item For all $\eps\in[0,1]$, $g\in\tilde\cV^\eps$, and $\psi\in\cC^\infty_0(\R_+\times D)$ one has
\begin{equ}\label{eq:hatR-PAM1}
\scal{\hat\cR^\eps g,\psi}=\scal{\cR^\eps g ,\psi};
\end{equ} 
\item The bound
\begin{equ}\label{eq:hatR-PAM2}
|\hat\cR^\eps\big(f^\eps|\tilde \bw_\eps|^2\big)-\hat\cR^0\big(f^0|\tilde \bw_0|^2\big)|_{-3/2+\kappa;\,T}
\lesssim\vn{f^\eps;f^0}_{\gamma,w;\,T}+o(1)
\end{equ}
holds as $\eps\to 0$, uniformly in modelled distributions $f^\eps\in\cD^{\gamma,w}(V,\Gamma^\eps)$, $f^0\in\cD^{\gamma,w}(V,\Gamma^0)$ bounded by a constant $C$, and in $T\in(0,1]$;
\item For $\eps>0$, $f\in\cD^{\gamma,w}(V)$, one has the identity
\begin{equ}\label{eq:hatR-PAM3}
\hat\cR^\eps\big(f|\tilde \bw_\eps|^2\big)=
\cR^\eps\big(f|\tilde \bw_\eps|^2\big)
-\big(a_\rho+\tfrac{|{\log \eps}|}{8\pi}\big)\delta_\d (\cR^\eps f),
\end{equ}
where $a_\rho$ is as in Lemma \ref{lem:reno Psi^2 PAM}.
\end{claim}
\end{lemma}
\begin{proof}
By the definition of the model and Lemma \ref{lem:reno Psi^2 PAM}, for $\eps>0$ one has the identity
\begin{equs}
\cR^\eps\big(f|\tilde \bw_\eps|^2\big)
&=\cR^\eps(f\<PAM-2new>)+(\cR^\eps f)\delta_\d  R^1_\eps+(\cR^\eps f)\scR R^2_\eps+(\cR^\eps f)R^3_\eps
\\
&\qquad+(\cR^\eps f)\big(a_\rho+\tfrac{|{\log \eps}|}{8\pi}\big)\delta_\partial.\label{eq:decomposing}
\end{equs}
Therefore it suffices to show that the first four terms on the right-hand side extend in an appropriately continuous way to $\eps=0$.

First note that if $f\in\cD^{\gamma,w}(V)$, then $f\<PAM-2new>\in\cD^{\gamma',w'}$ with $\gamma'=\gamma-1-2\kappa=1/2+\kappa$ and $w'=w+(-1-2\kappa)_3$, that is, $\eta'=-3/2+\kappa$, $\sigma'=-5\kappa$, $\mu'=-3/2+\kappa$.
Although one has $\sigma>-1$, unfortunately $\alpha'=-1-2\kappa<1$, so
there is no canonical reconstruction on $\cD^{\gamma',w'}$ as \cite[Sec~4.3]{GH17} does not apply.
This barrier is exactly what is overcome in \cite[App~C]{HP19}:
since in our models we have a decomposition $\Pi_x^\eps=\Pi_x^{+,\eps}+\Pi_x^{-,\eps}$, by \cite[Thm~C.5]{HP19}, there exists maps
$\hat \cR^\eps_0:\cD^{\gamma',w'}\to\CC^{-3/2+\kappa}$ satisfying \eqref{eq:hatR-PAM1} and the desired continuity properties.
Furthermore, for $\eps>0$ one has the identity
$\hat \cR^\eps_0(f\<PAM-2new>)=\cR^\eps(f\<PAM-2new>),$ therefore indeed extending the first term on the right-hand side of \eqref{eq:decomposing}.

For the other terms we may forget that $f$ is a modelled distribution and consider $\cR^\eps f$ as a generic element $F$ of $\CC^{1-2\kappa,w}_P$.
By Proposition \ref{prop:weighted holder2} (a)  $\Tr_{\R_+\times \d D}$ maps $\CC^{1-2\kappa,w}_P$ continuously into
$\CC^{\sigma,\eta\wedge\mu}_{P_0}(\R_+\times\d D)=\CC^{1-3\kappa,-1/2+3\kappa}_{P_0}(\R_+\times\d D)$.
By Proposition \ref{prop:weighted holder2} (b) multiplication on the latter space with $\CC^{1-\kappa,-\kappa}_{\d^2}(Q_\d)$ is continuous in both arguments to $\CC^{1-3\kappa,(-1/2+3\kappa,-\kappa,-1/2+2\kappa)}_{P_0,\d^2}(\R_+\times\d D)$, which in turn is continuously embedded into $\CC^{-1/2+2\kappa}(\R\times\d D)$ by Proposition \ref{prop:weighted holder2} (c).
Finally, by Proposition \ref{prop:weighted holder extend}, $\CC^{-1/2+2\kappa}(\R\times\d D)$ continuously embeds into 
$\CC^{-3/2+2\kappa}(\R^4)$. This handles the second term on the right-hand side of \eqref{eq:decomposing}.

Next, we claim that for $F\in\CC^{1-2\kappa,w}_P(\R_+\times D)$, and $R\in\CC_\d^{1-\kappa,-1-\kappa}(Q)$ such that $D_{x_1}R=D_{x_2}R=0$, the distribution $(\scR R) F$ defined as
\begin{equ}
\scal{(\scR R) F,\varphi}=\int_{\R_+\times D}R(z)\Big(F\varphi(z)-F\varphi(\pi_\d z)\Big)\,dz
\end{equ}
is meaningful for all $\eps\in[0,1]$, belongs to $\CC^{-3/2+2\kappa}$,  and is a continuous function of both arguments.
First notice that away from the boundaries this is just a continuous function, which furthermore has an integrable blowup at $P_0$.
In particular, the behavior on the set $\{z:|z|_{P_0}\leq2|z|_{P_1}\}$ is easily controlled as in \cite[Sec.~4.3]{GH17}, and so we focus on test functions centered on the region $\{z:|z|_{P_0}\leq2|z|_{P_1}\}$ and supported away from $P_0$.
Take $R$ and $F$ as above with norm $1$, $y\in\{z:|z|_{P_0}\leq2|z|_{P_1}\}$ and a normalised test function $\varphi$.
First consider the case $\lambda\leq |y|_{P_1}/2$.
Then we can simply write
\begin{equ}
|\scal{(\scR R) F,\varphi_y^\lambda}|=|\scal{R F,\varphi_y^\lambda}|
\leq \sup_{z\in\supp\varphi_y^\lambda}|RF(z)|\lesssim |y|^{-1-\kappa}_{P_1}|y|_{P_0}^{-1/2+3\kappa}\leq \lambda^{-3/2+2\kappa},
\end{equ}
which is a bound of required order.
In the case $|y|_{P_1}\leq 2\lambda$ the boundary terms have to be taken into account. One has
\begin{equs}
|\scal{(\scR R) F,\varphi_{y}^\lambda}|&\lesssim
\int_{\supp\varphi_{y}^\lambda}|R(z)|\big|F(z)(\varphi(z)-\varphi(\pi_\d z))+(F(z)-F(\pi_\d z))\varphi(\pi_\d z)\big|\,dz
\\
&\lesssim\int_{\supp\varphi_{y}^\lambda}|z|_{P_1}^{-1-\kappa}
\big(|z|_{P_0}^{-1/2+3\kappa}\lambda^{-6}|z|_{P_1}+|z|_{P_0}^{-3/2+6\kappa}|z|_{P_1}^{1-3\kappa}\lambda^{-5}\big)\,dz
\\
&\lesssim |y|_{P_0}^{-1/2+3\kappa}\lambda^{-1-\kappa}+|y|_{P_0}^{-3/2+6\kappa}\lambda^{-4\kappa}\lesssim \lambda^{-3/2+2\kappa},
\end{equs}
where we have used that since the exponent of $|z|_{P_1}$ is above $-1$, the integral is finite.
Therefore, $(\scR R)F$ behaves as required.

Finally, for the fourth term, it suffices to notice that $\CC^{1-2\kappa,w}_P$ embeds continuously into $\CC^{-1/2+3\kappa}(\R^4)$,
on which by Proposition \ref{prop:weighted holder mult} multiplication with $\CC^{1-\kappa,-1-\kappa}_{\d^2}(Q)$ is continuous and maps into $\CC^{-1/2,-3/2+2\kappa}_{\d^2}(Q)$, which in turn by Proposition \ref{prop:weighted holder extend} embeds continuously into $\CC^{-3/2+2\kappa}(\R^4)$.
\end{proof}
Using $\hat\cR^\eps$ from Lemma \ref{lem:hatR-PAM} in the operator $\scP_\eps$, the fixed point problem \eqref{eq:PAM-abstract} is completely well-defined and one can use \cite[Thm~5.6]{GH17} to conclude that,
provided $\kappa>0$ is sufficiently small, \eqref{eq:PAM-abstract} has a unique local solution $\bv_\eps$.
Now the main statements follow easily.

\begin{proof}[Proof of Theorem \ref{thm:PAM-N}]
Unlike in the Dirichlet case, the difference between $\hat \cR^\eps$ and the canonical extension of $\cR^\eps$ \emph{does} effect the equation.
Using \eqref{eq:hatR-PAM3}, we have that the 
function $v_\eps:=\cR^\eps\bv_\eps$ do not exactly satisfy \eqref{e:etPAM},
but rather 
\begin{equ}
(\partial_t -\Delta )v_\eps
=v_\eps(|\Psi_\eps|^2-C_\eps)- 2\nabla v_\eps\cdot\Psi_\eps
-\big(a_\rho+\tfrac{|{\log \eps}|}{8\pi}\big)\delta_\d  v_\eps,
\end{equ}
with $0$ Neumann boundary condition.
That is however equivalent to \eqref{e:etPAM} with boundary condition
\begin{equ}
\partial_n v_\eps=-(a_\rho+\tfrac{|{\log \eps}|}{8\pi})v_\eps.
\end{equ}
Hence, $v_\eps e^{-Y_\eps}$ satisfies \eqref{e:ePAM} with boundary condition
\begin{equ}
\partial_n (v_\eps e^{-Y_\eps})
=-(\partial_n Y_\eps)(v_\eps e^{-Y_\eps})+(\partial_nv_\eps) e^{-Y_\eps}=
-\big(a_\rho+\tfrac{|{\log \eps}|}{8\pi}\big)v_\eps e^{-Y_\eps},
\end{equ}
and so it coincides with $u_\eps$.
The proof is then finished as above.
\end{proof}

\subsection{\texorpdfstring{$\Phi^4_3$}{Phi\^4\_3}}
In the case of the $\Phi^4_3$, one has to keep in mind that the remainder kernel may change as $\eps\to 0$.
As before, we start by defining the modelled distributions,
for $(\eps,a)\in[0,1]\times[0,\infty]$,
\begin{equ}
\bw_{\eps,a}=\scP_{\eps,a}(\bone_D\<0>).
\end{equ}
The reconstruction of $\bone_D\xi_\eps$ is again straightforward. As for $Z_\eps$, in the notation of Lemma \ref{lem:robin}, we choose it to be $Z^{(3a)}$.
As before, it is straightforward that $\bw_{\eps,a}$ belongs to $\cD^{1+3\kappa,(-1/2-\kappa)_3}$, it
depends continuously on $(\eps,a)\in[0,1]\times[0,\infty]$,
and that $\cR^\eps\bw_{\eps,a}=\Psi_{\eps,a}$.

Similarly to Section \ref{sec:restrict}, one can see that $\Psi_{\eps,a}$ can be restricted to the \emph{temporal} hyperplane $P_0=\{t=0\}$
(in fact this is already done in the case without boundaries in \cite{H0}).
Let again be $\bi_{\eps,a}$ be the modelled distribution
responsible for the contribution of the initial condition:
it is simply obtained by taking the Taylor polynomial lift of
the function $t,x\mapsto\scal{(\bar\cK+Z^{(3a)})_t(x,\cdot),u_0-\Psi_{\eps,a}\restriction_{P_0}}$.
Equation \eqref{eq:Phi4-lin} then motivates to consider the fixed point problems
\begin{equ}\label{eq:Phi-abstract}
\bv_{\eps,a}=-\scP_{\eps,a}\big(\bone_D^+(\bv_{\eps,a}^3+3\bv_{\eps,a}^2\bw_{\eps,a})\big)
-3\scP_{\eps,a}'\big(\bone_D^+\bv_{\eps,a}\bw_{\eps,a}^2\big)-\scP_{\eps,a}''(\bone_D^+\bw_{\eps,a}^3)+\bi_{\eps,a}.
\end{equ}
The operators $\scP_{\eps,a},\scP_{\eps,a}',\scP_{\eps,a}''$ are built from the models $(\Pi^{\eps,a},\Gamma^\eps)$ and the remainder $Z^{(3a)}$, the only difference in them will be the choice of $\hat\cR^\eps$.

We now specify the spaces of modelled distributions on which \eqref{eq:Phi-abstract} can be shown to be a contraction and therefore to have a unique solution.
Denote by $T_0^{\geq \alpha}$ the subspace of $T_0$ generated by basis vectors of homogeneity greater or equal to $\alpha$
(note that no new symbols are included in these spaces).
Fix the exponents $\gamma=1+3\kappa$, $\gamma_0=\kappa$,
$\alpha=0$, $\alpha_0=-1/2-\kappa$, $\alpha_1=-1-2\kappa$, $\alpha_2=-3/2-3\kappa$, the sectors
\begin{equ}
V=\cI(T_0^{\geq \alpha_0})+\bar T,\quad V_0=T_0^{\geq \alpha_0},\quad V_1=T_0^{\geq \alpha_1},\quad V_2=T_0^{\geq\alpha_2},
\end{equ}
as well as
\begin{equs}
\eta &=-2/3+\kappa&\quad \sigma &=1/2-4\kappa &\quad \mu &=-2/3+\kappa;\\
\eta_0 &=-2+3\kappa,&\quad \sigma_0 &=-1/2-\kappa,&\quad \mu_0 &=-2+\kappa;\\
\eta_1 &=-5/3-\kappa,&\quad \sigma_1 &=-1-2\kappa,&\quad \mu_1 &=-5/3-\kappa;\\
\eta_2 &=-3/2-3\kappa,&\quad \sigma_2 &=-3/2-3\kappa,&\quad\mu_2 &=-3/2-3\kappa.
\end{equs}
It is easy to verify that, by \cite[Lem~5.4]{GH17}, $\bi_{\eps,a}\in\cD^{\gamma,w}(V)$ depends continuously on $(\eps,a)\in[0,1]\times[0,\infty]$, provided $\kappa$ is sufficiently small.
It also follows from \cite[Lem.~4.3]{GH17} that the mappings
\begin{equ}
f\mapsto f^3+3f\bw_\eps,\quad\quad f\mapsto f\bw_\eps^2
\end{equ}
are locally Lipschitz continuous from $\cD^{\gamma,w}(V)$
to $\cD^{\gamma_0,w_0}(V_0)$ and
$\cD^{\gamma_0,w_1}(V_1)$, respectively,
and that $\bw_\eps^3\to\bw_0^3$ in $\scD^{\gamma_2,w_2}(V_2)$.
Furthermore, since $\sigma_0\wedge\alpha_0>-1$,
\cite[Thm~4.9]{GH17} provides a continuous extension $\hat \cR^\eps$ of the reconstruction operator on $\cD^{\gamma_0,w_0}(V_0)$,
which agrees with the canonical extension for $\eps>0$.
This now completes the definition of $\scP_\eps$ in \eqref{eq:Phi-abstract}.

Furthermore, in the case of Dirichlet boundary conditions, that is, $a=\infty$, we are in an even better shape. Indeed, $\scP_{\eps,\infty}'$ and $\scP_{\eps,\infty}''$ are also well-defined, thanks to Lemma \ref{lem:Dirichlet-convolution}.
By \cite[Thm~5.6]{GH17} one can conclude that,
provided $\kappa>0$ is sufficiently small, it has a unique local solution $\bv_{\eps,\infty}$. Now the proof can be easily concluded.
\begin{proof}[Proof of Theorem \ref{thm:Phi4-D}]
Invoking the usual arguments concerning the renormalisation in the bulk (see e.g. \cite[Sec~9]{H0}),
one sees that the function $v_{\eps,\infty}=\cR^\eps\bv_{\eps,\infty}$ solves
\begin{equ}
(\d_t-\Delta)v=-v^3-3v^2\Psi_{\eps,\infty}-3v(\Psi_{\eps,\infty}^2-C_\eps)
-(\Psi_{\eps,\infty}^3-3C_\eps\Psi_{\eps,\infty}),
\end{equ}
with $0$ Dirichlet boundary conditions.
Hence $v_{\eps,\infty} +\Psi_{\eps,\infty}$ solves solves \eqref{e:Phi43-eps},
trivially also with $0$ Dirichlet boundary conditions,
and therefore it coincides with $u_\eps^{\Dir}$.
By \cite[Thm.~5.6]{GH17}, $\bv_{\eps,\infty}\to\bv_{0,\infty}$, the latter
clearly does not depend on $\rho$,
and so the limit $\lim_{\eps\to0}u_\eps^{\Dir}=(\cR^0\bv_{0,\infty})+\Psi_{0,\infty}=:u^{\Dir}$ is also independent of $\rho$, yielding the claim.
\end{proof}

\begin{remark}\label{rem:Dirichlet}
To see in which sense $u^{\Dir}$ vanishes on the boundary, first note that $\cR^0\bv_{0,\infty}$ is a continuous function with $0$ trace.
It is also clear that by definition, $\Psi_{\eps,\infty}$ vanishes on the boundary for $\eps>0$.
By Corollary \ref{cor:restrict1}, $\Psi_{0,\infty}$ has a meaningful trace, which is also $0$.
\end{remark}

Now we move on to the setup of Theorem \ref{thm:Phi4-N}.
First consider the case $b\geq 0$.
Take $b_\eps$ as in the theorem and $c_\eps$ as constructed from $b_\eps$ in Lemma \ref{lem:reno Phi43 square}.
Consider the sequence $\bv_{\eps,c_\eps}$.
Since $c_\eps<\infty$, we are out of the scope of Lemma \ref{lem:Dirichlet-convolution}, so we have to construct $\hat \cR^\eps$ for the operators $\scP_{\eps,c_\eps}'$ and $\scP_{\eps,c_\eps}''$.

The former goes in essentially the same way as in the case of the PAM.
We can summarise the relevant properties of $\hat \cR^\eps$ in the following statement. Its proof follows from the same arguments as that of Lemma \ref{lem:hatR-PAM} 
(simply using Lemmas \ref{lem:reno Phi43 square} and \ref{lem:Phi model convergence} in place of Lemmas \ref{lem:reno Psi^2 PAM} and \ref{lem:PAM model convergence}, respectively)
and is therefore omitted.

\begin{lemma}\label{lem:hatR-Phi}
Let $a_\rho$, $b_\eps$, and $c_\eps$ as in Lemma \ref{lem:reno Phi43 square}.
Then for all $\eps\in[0,1]$ there exist maps $\hat \cR^\eps:\scD^{\gamma,w}(V)\bw_{\eps,c_\eps}^2\to\CC^{-1-2\kappa}$ with the following properties.
\begin{claim}
\item For all $\eps\in[0,1]$ and $\psi\in\cC^\infty_0(\R_+\times D)$ one has
\begin{equ}\label{eq:hatR-Phi1}
\scal{\hat\cR^\eps g,\psi}=\scal{\cR^\eps g ,\psi};
\end{equ} 
\item The bound
\begin{equ}\label{eq:hatR-Phi2}
|\hat\cR^\eps(f^\eps\bw_{\eps,c_\eps}^2)-\hat\cR^0(f^0\bw_{0,b}^2)|_{w_0;\,T}
\lesssim\vn{f^\eps;f^0}_{\gamma,w;\,T}+o(1)
\end{equ}
holds as $\eps\to 0$, uniformly in modelled distributions $f^\eps\in\cD^{\gamma,w}(V,\Gamma^\eps)$, $f^0\in\cD^{\gamma,w}(V,\Gamma^0)$ bounded by a constant $C$, and in $T\in(0,1]$;
\item For $\eps>0$, $f\in\cD^{\gamma,w}$, one has the identity
\begin{equ}\label{eq:hatR-Phi3}
\hat\cR^\eps(f \bw_{\eps,c_\eps}^2)=
\cR^\eps(f \bw_{\eps,c_\eps}^2)
-\big(a_\rho+b_\eps+c_\eps\big)\delta_\d (\cR^\eps f).
\end{equ}
\end{claim}
\end{lemma}

Moving on to $\scP_{\eps,c_\eps}''$, notice that
for each $\eps\in[0,1]$ there is only a single modelled distribution, $\bone_D^+\bw_{\eps,c_\eps}^3$, on which $\hat \cR^\eps$ needs to be constructed.
Also notice that one has for $\eps>0$
\begin{equ}\label{eq:cubic1}
\cR^\eps(\bone_D^+\bw_{\eps,c_\eps}^3)=\bone_D^+\big(\Psi_{\eps,c_\eps}^3-3\ell_\eps({\<Phi-2-Small>})\Psi_{\eps,c_\eps}\big).
\end{equ}
Let us then set, for $\eps\in[0,1]$,
\begin{equ}\label{eq:cubic2}
\hat\cR^\eps(\bone_D^+\bw_{\eps,c_\eps}^3)=\bone_D^+\Pi^{\eps,c_\eps}\<Phi-3new>+
3\bone^+\big((\delta_\d  R^1_\eps)\Psi_{\eps,c_\eps}+(\scR R^2_\eps)\Psi_{\eps,c_\eps}+R^3_\eps\Psi_{\eps,c_\eps}\big).
\end{equ}
First notice that by \eqref{eq:cubic1}, \eqref{eq:cubic2}, the definition of $\Pi^{\eps,c_\eps}\<Phi-3new>$, and Lemma \ref{lem:reno Phi43 square}, we have that for $\eps>0$
\begin{equ}\label{eq:cubic3}
\hat\cR^\eps(\bone_D^+\bw_{\eps,c_\eps}^3)=
\cR^\eps(\bone_D^+\bw_{\eps,c_\eps}^3)
-3\bone^+\big(a_\rho+b_\eps+c_\eps\big)\delta_\d (\Psi_{\eps,c_\eps}).
\end{equ}
In particular, $\hat\cR^\eps(\bone_D^+\bw_{\eps,c_\eps}^3)$ and $\cR^\eps(\bone_D^+\bw_{\eps,c_\eps}^3)$ coincide in the interior of $\R_+\times D$, as required.
It is also clear that $\hat\cR^\eps$ has the required continuity properties:
for the first term on the right-hand side of \eqref{eq:cubic2} this follows from the convergence of the models, while for the second this is precisely the content of Lemma \ref{lem:restrict2} (including that the expressions in \eqref{eq:cubic2} are actually meaningful when $\eps=0$).


\begin{proof}[Proof of Theorem \ref{thm:Phi4-N}]
It follows from the above that the abstract equation \eqref{eq:Phi-abstract} with $c_\eps$ in place of $a$ admits local solutions $\bv_{\eps,c_\eps}\in\cD^{\gamma,w}$.
Furthermore, by \eqref{eq:hatR-Phi3} and \eqref{eq:cubic3}, $v_{\eps,c_\eps}=\cR^\eps\bv_{\eps,c_\eps}$ satisfy \eqref{eq:Phi4-lin} with boundary conditions $\d_nv_{\eps,c_\eps}=-3c_\eps v_{\eps,c_\eps}$.
The argument leading to \eqref{eq:Phi4-lin} then shows that $\bar u_\eps:=v_{\eps,c_\eps}+\Psi_{\eps,c_\eps}$ satisfy \eqref{e:Phi43-eps} with boundary conditions $\d_n\bar u_\eps=3(a_\rho+b_\eps)\bar u_\eps$ and so it coincides with $u_\eps$.
By Proposition \ref{prop:tiny extension} we have that $\bv_{\eps,c_\eps}\to\bv_{0,b}$, the latter clearly does not depend on $\rho$, so the limit $\lim_{\eps\to0}u_\eps=(\cR^0\bv_{0,b})+\Psi_{0,b}=:u^b$ is also independent of $\rho$.
Finally, one notices that $u^{\infty}$ coincides with $u^{\Dir}$ as defined in the proof of Theorem \ref{thm:Phi4-D}.
\end{proof}

The case $b<0$ requires a very slight modification in order to reduce ourselves to the case $a = 0$ in
\eqref{eq:Psi}.  Defining $\tilde b_\eps=b_\eps-b$, the sequence $(\tilde b_\eps)_{\eps\in(0,1]}$ satisfies \eqref{eq:boundary-reno} with $\tilde b=0$.
Replacing the maps $\hat\cR^\eps$ from Lemma \ref{lem:hatR-Phi} and \eqref{eq:cubic2} by $\tilde\cR^\eps$ defined as
\begin{equs}
\tilde \cR^\eps (f\bw_{\eps,\tilde c_\eps}^2)&:=\hat \cR^\eps (f\bw_{\eps,\tilde c_\eps}^2)+b\delta_\d(\cR^\eps f),
\\
\tilde\cR^\eps(\bone_D^+\bw_{\eps,\tilde c_\eps}^3)&:=\hat\cR^\eps(\bone_D^+\bw_{\eps,\tilde c_\eps}^3)+3\bone^+ b\delta_\d\Psi_{\eps,\tilde c_\eps},
\end{equs}
the proof is concluded just as above.

\appendix
\section{Abstract Schauder estimate -- the Dirichlet case}
We wish to obtain a Schauder estimate in the special case of the Dirichlet heat kernel, but \emph{without} dealing with spatial extensions of $\cR f$, even in the regime $\alpha\wedge\sigma\in (-2,-1)$.
We remark that a similar issue was encountered in \cite[Sec~3]{Cyril19}, but therein the problem could be handled with a more ad hoc workaround.
Let us first formulate what we mean by a Dirichlet kernel.
Simplify notations by setting, in the context of \cite{GH17}, $P_0=\{0\}\times\R^{d-1}$, $P_1=\R\times \partial D$, and
\begin{equ}
\beta=2,\, \frm_0=2, \,\frm_1=1.
\end{equ}
Denote furthermore $Q=[-1,1]\times D$.
\begin{assumption}\label{as:Dirichlet}
Let $K$ (with a corresponding abstract integration map $\cI$) and $Z$ be as in \cite{GH17}.
Suppose furthermore that for all $n\in\N$, $G_n:=K_n+Z_n$ is non-anticipative, symmetric, and vanishes on $(\R^d\setminus P_1)\times P_1$.
\end{assumption}
Note that this implies the following: for any $y\in Q$, $D_2^kG_n(\bar x,y)=D_1^kG_n(y,\bar x)$, and the latter vanishes whenever $\bar x$ belongs to the boundary, so
\begin{equ}\label{eq:Dirichlet impr}
\sup_y |D^k_2G_n(x,y)|\lesssim |x|_{P_1}\sup_{y,z}|D^k_1 \nabla_2 G_n(y,z)|\lesssim |x|_{P_1}2^{n(|k|_\frs+|\frs|-1)}.
\end{equ}
So, in the case $|x|_{P_1}\les 2^{-n}$, we improve the usual bound $|D_2^kG_n|\les 2^{n(|k|_\frs+|\frs|-2)}$ to $|D_2^kG_n|\les |x|_{P_1}2^{n(|k|_\frs+|\frs|-1)}$.

Next, take a partition of unity (essentially the same as in the proof of \cite[Prop.~6.9]{H0}) with the following properties:
\begin{itemize}
\item One has the bounds $|D^k\varphi_{m,i}|\lesssim 2^{m|k|_{\frs}}$, $|\supp\varphi_{m,i}|\lesssim 2^{-m|\frs|}$, as well as the property $\supp\varphi_{m,i}\subset\{y:\,|y|_{P_1}\sim 2^{-m}\}$.
\item 
$\sum_{m\in \N}\sum_{i=1}^{\ell(m)}\varphi_{m,i}(y)=1$ for all $y\in Q$, and there exists a constant $C$ such that for all $m,i$, the support of $\varphi_{m,i}$ intersects with $\varphi_{m,j}$ for at most $C$ indices $j$ (in particular, $\ell(m)\les 2^{m(|\frs|-1)}$).
\end{itemize}
Let us set
\begin{equ}
G_{n}^{m,i}(x,y)=G_n(x,y)\varphi^{m,i}(y)
\end{equ}
We shall only use these functions when $2^{-n}\geq|x|_{P_1}$, in which case $G_n^{m,i}$ is nonzero only if $n\leq m$.
Also, for each $n\leq m$ there is a set $A_{n,m}\subset\{1,\ldots,\ell(m)\}$ with cardinality of order $2^{(m-n)(|\frs|-1)}$ such that $G_n^{m,i}=0$ for $i\notin A_{n,m}$.
We can now define our candidate for acting on $G(x,\cdot)$ by distributions $\zeta$ on $Q$:
\begin{equ}\label{eq:G in D}
\zeta(G(x,\cdot)):=\sum_{2^{-n}< |x|_{P_1}}\zeta\big(G_n^{}(x,\cdot)\big)
+
\sum_{2^{-n}\geq |x|_{P_1}}\sum_{m\in\N}\sum_{i\in A_{n,m}}\zeta\big(G_n^{m,i}(x,\cdot)\big),
\end{equ}
which is meaningful provided the infinite sums (the first one in $n$, the second one in $m$) converge. Note also that if $\zeta$ happens to be a piecewise smooth function up to $P_1$, vanishing for negative times, then the following identity holds on $Q_+=[0,1]\times D$:
\begin{equ}
\zeta(G(x,\cdot))=\int_Q G(x,y)\zeta(y)\,dy.
\end{equ}
The improvement in the Dirichlet case will follow from the following simple property. 
It is trivial from \eqref{eq:Dirichlet impr} that $G_n^{m,i}(x,\cdot)$ is $|x|_{P_1}2^{n(|\frs|-1)}2^{-m|\frs|}$ times an approximate Dirac-delta, that is, a function of the form $\psi_y^{2^{-m}}$, with some $y\in Q$ and $\psi\in\cB$ as usual.
However, thanks to the vanishing at the boundary, this can be improved as follows.
\begin{proposition}\label{prop:G vanish}
Suppose $2^{-n}\geq \big(2^{-m}\vee|x|_{P_1}\big)$. Define
\begin{equ}
\hat G_n^{m,i}(x,\cdot)=2^{m(|\frs|+1)}|x|_{P_1}^{-1}2^{-n|\frs|}G_n^{m,i}(x,\cdot).
\end{equ}
Then for any $x\in[-1,1]\times D$, the function $\hat G_n^{m,i}(x,\cdot)$ is of the form $\psi_y^{2^{-m}}$, with $\psi\in\cB$ as usual, and $|y|_{P_1}\geq 2^{-m+1}$.
\end{proposition}
\begin{proof}
It suffices to show that (i) $\supp \hat G_n^{m,i}(x,\cdot)\subset\{y:\,|y|_{P_1}\sim 2^{-m}\}$, (ii) $|\supp\hat G_n^{m,i}(x,\cdot)|\lesssim 2^{-m|\frs|}$, (iii) $|D^k\hat G_n^{m,i}(x,\cdot)|\lesssim 2^{m(|\frs|+|k|_\frs)}$.

Property (i) and (ii) follows easily from the properties of $\varphi_{m,i}$.
For property (iii), first write
\begin{equ}
|D^k\hat G_n^{m,i}(x,y)|\lesssim 2^{m(|\frs|+1)}|x|_{P_1}^{-1}2^{-n|\frs|}\sum_{k_1+k_2=k}\sup_{|z|_{P_1}\sim 2^{-m}}|D^{k_1} G_n(x,z)|\sup_z|D^{k_2}\varphi_{m,i}(z)|
\end{equ}
Since by assumption, is $G_n=0$ on $P_1$, one can write, using \eqref{eq:Dirichlet impr}
\begin{equ}
\sup_{|z|_{P_1}\sim 2^{-m}}|D^{k_1} G_n(x,z)|\lesssim 2^{-m}\sup_{|z|_{P_1}\lesssim 2^{-m}}|\nabla D^{k_1} G_n(x,z)|\les 2^{-m}|x|_{P_1}2^{n(|\frs|+|k_1|_\frs)}
\end{equ}
where $\nabla$ above denotes the spatial gradient. Along with the trivial bound $|D^{k_2}\varphi_{m,i}|\les 2^{m|k_2|_\frs}$, this yields (iii) as claimed.
\end{proof}

To illustrate how this bound can be used, consider $\zeta\in\cC^{\alpha}(Q)$ with $\alpha\in(-2,-1)$. Extending it by $0$ to $\R\times D$ is of course not a problem, since the time boundaries have codimension $2$.
One can write by \eqref{eq:G in D} and Proposition \ref{prop:G vanish}
\begin{equs}[eq:calculation0]
|\zeta(G(x,\cdot)|&\leq \sum_{2^{-n}< |x|_{P_1}}2^{-2n}|\zeta\big(2^{2n}G_n(x,\cdot)\big)|
\\
&\qquad+\sum_{2^{-n}\geq |x|_{P_1}}\sum_{m\in\N}\sum_{i\in A_{n,m}}2^{-m(|\frs|+1)}|x|_{P_1}2^{n|\frs|}|\zeta\big(\hat G_n^{m,i}(x,\cdot)\big)|
\\
&\les \sum_{2^{-n}< |x|_{P_1}}2^{-2n} 2^{-n\alpha}
+\sum_{2^{-n}\geq |x|_{P_1}}|x|_{P_1}2^{n|\frs|}\sum_{m\gtrsim n} 2^{(m-n)(|\frs|-1)}2^{-m(|\frs|+1)}2^{-m\alpha},
\end{equs}
where to get the second inequality, we used that in the first line $\zeta$ is always applied to approximate Dirac-deltas lying completely in $\R\times D$.
Note that the infinite sums converge: the exponent of $2^n$ in the first and of $2^m$ in the second, is $-2-\alpha<0$.
Hence,
\begin{equ}\label{eq:calculation1}
|\zeta(G(x,\cdot)|\les |x|_{P_1}^{2+\alpha}+\sum_{2^{-n}\geq |x|_{P_1}}2^{-n(2+\alpha)}2^{n}|x|_{P_1}\les |x|_{P_1}^{2+\alpha},
\end{equ}
where in the last inequality we used $1+\alpha<0$. So indeed $\zeta(G(x,\cdot)$ is a well-defined function, and in fact we even see the `improvement by order 2' of the weight at the boundary, since $\zeta(G(x,\cdot)$ vanishes on the boundary with speed $2+\alpha$.

Fixing a $\gamma>0$, we then define
\begin{equ}
\scP f(x)=\cI f(x)+\cJ(x)f(x)+Z\Pi_x f(x)+\bar\cN f(x),
\end{equ}
where we set
\begin{equ}
\bar\cN f(x)=\sum_{|k|_\frs<\gamma+2}\frac{X^k}{k!}\big(\cR f-\Pi_x f(x)\big)\big(D_x^kG(x,\cdot)\big),
\end{equ}
and the last term is understood in the sense of \eqref{eq:G in D}.
We then have the following extension of \cite[Lem~4.12]{GH17}.

\begin{lemma}\label{lem:Dirichlet-convolution}
Suppose that Assumption \ref{as:Dirichlet} holds.
Fix $\gamma > 0$, $w = (\eta,\sigma,\mu)$, and let $V$ be a sector of regularity $\alpha$.
Assume that these exponents satisfy
\begin{equ}
a_\wedge=\eta\wedge\sigma\wedge\mu\wedge\alpha\in(-2,-1),
\end{equ}
and take
\begin{equation}\label{eq:exponents2}
\bar \gamma = \gamma + 2,\quad
\bar\eta=(\eta\wedge\alpha)+2, \quad
\bar\sigma=(\sigma\wedge\alpha)+2,\quad
\bar\mu\leq(a_\wedge+2)\wedge 0.
\end{equation}
Suppose furthermore that none of  $\bar \gamma$, $\bar\eta$, $\bar\sigma$, or $\bar\mu$ are integers.
Then, for $f\in\cD_P^{\gamma,w}(V)$, one has
$\scP f\in\cD_P^{\bar \gamma,\bar w}$, where $\bar w=(\bar\eta,\bar\sigma,\bar\mu)$.

Furthermore, if $(\bar\Pi,\bar\Gamma)$ is a second model realising $K$ for $\cI$ and $\bar f\in\cD_P^{\gamma,w}(V,\bar\Gamma)$, then for any $C>0$ the bound
$$
\vn{\scP f;\bar\scP\bar f}_{\bar\gamma,\bar w;\frK}\lesssim\vn{f;\bar f}_{\gamma,w;\bar \frK}+\|\Pi-\bar\Pi\|_{\gamma;\bar\frK}+\|\Gamma-\bar\Gamma\|_{\bar\gamma;\bar\frK}
$$
holds uniformly in models and modelled distributions both satisfying $\vn{f}_{\gamma,w;\bar \frK}+\|\Pi\|_{\gamma;\bar\frK}+\|\Gamma\|_{\bar\gamma;\bar\frK}+\|\zeta\|_{a,\bar \frK}\leq C$, where $\bar\frK$ denotes the $1$-fattening of $\frK$.

Finally, the identity
\begin{equation}\label{eq:reco identity}
\cR\scP f(x)= \cR f(G(x,\cdot))
\end{equation}
holds on $Q_+$, for $f$ vanishing for negative times, where the right-hand side is understood in the sense of \eqref{eq:G in D}.
\end{lemma}

\begin{remark}\label{rem:tilde Rf regularity}
Recall the regularity of $\cR f$, for $f\in\cD_P^{\gamma,w}(V)$ as above: It is a distribution on $Q_+:=[0,1]\times D$ that satisfies the bounds
\begin{equ}\label{eq:cR bound 1}
\cR f(\psi_x^\lambda)\les \lambda^{\eta\wedge\alpha}|x|_{P_1}^{\mu-\eta\wedge\alpha}
\end{equ}
on $Q_+\cap\{|x|_{P_0}\leq 2|x|_{P_1}\}$ for $4\lambda\leq|x|_{P_0}$, and
\begin{equ}\label{eq:cR bound 2}
\cR f(\psi_x^\lambda)\les \lambda^{\sigma\wedge\alpha}|x|_{P_0}^{\mu-\sigma\wedge\alpha}
\end{equ}
on $Q_+\cap\{|x|_{P_1}\leq|x|_{P_0}\}$ for $4\lambda\leq|x|_{P_1}$. It then follows (see e.g. \cite[Prop~2.15]{GH17}) that if $\eta\wedge\alpha>-2$, then $\cR f$ has a unique extension as a distribution on $Q$ such that it agrees with $\cR f$ on $Q_+$, it vanishes on $[-1,0]\times D$, and it satisfies \eqref{eq:cR bound 1} on $Q\cap\{|x|_{P_0}\leq 2|x|_{P_1}\}$ for $2\lambda\leq|x|_{P_1}$.
In the sequel we denote this extension also by $\cR f$.
In particular, $\cR f\in\cC^{\eta\wedge\sigma\wedge\mu\wedge\alpha}(Q)$.
\end{remark}

\begin{proof}
Inspecting the proof of \cite[Lem~4.12]{GH17}, there are two instances where an extension of $\cR f$ is required.
On one hand, when establishing the bound (including making sense of the left-hand side)
\begin{equ}\label{eq:int bound 1}
\sum_{2^{-n}\geq|x|_{P_0}}|\cR f(D^l_1G_n(x,\cdot))|\les |x|_{P_1}^{\bar\mu-|l|_\frs}\left(\frac{|x|_{P_0}}{|x|_{P_1}}\right)^{(\bar\eta-|l|_\frs)\wedge0}
\end{equ}
on the set $|x|_{P_0}\leq|x|_{P_1}$, and the symmetric analogue on the set $|x|_{P_1}\leq|x|_{P_0}$.
On the other hand, the extension of $\tilde \cR f$ is used in the estimate, for $|k+l|_s\geq\bar\gamma$,
\begin{equ}\label{eq:int bound 2}
\sum_{2^{-n}\geq|x|_{P_0}}\|x-y\|^{|k+l|_\frs-\bar\gamma}|\cR f(D_1^{k+l}G_n(\bar y,\cdot))|\les|x|_{P_0}^{\bar\eta-\bar\gamma}|x|_{P_1}^{\bar\mu-\bar\eta}
\end{equ}
whenever $x,y,\bar y$ satisfy $2\|x-\bar y\|\leq 2\|x-y\|\leq|x|_{P_0}\leq|x|_{P_1}$, as well as the symmetric analogue in the case $|x|_{P_1}\leq|x|_{P_0}$. 
Our goal is therefore to make sense of these expressions and derive the required bounds without extending $ \cR f$ spatially, making use of Assumption \ref{as:Dirichlet}.
We note that, unlike in \cite{GH17}, the two cases $|x|_{P_0}\lessgtr|x|_{P_1}$ have to be treated slightly differently, since the special property of vanishing on the boundary is only assumed on $P_1$.

We start with \eqref{eq:int bound 1}, when $|x|_{P_0}\leq|x|_{P_1}$, first considering the terms
$|x|_{P_0}\leq 2^{-n}\leq|x|_{P_1}$. By \eqref{eq:cR bound 1}, we get
\begin{equ}
|\cR f(D^l_1G_n(x,\cdot))|
\leq 
2^{-n(\eta\wedge\alpha+2-|l|_\frs)}|x|_{P_1}^{\mu-\eta\wedge\alpha}
\leq 
2^{-n(\bar\eta-|l|_\frs)}|x|_{P_1}^{\bar\mu-\bar\eta}.
\end{equ}
If $\bar\eta-|l|_\frs>0$, then summing over $|x|_{P_0}\leq 2^{-n}\leq|x|_{P_1}$ gives the bound
\begin{equ}
|x|_{P_1}^{\bar\mu-|l|_\frs},
\end{equ}
which is of the required order. For $\bar\eta-|l|_\frs<0$ (equality cannot occur by assumption), the summation yields
\begin{equ}\label{eq:00}
|x|_{P_0}^{\bar\eta-|l|_\frs}|x|_{P_1}^{\bar\mu-\bar\eta}\les
|x|_{P_1}^{\bar\mu-|l|_\frs}\left(\frac{|x|_{P_0}}{|x|_{P_1}}\right)^{\bar\eta-|l|_\frs}.
\end{equ}
On the scale $|x|_{P_1}\leq 2^{-n}$, we repeat the calculation in \eqref{eq:calculation0}-\eqref{eq:calculation1}, making use of $\cR f\in\cC^{a_\wedge}(Q)$ to obtain
\begin{equs}\label{eq:01}
\sum_{2^{-n}\geq |x|_{P_1}}&\sum_{m\in\N}\sum_{i\in A_{n,m}}|\cR f(D_1^lG_n^{m,i}(x,\cdot))|
\\
&\les \sum_{2^{-n}\geq |x|_{P_1}}|x|_{P_1}2^{n|\frs|}\sum_{m\gtrsim n}2^{(m-n)(|\frs|-1)}2^{-m(|\frs|+1-|l|_\frs)}2^{-m a_\wedge}\les|x|_{P_1}^{2+a_\wedge-|l|_\frs},
\end{equs}
where in the summation with respect to $n$, we used that $1+a_\wedge-|l|_\frs$ is negative, so the exponent of $2^n$ is positive.
Now suppose $|x|_{P_1}\leq|x|_{P_0}$ and consider the scale $|x|_{P_1}\leq 2^{-n}\leq|x|_{P_0}$.
Using again the argument in \eqref{eq:calculation0}-\eqref{eq:calculation1}, this time making use of the regularity of $\cR f$ given in \eqref{eq:cR bound 2}, we get
\begin{equs}
\sum_{|x|_{P_0}\geq 2^{-n}\geq |x|_{P_1}}&\sum_{m\in\N}\sum_{i\in A_{n,m}}|\cR f(D_1^lG_n^{m,i}(x,\cdot))|
\\
&\les \sum_{|x|_{P_0}\geq 2^{-n}\geq |x|_{P_1}}|x|_{P_1}2^n\sum_{m\gtrsim n}2^{-m(2-|l|_\frs)}2^{-m (\sigma\wedge\alpha)}|x|_{P_0}^{\mu-\sigma\wedge\alpha}
\\
&\les|x|_{P_0}^{\bar\mu-\bar\sigma}|x|_{P_1}\sum_{|x|_{P_0}\geq 2^{-n}\geq |x|_{P_1}}
2^{n(|l|_\frs+1-\bar\sigma)}.
\end{equs}
If $|l|_\frs+1-\bar\sigma>0$, then this gives a bound of order $|x|_{P_0}^{\bar\mu-\bar\sigma}|x|_{P_1}^{\bar\sigma-|l|_\frs}$, which can be seen to be of the right order just as in \eqref{eq:00}.
If $|l|_\frs+1-\bar\sigma<0$, then the above bound is of order $|x|_{P_0}^{\bar\mu-\bar\sigma-1}|x|_{P_1}\leq|x|_{P_0}^{\bar\mu-\bar\sigma}$, again as required. On the scales $2^{-n}\geq |x|_{P_0}$, we can write similarly to \eqref{eq:01},
\begin{equ}
\sum_{2^{-n}\geq |x|_{P_1}}\sum_{m\in\N}\sum_{i\in A_{n,m}}|\cR f(D_1^lG_n^{m,i}(x,\cdot))|
\les|x|_{P_1}|x|_{P_0}^{1+a_\wedge-|l|_\frs}\leq |x|_{P_0}^{2+a_\wedge-|l|_\frs}.
\end{equ}

To prove \eqref{eq:int bound 2}, first recall that in the situation therein $\|x-y\|\les|x|_{P_i}\sim|y|_{P_i}\sim|\bar y|_{P_i}.$ Therefore by \eqref{eq:int bound 1} one gets
\begin{equs}
\sum_{2^{-n}\geq|x|_{P_0}}&\|x-y\|^{|k+l|_\frs-\bar\gamma}|\cR f(D_1^{k+l}G_n(\bar y,\cdot))|
\\
&\les \|x-y\|^{|k+l|_\frs-\bar\gamma}|x|_{P_1}^{\bar\mu-\bar\eta}|x|_{P_0}^{\bar\eta-|k+l|_\frs}
\les |x|_{P_0}^{\bar\eta-\bar\gamma}|x|_{P_1}^{\bar\mu-\bar\eta},
\end{equs}
where the condition $|k+l|_\frs>\bar\gamma\geq\bar\eta$ was repeatedly used. The analogous estimate in the case $|x|_{P_1}\leq|x|_{P_0}$ in this case follows exactly the same way.
\end{proof}

\endappendix

\bibliography{PAM_Neumann}{}
\bibliographystyle{Martin}

\end{document}